\documentclass[10pt,a4paper]{amsart}
\usepackage[utf8]{inputenc}
\usepackage[english]{babel}
\usepackage{amsmath}
\usepackage{amsfonts}
\usepackage{amssymb}
\usepackage{bbm}
\usepackage{mathrsfs}
\usepackage[all]{xy}
\usepackage{bm}
\usepackage[left=2cm,right=2cm,top=2cm,bottom=2cm]{geometry}
\usepackage{tikz}
\author{Geoffrey Powell}
\title{Baby bead representations}
\date{}
\thanks{This work was partially supported by the ANR Project {\em ChroK}, {\tt ANR-16-CE40-0003}.}
\subjclass[2000]{}
\newtheorem{THM}{Theorem}
\newtheorem{PROP}[THM]{Proposition}
\newtheorem{COR}[THM]{Corollary}
\newtheorem{thm}{Theorem}[section]
\newtheorem{prop}[thm]{Proposition}
\newtheorem{cor}[thm]{Corollary}
\newtheorem{lem}[thm]{Lemma}
\theoremstyle{definition}
\newtheorem{defn}[thm]{Definition}
\newtheorem{exam}[thm]{Example}
\theoremstyle{remark}
\newtheorem{rem}[thm]{Remark}
\newtheorem{nota}[thm]{Notation}

\newtheorem{conve}[thm]{Convention}
\newcommand{\fg}{\mathsf{F}}
\newcommand{\beads}{\mathsf{Beads}}
\newcommand{\foutan}{\f^{\mathrm{Out}}_\omega (\gr\op)}

\newcommand{\ind}[1][]{I_{\cc_{#1}}^{\cd_{#1}}}
\newcommand{\f}{\mathcal{F}}
\newcommand{\commag}{\mathfrak{ComMag}}
\newcommand{\com}{\mathfrak{Com}}
\newcommand{\cat}{\mathbf{Cat}\hspace{1pt}}
\newcommand{\flie}{\f_{\lie}}
\newcommand{\tlie}{\lie_{\leq 2}}
\newcommand{\ftlie}{\f_{\tlie}}
\newcommand{\catlie}{\cat\lie}
\newcommand{\cattlie}{\cat\tlie}
\newcommand{\nt}{\mathrm{Nat}_{V \in \fvs}}
\newcommand{\ylin}{\mathscr{E}}
\newcommand{\triv}{\mathrm{triv}}
\newcommand{\sgn}{\mathrm{sgn}}
\newcommand{\radj}{\rho_{(\mathbf{0},\mathbf{1})}}
\newcommand{\hcc}{H^\cc_{\cc_0}}
\newcommand{\hcd}{H^\cd_{\cd_0}}
\newcommand{\hce}{H^{\ce}_{\ce_0}}
\newcommand{\filt}{\mathfrak{f}}
\newcommand{\invcc}{\mathfrak{t}}
\newcommand{\contract}{\mathsf{Ctrct}}
\newcommand{\sets}{\mathsf{Set}}
\newcommand{\gph}{\mathscr{G}}
\newcommand{\gbar}{\overline{\gph}}
\newcommand{\ghat}{\widehat{\gph}}
\newcommand{\gtilde}{\widetilde{\gph}}
\newcommand{\mc}{\mathscr{M}}
\newcommand{\cc}{\mathfrak{C}}
\newcommand{\cd}{\mathfrak{D}}
\newcommand{\ce}[1][\kring]{\mathfrak{A}^{#1}}
\newcommand{\id}{\mathrm{Id}}
\newcommand{\init}{\mathrm{Init}}
\newcommand{\dbl}{\mathsf{dbl}}
\newcommand{\dblelt}{\widetilde{\mathsf{dbl}}}
\newcommand{\apsh}[1]{\mathcal{F}(#1)}
\newcommand{\lad}{{\mathsf{q}^{\pm}}}
\newcommand{\ldm}{(\lad I)_{\mc}}

\newcommand{\bba}{\mathsf{Beads}_\kring^{[1,2],\pm}}
\newcommand{\bbd}{\mathsf{Beads}^{[1,2]}}
\newcommand{\conv}{\odot}
\newcommand{\ori}{\mathrm{Or}}
\newcommand{\hfi}{\hom_\finj}
\newcommand{\aut}{\mathrm{Aut}}
\newcommand{\fsets}{\sets^{\mathrm{f}}}
\newcommand{\vs}{\mathcal{V}_\kring}
\newcommand{\fvs}{\vs^{\mathrm{f}}}
\newcommand{\g}{\mathsf{g}}
\newcommand{\trans}{^{\mathrm{tr}}}
\renewcommand{\phi}{\varphi}
\renewcommand{\hom}{\mathrm{Hom}}
\newcommand{\sym}{\mathfrak{S}}
\newcommand{\gr}{\mathbf{gr}}

\newcommand{\calc}{\mathcal{C}}
\newcommand{\nat}{\mathbb{N}}
\newcommand{\zed}{\mathbb{Z}}
\newcommand{\ext}{\mathrm{Ext}}
\newcommand{\op}{^\mathrm{op}}
\newcommand{\ob}{\mathrm{Ob}\hspace{2pt}}
\newcommand{\kring}{\mathbbm{k}}
\newcommand{\dash}{\hspace{-2pt}-\hspace{-2pt}}
\newcommand{\modules}{\mathrm{mod}}
\newcommand{\fb}{{\bm{\Sigma}}}

\newcommand{\lie}{\mathfrak{Lie}}
\newcommand{\uass}{\mathfrak{Ass}^u}
\newcommand{\finj}{\mathbf{FI}}
\numberwithin{equation}{section}
\begin{document}

\footnotetext{https://orcid.org/0000-0003-2564-1202}

\begin{abstract}
This paper is motivated by the study of Turchin and Willwacher's bead representations. The problem is reformulated here in terms of the Lie algebra homology of a free Lie algebra with coefficients in tensor products of the adjoint representation. 

The main idea is to exploit the truncation of the coefficients given by killing Lie brackets of length greater than two. Although this truncation is brutal, it retains significant and highly non-trivial information, as exhibited by explicit results.

A {\em dévissage} is used that splits the problem into two steps, separating out a `homology' calculation from `antisymmetrization'.  This involves some auxiliary categories, including a generalization of the upper walled Brauer category. 

This approach passes through the `baby bead representations' of the title, for which 
 complete results are obtained. As an application, the composition factors of Turchin and Willwacher's bead representations are calculated for a new infinite family. 
\end{abstract}

\maketitle

\section{Introduction}

This work is motivated by the study of the {\em bead representations} of the outer automorphism groups $\mathrm{Out}(F_r)$ of free groups (for $r \in \nat$) that occur in the work of Turchin and Willwacher \cite{MR3653316,MR3982870}. These are of significant interest through their relationship with hairy graph homology, the original motivation of Turchin and Willwacher, but also in more recent geometric work, appearing for instance in the study of the homology of configuration spaces and in the study of the cohomology of moduli spaces of algebraic curves (see  \cite{2021arXiv210903302B} and \cite{2022arXiv220212494G}  for example). This reflects deep interactions with Hochschild-Pirashvili homology.

The focus here is algebraic; in this introduction  $\kring$ is taken to be a field of characteristic zero. The specific question addressed is that of understanding the functor on finite-dimensional $\kring$-vector spaces
 \[
 V 
 \mapsto 
 H_0 (\lie(V); \lie(V)^{\otimes n}) 
 \]
 for $n \in \nat$, taking values in $\kring [\sym_n]$-modules, where the free Lie algebra $\lie(V)$ acts diagonally via the adjoint action and the $\sym_n$-action is via place permutations. The relationship between this and the work of Turchin and Willwacher is outlined in  Sections \ref{sect:catlie} and \ref{sect:tw}.
 
For $n \leq 1$, this is easy;  for larger $n$ it is a formidable problem. To circumvent this, we  truncate  drastically via  the surjection of Lie algebras 
$$\lie (V) \twoheadrightarrow \tlie(V) \cong \Lambda^1 (V) \oplus \Lambda^2 (V)$$
 killing all brackets of length $\geq 3$. Thus $\tlie (V)$ is a $\lie(V)$-module  and one has the surjection of $\lie (V)$-modules $\lie(V)^{\otimes n} \twoheadrightarrow \tlie(V)^{\otimes n}$. 

This induces the natural surjection
\begin{eqnarray}
\label{eqn:surj_H0}
H_0 (\lie(V); \lie (V)^{\otimes n}) 
\twoheadrightarrow 
 H_0 (\lie(V); \tlie (V)^{\otimes n}). 
\end{eqnarray}
The codomain of (\ref{eqn:surj_H0}) retains important and highly non-trivial information. 

This is most clearly exhibited by passing to isotypical components with respect to the functoriality in $V$, as follows. For a partition $\rho$, the associated simple representation of the symmetric group $\sym_{|\rho|}$  is denoted $S^\rho$. This has associated Schur functor $V \mapsto S^\rho (V)$. The isotypical component indexed by $\rho$ is given by the space of natural transformations $\nt (S^\rho (V) , -)$ of functors with respect to the finite-dimensional vector space $V \in \ob \fvs$. 

For instance, one has: 

\begin{THM}
\label{THMA}
The natural transformation induced by (\ref{eqn:surj_H0})
\[
\nt(S^\rho (V), H_0 (\lie(V); \lie (V)^{\otimes n})) 
\twoheadrightarrow 
\nt(S^\rho (V), H_0 (\lie(V); \tlie (V)^{\otimes n}))
\]
is an isomorphism for all $n \in \nat$ in the following cases:
\begin{enumerate}
\item 
$\rho = (1^N) $, $N \in \nat$; 
\item 
$\rho = (2,1^{N-2})$, $3 \leq N \in \nat$.
\end{enumerate}
\end{THM}

The case $\rho = (1^N) $ (see Proposition \ref{prop:1^N}) holds for fairly elementary reasons; for example,
\[
\nt(S^{(1^N)} (V),  \lie (V)^{\otimes n})
\twoheadrightarrow 
\nt(S^{(1^N)} (V),  \tlie (V)^{\otimes n})
\]
is an isomorphism. This is a consequence of the fact that the $d$th homogeneous component $\lie_d (V)$ of $\lie(V)$ does not have a composition factor $ S^{(1^d)}(V)$ for $d>2$ as a functor of $V  \in \ob  \fvs$. 

 The case $\rho = (2,1^{N-2})$ (see Theorem \ref{thm:2_1_N-2_isotypical}) is much less immediate and serves to show the subtlety of the truncation process in combination with the respective quotients  $\lie(V)^{\otimes n} \twoheadrightarrow H_0 (\lie (V); \lie(V)^{\otimes n})$ and $\tlie(V)^{\otimes n} \twoheadrightarrow H_0 (\lie (V); \tlie(V)^{\otimes n})$. For example, it is not true  for $n \geq 1$ that 
\[
\nt(S^{(2,1^{N-2})} (V),  \lie (V)^{\otimes n})
\twoheadrightarrow 
\nt(S^{(2,1^{N-2})} (V),  \tlie (V)^{\otimes n})
\]
is an isomorphism. This boils down to the fact that $\lie_3(V)$ is isomorphic to $S^{(2,1)}(V)$. 

There is also an analogue of Theorem \ref{THMA} for the partitions $(N)$ and $(N-1,1)$; however, as explained in Section \ref{sect:detrunc}, these cases are of less interest.  In general, the analogue of Theorem \ref{THMA}  is false; this is illustrated in Section \ref{sect:detrunc} by the case $\rho = (2,2)$.

Theorem \ref{THMA} is especially significant since the codomain is calculable in each case.  For the partitions $\rho = (1^N)$, this is given by  Theorem \ref{thm:isotyp_1N}, which recovers very directly families of representations exhibited by Turchin and Willwacher:

\begin{THM}
\label{THMBbefore}
For $0< N \in \nat$, the  $\sym_n$-module 
$\nt(S^{(1^{N})} (V), H_0 (\lie(V); \lie (V)^{\otimes n}))$ is isomorphic to 
\[
\begin{array}{ll}
S^{(m+1, 1^{N-2m-1})}  & 0 \leq m, \  2m< N \\
0 & \mbox{otherwise,}
\end{array}
\]
where $m = N-n $.
\end{THM}

The second case, as stated in Theorem \ref{THMB}, corresponds to  Theorem \ref{thm:isotypical_lad_int_2,1s} via Theorem \ref{THMA}. In the following statement, Schur functors that are indexed by sequences that are not partitions are understood to be zero, by convention.

\begin{THM}
\label{THMB}
For $3 \leq N \in \nat$, $n \in \nat$ and setting $m:= N-n$,  $\nt(S^{(2,1^{N-2})} (V), H_0 (\lie(V); \lie (V)^{\otimes n}))$ identifies as a $\sym_n$-module in the respective cases as:
 \[
 \begin{array}{ll}
 m=0 & 
 S^{(2,1^{N-2})}
 \\
 m=1, N=3 & 0 
 \\
  m=1, N>3 &
 S^{(2,1^{N-3})} \oplus S^{(3,1^{N-4})} \oplus S^{(2,2,1^{N-5})} 
  \\
m>1, N-2m \geq 2 &
S^{(m+1,1^{N-2m-1}) }
\oplus 
S^{(m+2, 1^{N-2m -2})} 
\oplus 
S^{(m,2, 1^{N-2m-2})}
\oplus 
S^{(m+1, 2, 1^{N-2m-3})} 
\\
 \mbox{otherwise} & 0. 
\end{array} 
 \]
\end{THM}

This gives a new, infinite family of non-trivial examples. This is of interest both in itself and as an illustration of the application of the methods that are developed in this paper. That such a result should be true was suggested by the author's {\em ad hoc} calculations using Sage \cite{sagemath} in December 2021. (See also the calculations and conjectures proposed by Gadish and Hainaut in \cite{2022arXiv220212494G}.)

Theorems \ref{THMBbefore} and \ref{THMB} are proved by applying the following general statement (cf. Theorem \ref{thm:isotyp_hce_bba}), in which  $\widehat{\lambda}$ denotes the partition obtained from  $\lambda$ by removing $\lambda_1$ and $S^{\rho/\widehat{\lambda}}$ and $S^{\rho / \widehat{\lambda'}}$ are skew representations:

\begin{THM}
\label{THMBafter}
For $ n \in \nat$ and a partition $\rho$, the $\sym_n$-module $\nt(S^{\rho} (V), H_0 (\lie(V); \tlie (V)^{\otimes n}))$ is isomorphic to the cokernel of an explicit map:
\[
\bigoplus_{\substack{\lambda' \vdash n-1, \ \widehat{\lambda'} \preceq \rho  \\ \lambda'_1 = 2n -|\rho|  }}
\big((S^{(\rho / \widehat{\lambda'})})^{\sym_2} \otimes S^{\lambda'}\big)\uparrow_{\sym_{n-1}}^{\sym_n}
\rightarrow 
\bigoplus_{\substack{\lambda \vdash n, \ \widehat{\lambda} \preceq \rho \\ \lambda_1 = 2n -|\rho| }}
(S^{\rho / \widehat{\lambda}}  \otimes S^\lambda).
 \] 
\end{THM}

The codomain of the map in the statement corresponds to Theorem \ref{THMC} below, which should be viewed as  the most general {\em full} calculation of the paper, since the representation can be approached by standard methods in the representation theory of the symmetric groups. In the case of Theorem \ref{THMBafter},  although the domain can be treated likewise,  further work is required to calculate the map; hence this cannot be considered as a {\em full} calculation for the general case.

\subsection{Methods and further results}

The remainder of this Introduction is devoted to explaining the general techniques that are used to calculate $H_0 (\lie(V); \tlie (V)^{\otimes n})$ as a $\sym_n$-module, functorially with respect to $V$.  Along the way, this explains the {\em baby beads}  of the title. The reader is referred to Section \ref{sect:tw} for an introduction to the Turchin and Willwacher {\em bead representations}, their interpretation within the framework used here, together with the (subtle) relationship with the {\em baby beads}.

Since $\lie (V)$ is a free Lie algebra, the Lie algebra homology $H_* (  \lie(V); \tlie(V)^{\otimes n}) $ is calculated as the homology of the explicit, $\sym_n$-equivariant complex:
\begin{eqnarray}
\label{eqn:tlie_cx}
V \otimes \tlie(V)^{\otimes n} 
\rightarrow 
\tlie(V) ^{\otimes n}.
\end{eqnarray}
Here the differential is the restriction to $V = \lie_1 (V)$ of the tensor product of the  adjoint action 
$ \lie(V) \otimes \tlie(V)^{\otimes n} 
\rightarrow 
\tlie(V) ^{\otimes n}. $

Unfortunately, the simplicity of the description of the above complex belies the difficulty of calculating its homology.  The solution proposed in this paper is  to break the problem into two steps, as follows.

As functors, one has 
\[
\tlie (V) \cong \Lambda^1 (V) \oplus \Lambda^2 (V).
\]
One thus has a natural quotient map $\Lambda^1 (V) \oplus V^{\otimes 2} \twoheadrightarrow \Lambda^1 (V) \oplus \Lambda^2 (V)$  induced by the surjection $V^{\otimes 2} \twoheadrightarrow \Lambda^2 (V)$ sending $v \otimes w$ to $v \wedge w$. This is viewed as {\em imposing antisymmetry}.

The complex (\ref{eqn:tlie_cx})  has an explicit `lift':
\begin{eqnarray}
\label{eqn:cx_non_anti}
\xymatrix{
V \otimes (V \oplus V^{\otimes 2}) ^{\otimes n}
\ar@{->>}[d]
\ar[r]
&
(V \oplus V^{\otimes 2}) ^{\otimes n}
\ar@{->>}[d]
\\
V \otimes \tlie(V)^{\otimes n} 
\ar[r]
&
\tlie(V) ^{\otimes n},
}
\end{eqnarray}
where the rows are the complexes and the vertical maps the surjections induced by  antisymmetrization. This lift may be defined at the level of functors to sets.  The `coefficients' $(V \oplus V^{\otimes 2}) ^{\otimes n}$ are {\em not} given a $\lie(V)$-module structure.

The diagram (\ref{eqn:cx_non_anti}) allows the {\em dévissage} of the problem into the following two steps:
\begin{enumerate}
\item 
calculation of the cokernel of the top horizontal map; 
\item 
`antisymmetrization'. 
\end{enumerate}

The first step is carried out by using the main result of \cite{P_bb_finj}, which provides the key calculational input (the method will be outlined in more detail below). The following explicit statement is a reformulation of Corollary \ref{cor:hcd_k_beads_values} for the $\rho$-isotypical component of $\mathrm{Coker} ( V \otimes (V \oplus V^{\otimes 2}) ^{\otimes n}
\rightarrow 
(V \oplus V^{\otimes 2}) ^{\otimes n} )$, which is given by 
\[
\nt \big(S^\rho (V), \mathrm{Coker} ( V \otimes (V \oplus V^{\otimes 2}) ^{\otimes n}
\rightarrow 
(V \oplus V^{\otimes 2}) ^{\otimes n} ) \big).
\]

\begin{THM}
\label{THMC}
For $n \in \nat$ and a partition $\rho $, there is an isomorphism of $\sym_n$-modules:
\[
\nt \big(S^\rho (V), \mathrm{Coker} ( V \otimes (V \oplus V^{\otimes 2}) ^{\otimes n}
\rightarrow 
(V \oplus V^{\otimes 2}) ^{\otimes n} ) \big)
\cong 
\bigoplus_{\substack{\lambda \vdash n, \ \widehat{\lambda} \preceq \rho \\ \lambda_1 = 2n - |\rho|  }}
S^{\rho / \widehat{\lambda}}  \otimes S^\lambda.
\]
\end{THM}

This leaves the second step, the antisymmetrization. Concretely, one must impose the passage to the quotient corresponding to the induced map between the cokernels of the horizontal maps in (\ref{eqn:cx_non_anti}):
\[
\mathrm{Coker}\Big(   
V \otimes (V \oplus V^{\otimes 2}) ^{\otimes n}
\rightarrow
(V \oplus V^{\otimes 2}) ^{\otimes n}
\Big) 
\twoheadrightarrow 
H_0 (\lie(V) ; \tlie(V)^{\otimes n})
.
\]

 For this, it is not sufficient simply to know the underlying representation as in Theorem \ref{THMC}; more structure has to be carried through. This involves working with variants of the upper walled Brauer category, which is of importance in representation theory (see   \cite{MR3376738} and  \cite{MR991410}, for example).

The remainder of this introduction serves to introduce the categories that intervene, to indicate  how Theorem \ref{THMC} is proved whilst carrying along the required additional structure, and finally to  give a meaning to  `antisymmetrization'.

The category is $\cc$ is introduced in Section \ref{subsect:cc}; objects are pairs $(X,Y)$ of finite sets; a morphism from $(X,Y)$ to $(U,V)$ is given by a pair of injective maps $i:X \hookrightarrow U$, $j: Y \hookrightarrow V$ together with an injection $\alpha : V\backslash j(Y) \hookrightarrow U\backslash i(X)$. The wide subcategory $\cc_0$ is given by those morphisms for which $\alpha$ is a bijection; this is the upper walled Brauer category. The category $\cc$ also contains $\finj \times \fb$ as a wide subcategory, corresponding to maps for which $j$ is a bijection (here $\finj$ is the category of finite sets and injective maps and $\fb$ is its maximal subgroupoid).

Thus one can consider the functor category  $\apsh{\cc}$ (for an essentially small category $\calc$, $\apsh{\calc}$ denotes the category of functors from $\calc$ to $\kring$-modules). The key example of such an object for us is introduced in Section \ref{sect:exam_trans}, namely the functor
\[
(X,Y) \mapsto \kring \hfi (X,Y) \trans.
\]
For fixed $Y$, this gives a {\em covariant} functor on $\finj$, which corresponds to a `transpose' structure (as indicated by the $\trans$). For instance, taking $i : X \hookrightarrow X \amalg \mathbf{1}$, where $\mathbf{1}= \{1\}$, to be the canonical inclusion, 
\[
\kring \hfi (X,Y) \trans
\rightarrow 
\kring \hfi (X \amalg \mathbf{1},Y) \trans
\]
sends the generator $[f]$ corresponding to an injection $f : X \hookrightarrow Y$ to the sum $\sum [f^+]$ for  all possible extensions $f^+ : X \amalg \mathbf{1} \hookrightarrow Y$. Likewise, for the obvious map $(X,Y) \rightarrow (X\amalg \mathbf{1}, Y \amalg \mathbf{1}) $ in $\cc_0$, the  corresponding
\[
\kring \hfi (X,Y) \trans
\rightarrow 
\kring \hfi (X \amalg \mathbf{1},Y \amalg \mathbf{1}) \trans
\]
is given by  sending $[f] \mapsto [f \amalg \mathrm{id}_{\mathbf{1}}]$.

Extension by zero gives an inclusion $\apsh{\cc_0} \hookrightarrow \apsh{\cc}$ and this admits a left adjoint 
\[
\hcc : \apsh {\cc} \rightarrow \apsh{\cc_0}.
\]
This is related to  zeroth $\finj$-homology, $H_0^\finj$ (see Proposition \ref{prop:extn_left_adj_cc}). It is also the key ingredient behind Theorem \ref{THMC}, based upon the main result of \cite{P_bb_finj} (corresponding to  Theorem \ref{thm:calc} here), which  calculates 
\[
\hcc \kring \hfi (-,-) \trans.
\]

The link with Theorem \ref{THMC} requires introducing  the category $\cd$. This is similar in nature to $\cc$; objects are again pairs of finite sets and morphisms have underlying injections $(i,j)$ as before. The difference is that the additional datum is given by an injection $\zeta : \mathbf{2} \times    V\backslash j(Y) \hookrightarrow U\backslash i(X)$, where $\mathbf{2}=\{1,2\}$. For $\cd$, in addition to the isomorphisms, one has to treat the evident generating morphisms of the form
\begin{eqnarray*}
(X,Y) & \rightarrow & (X \amalg \mathbf{1}, Y) 
\\
(X,Y) & \rightarrow & (X \amalg \mathbf{2}, Y\amalg \mathbf{1}). 
\end{eqnarray*}
A new feature here is that the non-identity isomorphism of $\mathbf{2}$ plays a rôle.

The wide subcategory $\cd_0$ is given by the morphisms for which $\zeta$ is a bijection. Again,  extension by zero $\apsh {\cd_0} \hookrightarrow \apsh{\cd}$ admits a left adjoint:
\[
\hcd : \apsh{\cd} \rightarrow \apsh{\cd_0}
\]
that is related to $H_0^\finj$ (see Proposition \ref{prop:extn_cd}).

Our key example of an object of $\apsh{\cd}$ is $\kring \bbd (-,-)$ that is  defined in terms of bead arrangements (see Sections \ref{sect:tw} and  \ref{subsect:beads}); these are the baby beads of the title. For $n,N \in \nat$,  $\bbd (\mathbf{N},\mathbf{n})$ is the set of arrangements of $N$ beads into $n$-columns, each containing either one or two elements. This is empty if either $N< n$ or $2n < N$. For example, taking $n=5$ and $N=8$, the following represents an element of $\bbd (\mathbf{8}, \mathbf{5})$:
\begin{center}
 \begin{tikzpicture}[scale = 1.5]
 \draw [fill=white,thick] (.5,0) circle[radius = .15];
 \node at (.5,0) {$1$};
 \draw [fill=white,thick] (1,0) circle [radius = .15];
 \node at (1,0) {$4$};
\draw [fill=white,thick] (1.5,0) circle [radius = .15];
 \node at (1.5,0) {$7$};
\draw [fill=white,thick] (2,0) circle [radius = .15];
 \node at (2,0) {$8$};
 \draw [fill=white,thick] (2.5,0) circle [radius = .15];
 \node at (2.5,0) {$3$};
 \draw [fill=white,thick] (1,.5) circle [radius = .15];
 \node at (1,.5) {$5$};
 \draw [fill=white,thick] (2,.5) circle [radius = .15];
 \node at (2,.5) {$2$};
  \draw [fill=white,thick] (2.5,.5) circle [radius = .15];
 \node at (2.5,.5) {$6$};
 \node [below right] at (2.75,0) {.};
\end{tikzpicture}
\end{center}

The group $\sym_n$ acts on the set $\bbd (\mathbf{N},\mathbf{n})$  by permuting columns and $\sym_N$ by relabelling beads. The $\kring$-linearization defines a functor on $\cd$ in a similar way that $\kring \hfi (-,-)\trans$ yields a functor on $\cc$. Indeed, these are closely related:  if one considers only the arrangements for which the beads labelled $\{1, \ldots , n\}$ (supposing that $n \leq N$) appear in the bottom row and forgets their positions, a bead arrangement corresponds to a map in $\hfi (\mathbf{N-n}, \mathbf{n})$. 

Theorem \ref{thm:int_left_adjoint} formalizes this as follows.  There is an exact induction functor $\ind : \apsh{\cc} \rightarrow \apsh{\cd}$ that restricts to  $\ind[0] : \apsh{\cc_0} \rightarrow \apsh{\cd_0}$ and these are compatible in that  the following commutes up to natural isomorphism:
\[
 \xymatrix{
 \apsh{\cc}
 \ar[r]^\ind
 \ar[d]_{\hcc}
 &
\apsh{\cd} 
 \ar[d]^{\hcd}
 \\
 \apsh{\cc_0}
 \ar[r]_{\ind[0]}
 &
 \apsh{\cd_0}.
 }
 \]
Theorem \ref{thm:bbd_int} gives the isomorphism $\kring \bbd (-,-) \cong \ind \kring \hfi (-,-)\trans$ and hence:
\[
\hcd \kring \bbd (-,-) \cong \ind[0] \hcc \kring \hfi (-,-)\trans.
\]
The functor $\ind[0]$ can be calculated explicitly; hence the values of $\hcd \kring \bbd (-,-)$ can be calculated by applying Theorem \ref{thm:calc}.

Theorem \ref{THMC} is deduced from the above by showing that the Schur functor associated to  $\hcd \kring \bbd (-, \mathbf{n})$ (considered as a $\fb$-module) is isomorphic to the cokernel of the top row of diagram (\ref{eqn:cx_non_anti}). This is explained in Section \ref{sect:schur}.

We are now in a position to explain the antisymmetrization.  This uses the $\kring $-linear category $\ce$,  a quotient of the $\kring$-linearization $\kring \cd$ given by imposing a form of antisymmetry at the level of morphisms (see Section \ref{sect:signed}); the wide subcategory $\ce_0 \subset \ce$ is given as the corresponding quotient of $\kring \cd_0$.

The category $\ce\dash\modules$ is the category of $\kring$-linear functors from $\ce$ to $\kring$-modules, likewise for $\ce_0\dash\modules$. One has 
the extension by zero functor $\ce_0 \dash\modules\hookrightarrow \ce\dash\modules$ and this admits a left adjoint
\[
\hce : \ce\dash\modules \rightarrow \ce_0\dash\modules.
\]

The functor 
$ 
\lad : \apsh{\cd} 
\rightarrow 
\ce \dash\modules
$ 
imposing antisymmetry  is the left adjoint to the inclusion $\ce \dash\modules \subset  \apsh{\cd} $
 and restricts to $ 
\lad : \apsh{\cd_0} 
\rightarrow 
\ce_0 \dash\modules.
$ 
In particular, for $F \in \ob \apsh {\cd}$, the adjunction unit gives a natural surjection $F \twoheadrightarrow \lad F$; the notation $\lad$ is chosen to reflect the fact that $\lad F$ is a natural quotient of $F$ obtained by imposing antisymmetry.

Corollary \ref{cor:lad_cd0} shows that these are compatible, in that there exists a commutative diagram up to natural isomorphism:
 \[
 \xymatrix{
 \apsh {\cd} 
 \ar[r]^\lad 
\ar[d]_{\hcd}
&
\ce\dash\modules 
\ar[d]^{\hce}
\\
 \apsh {\cd_0} 
 \ar[r]_\lad 
 &
 \ce_0\dash\modules.
   }
 \]

This yields the following anticommutative version of $\kring \bbd(-,-)$:
\[
\bba (-,-) :=  \lad \kring \bbd (-,-) \in \ob \ce \dash \modules
\]
(see Section \ref{sect:exam_anti}). We have now arrived at the object that really interests us, namely $\bba (-,-)$, which  encodes the structure underlying the complex that calculates the  Lie algebra homology 
$H_* (\lie (V); (\lie_{\leq 2} (V))^{\otimes n})$. More precisely, in terms of the associated Schur functor, one has:

\begin{PROP}
\label{PROP1}
(Cf. Proposition \ref{prop:schur_hce_bba}.)
For  $n \in \nat$, there is a natural $\sym_n$-equivariant isomorphism:
\[
H_0 (\lie (V) ; \lie_{\leq 2} (V)^{\otimes n} )
\cong 
\big( \hce \bba (-,\mathbf{n}) \big) (V).
\]
\end{PROP}

This reformulates the main question of this paper as that of calculating 
\[
\hce \bba (-,-)  \in \ob \ce_0\dash\modules. 
\]

This is studied using the following Corollary of the results outlined above (cf. Corollary \ref{cor:hce_bba}): 

\begin{COR}
\label{COR}
There is an isomorphism in $\ce_0\dash\modules$:
\[
\hce \bba (-,-) \cong \lad \ind[0] \big(\hcc \kring \hfi (-,-)\trans \big).
\]
\end{COR}

The concatenation of the induction and antisymmetrization steps is thus  encapsulated in the composite 
\[
\lad \ind[0] : \apsh{\cc_0} \rightarrow \ce_0\dash\modules.
\]
This has a  concrete description, as explained in Section \ref{sect:compose} and  Section \ref{sect:rep_compose}, which gives a  representation-theoretic interpretation of this functor.

\subsection{Acknowledgements}
The author is grateful to Christine Vespa for comments on  this work and suggestions.

\tableofcontents
\part{Preliminaries}
\label{part:preliminaries}

\section{Background}
\label{sect:background}

This Section serves to recall some standard facts and notation.

\subsection{Notation and Conventions}

The category of sets is denoted $\sets$ and the full subcategory of finite sets  by $\fsets$. 
The category of finite sets and injective maps is denoted $\finj$ and   its maximal sub groupoid by $\fb$, aka. the category of finite sets and bijections. 

\begin{nota}
\label{nota:sym_n}
 For $n \in \nat$,
 \begin{enumerate}
 \item 
  $\mathbf{n}$ denotes the set $\{ 1, \ldots , n \}$ (so that $\mathbf{0}=  \emptyset$);
\item 
$\sym_n$ denotes the symmetric group on $n$ letters, which is usually identified with $\aut (\mathbf{n})$.
\end{enumerate}
\end{nota}

\begin{rem}
\label{rem:mathbf_nat}
\ 
\begin{enumerate}
\item 
The set of objects $\{ \mathbf{n}\  | \ n \in \nat \}$ forms a  skeleton of $\finj$ and of $\fb$.
\item 
For $m,n \in \nat$, the above notation leads to $\mathbf{m+n}$ for the set $\{1, \ldots , m+n \}$; here $\mathbf{+}$ should not be understood as an operation on $\mathbf{m}$ and $\mathbf{n}$, although one does have a bijection $\mathbf{m+n} \cong \mathbf{m} \amalg \mathbf{n}$. 
 Similarly, $\mathbf{m-1}$ denotes $\{ 1 , \ldots , m-1 \}$. 
 \end{enumerate}
\end{rem}

\begin{nota}
\label{nota:partitions}
\ 
\begin{enumerate}
\item
If $\lambda$ is a partition of $n \in \nat$,  write $\lambda \vdash n$ or  $|\lambda|=n$.
\item 
Young diagrams associated to partitions are defined so that $(n)$ has  diagram consisting of a row with $n$ boxes; $(1^n)$ has  diagram consisting of a single column.
\item 
The set of all partitions is equipped with the partial order $\preceq$, where $\lambda \preceq \mu$ if $\lambda_i \leq \mu_i$ for each $i$ (non-assigned values being taken as zero). In particular, this implies that $|\lambda |\leq |\mu|$ with equality if and only if $\lambda = \mu$. 
\item 
For a partition $\lambda$, $\widehat{\lambda}$ denotes the partition obtained by removing $\lambda_1$. (In terms of the associated Young diagram, one removes the first row.)
\item 
For a partition $\mu$ and $t \in \nat$ such that $t \geq \mu_1$, $(t)\cdot \mu$ denotes the partition obtained by concatenation: $(t, \mu_1, \mu_2, \ldots )$.
\item 
For  a partition $\lambda$, the transpose (or conjugate) partition is denoted by $\lambda^\dagger$. (This is the partition with Young diagram constructed from that of $\lambda$ by reflection in the diagonal, so that the rôle of rows and columns is reversed.)
\end{enumerate}
\end{nota}

\begin{nota}
\label{nota:simples_sym}
For $\kring$ a field of characteristic zero and $n \in \nat$, the simple $\kring \sym_n$-representation indexed by the partition $\lambda \vdash n$ is denoted by $S^\lambda$, indexed so that $S^{(n)}$ is the trivial representation  $\triv_n$ and $S^{(1^n)}$ is the signature representation $\sgn_n$.
\end{nota}

\begin{rem}
\label{rem:group_groupoid_op}
The inverse of a group $G$ induces an isomorphism of groups $G\cong G\op$, so that the category of left $G$-modules is equivalent to the category of right $G$-modules. Likewise, the category $\fb$ is isomorphic to $\fb\op$ so that a category of functors on $\fb$ is isomorphic to that of functors on $\fb\op$. 
\end{rem}

\begin{exam}
\label{exam:bimodule}
Suppose that $M$ is a bimodule with $\sym_b$ acting on the left and $\sym_a$ on the right, for $a, b \in \nat$. Then, equivalently, $M$ is a (left) $\sym_a\op \times \sym_b$-module. By the above, $M$ can also be considered as a left $\sym_a \times \sym_b$. Here $(\rho, \sigma) \in \sym_a \times \sym_b$ acts on $M$ by $ x \mapsto \sigma x \rho^{-1}$. 
\end{exam}

The application of Remark \ref{rem:group_groupoid_op} as in Example \ref{exam:bimodule} will be used systematically, without further comment. Context should ensure that no confusion results.

\subsection{Functor categories}

Throughout, $\kring$ is a fixed commutative, unital ring and we work with the abelian category $
\kring\dash\modules$ of $\kring$-modules.

\begin{nota}
\label{nota:apsh}
For $\calc$ an (essentially) small category, $\apsh {\calc}$ denotes the category of functors from $\calc$ to $\kring\dash\modules$, equipped with its usual abelian structure. (An object of $\apsh {\calc}$ is sometimes referred to as a $\calc$-module.)
\end{nota}

\begin{prop}
\label{prop:tensor_calc-modules}
The abelian category  $\apsh{\calc}$ has tensor product given by the pointwise tensor product of $\kring$-modules, which defines a symmetric monoidal structure on $\apsh{\calc}$ with unit the constant functor $\kring$.
\end{prop}

\begin{nota}
\label{nota:standard_proj}
For  $\calc$ an (essentially) small category and $X \in \ob \calc$, $P^\calc _X$ denotes the functor
\[
P^\calc_X (-) = \kring \hom_\calc (X, -).
\]
\end{nota}

\begin{rem}
Yoneda's lemma implies that the functors $P^\calc_X$, as $X$ ranges over a set of isomorphism class representatives of $\calc$, forms a set of projective generators for $\apsh{\calc}$.
\end{rem}

\begin{exam}
\label{exam:proj_finj}
For  $n\in \nat$, one has the projective functor $P_\mathbf{n}^\finj (-):= \kring \hom _\finj (\mathbf{n}, -)$ and $\{ P_\mathbf{n}^\finj  \ |\  n \in \nat \}$ is a set of projective generators for $\apsh{\finj}$. For $n>t \in \nat$, $ \kring \hom _\finj (\mathbf{n}, \mathbf{t})=0$, whereas for $t \geq n$, one has the isomorphism of bimodules
\begin{eqnarray}
\label{eqn:hfi}
 \kring [\sym_t/ \sym_{t-n}]\cong \kring \hfi (\mathbf{n}, \mathbf{t}) ,
\end{eqnarray}
where the left hand side is the permutation module with the canonical left $\sym_t$-action and right $\sym_n$-action via the action  of the Young subgroup $\sym_n \times \sym_{t-n} \subset \sym_t$. This right action corresponds to the (contravariant) naturality of $\mathbf{n} \mapsto P^\finj_{\mathbf{n}}$, restricted to the automorphisms of $\finj$.

The isomorphism (\ref{eqn:hfi}) is given explicitly by sending the class $[e_t]$ given by the identity $e_t \in \sym_t$ to the generator of $\kring \hfi (\mathbf{n}, \mathbf{t})$ given by $\iota_{\mathbf{n} , \mathbf{t}} \in \hfi (\mathbf{n}, \mathbf{t})$, the canonical inclusion $\{1, \ldots, n \} \subset \{ 1, \ldots , t\}$.

In particular, one has isomorphisms of bimodules: 
\begin{eqnarray*}
P_\mathbf{n}^\finj (\mathbf{n})&\cong & \kring \sym_n \\
P_\mathbf{n}^\finj (\mathbf{n+1}) & \cong & \kring \sym_{n+1} \downarrow^{\sym_{n+1}} _{\sym_n}
\end{eqnarray*}
and the canonical inclusion $\mathbf{n} \subset \mathbf{n+1}$ induces the morphism of $\sym_n$-bimodules
\[
P_\mathbf{n}^\finj (\mathbf{n})\cong  \kring \sym_n
\hookrightarrow 
\kring \sym_{n+1} \cong P_\mathbf{n}^\finj (\mathbf{n+1})
\]
induced by the inclusion $\sym_n \subset \sym_{n+1}$, using the restricted bimodule structure on the codomain.

{\em A contrario}, for $t > n+1$, $P_\mathbf{n}^\finj (\mathbf{t})$ is {\em not} isomorphic to $\kring \sym_t \downarrow ^{\sym_t} _{\sym_n}$. 
\end{exam}

The following allows one to decompose a small category into `connected components':

\begin{prop}
\label{prop:conn_compt}
If $\calc$ is a small category, there is an equivalence relation on $\ob \calc$ generated by the relation $X \sim Y$ if $\hom _\calc (X, Y) \neq \emptyset$. 

Then, writing $\calc_{[X]}$ for the connected component of $\calc$ corresponding to an equivalence class $[X]$ of an object $X$ (i.e., the full subcategory with objects  $[X]$), there is a decomposition 
\[
\calc = \amalg \calc_{[X]}
\]
where the disjoint union of categories is indexed by the equivalence classes of objects of $\calc$.  
\end{prop}

\begin{rem}
Proposition \ref{prop:conn_compt}  extends to essentially small categories and will be applied for such.
\end{rem}

\begin{cor}
\label{cor:conn_compt}
For $\calc$ as in Proposition \ref{prop:conn_compt}, there is an equivalence of categories:
\[
\apsh{\calc} 
\cong 
\prod \apsh{\calc_{[X]}},
\]
where $\apsh{\calc_{[X]}} \subset \apsh{\calc}$ is the full subcategory of functors supported on $\calc_{[X]}$ (i.e., $F \in \ob \apsh{\calc}$ such that $F(Y)=0$ if $Y \not \in \calc_{[X]}$).
\end{cor}

\subsection{Schur functors}
\label{subsect:schur}

For $\kring$ a field of characteristic zero, the Schur functor construction is a fundamental tool for passing between representations of the symmetric groups and polynomial functors defined on $\kring$-vector spaces. This can be promoted to a functor 
$
\apsh{\fb  } \rightarrow \apsh {\fvs}$,   
where $\fvs \subset \kring\dash\modules$ is the full subcategory of finite-dimensional vector spaces. This is given by
\[
G 
\mapsto 
G(V):= \bigoplus_{n \in \nat} V^{\otimes n} \otimes _{\sym_n} G(\mathbf{n}),
\]
where $\sym_n$ acts by place permutations on $V^{\otimes n}$. 

\begin{rem}
The component $V \mapsto V^{\otimes n} \otimes _{\sym_n} G(\mathbf{n})$ is a (homogeneous) polynomial functor of degree $n$. The Schur functor $G(V)$ is sometimes referred to as an analytic functor, since it is the colimit of its polynomial subfunctors.
\end{rem}

\begin{exam}
Suppose that $\rho \vdash N$ for $N \in \nat$ and consider the associated simple $\kring \sym_N$-module $S^\rho$ as an object of $\apsh{\fb}$. Then the associated functor 
 $
V \mapsto S^\rho (V)
$ 
is the usual Schur functor associated to $\rho$. 

For instance, taking $\rho = (N)$, this yields the $N$th symmetric power functor $S^N(V): = (V^{\otimes N})/\sym_N$ and, taking $\rho =(1^N)$, this yields the $N$th exterior power functor
 $\Lambda^N (V) := (V^{\otimes N} \otimes \sgn_N) /\sym_N$.
\end{exam}

\begin{nota}
\label{nota:nt}
For functors $\Phi, \Psi \in \ob \apsh {\fvs}$, the natural transformations 
from $\Phi$ to $\Psi$ are denoted by $\nt (\Phi (V), \Psi (V))$. 
\end{nota}

Notation \ref{nota:nt} will only ever be applied for $\Phi$ and $\Psi$ Schur functors; moreover $\Phi$ will usually be a homogeneous polynomial functor. 

For $n \in \nat$, taking $\Phi (V) = V^{\otimes n}$, one has the functor 
\[
\nt (V^{\otimes n}, - ) :  \apsh {\fvs} \rightarrow \kring \sym_n \dash \modules,
\]
where the $\sym_n$-action is given by the place permutation action on $V^{\otimes n}$. 
 This has the following key property:

\begin{lem}
\label{lem:schur_to_rep}
For $G \in \ob \apsh{\fb}$, there is a natural isomorphism of $\sym_n$-modules:
\[
\nt (V^{\otimes n}, G(V) ) 
\cong 
G(\mathbf{n}).
\]
Thus the $\fb$-module $G$ can be recovered from its Schur functor $V \mapsto G(V)$.
\end{lem}

In particular, this can be used to treat the isotypical components of a $\fb$-module via its Schur functor. 

\begin{prop}
\label{prop:isotyp_via_schur}
For $\kring$ a field of characteristic zero, $G \in \ob \apsh {\fb}$ and $\rho \vdash N$, there is a natural isomorphism of $\kring$-vector spaces:
\[
\hom_{\sym_N} (S^\rho, G(\mathbf{N}) ) 
\cong 
\nt (S^{\rho}(V), G(V) ).
\]
The dimension of these spaces is equal to the multiplicity of the simple $\sym_N$-module $S^\rho$ in $G(\mathbf{N})$; equivalently, to the multiplicity of the homogeneous polynomial functor $S^\rho(V) $ of degree $N$ in the analytic functor $G(V)$. 
\end{prop}

\subsection{The Lie representations}
In this subsection, for simplicity of exposition, $\kring$ is still taken to be a field of characteristic zero.

\begin{nota}
For $n \in \nat$, the $n$th Lie representation is denoted by $\lie_n$, which is  a representation of $\kring \sym_n$. (For $n=0$, this is zero.)  When the Lie representations are considered as an $\fb$-module, $\lie_n$ may also be written $\lie(\mathbf{n})$. 
\end{nota}

By construction, the free Lie algebra on a $\kring$-vector space $V$ is given by 
$
\lie (V) = \bigoplus_{n \in \nat} \lie_n (V)
$, 
where $V \mapsto \lie_n (V)$ is the Schur functor associated to $\lie_n$: $\lie_n (V) = V^{\otimes n} \otimes_{\sym_n} \lie_n$. 

\begin{rem}
By Lemma \ref{lem:schur_to_rep}, understanding the free Lie algebra $\lie (V)$ as a functor of $V$ to $\kring \dash \modules$ is equivalent to understanding the representations $\lie_n$. (The  Lie algebra structure of $\lie (V)$  is encoded in the additional {\em operad} structure on the $\fb$-module $\lie_\bullet$.) 
\end{rem}

The following is well-known (see \cite{MR1231799}, for example):

\begin{prop}
\label{prop:lie_reps}
Suppose that $\kring$ is a field of characteristic zero. 
There are isomorphisms of $\kring \sym_n$-modules (for the appropriate $n$):
\begin{enumerate}
\item 
$\lie _1 \cong \triv_1 \cong \sgn_1$; 
\item 
$\lie_2 \cong \sgn_2$; 
\item 
$\sgn_n$ is a composition factor of $\lie _n$ if and only if $n \in \{1, 2\}$; 
\item 
$\triv_n $ is a composition factor of $\lie _n$ if and only if $n=1$; 
\item 
$\lie_3 \cong S^{(2,1)}$; 
\item 
for $n \geq 3$, $S^{(2,1^{n-2})}$ (respectively $S^{(n-1,1)}$)  has multiplicity one in $\lie_n$.
\end{enumerate}
\end{prop}

\begin{proof}
\cite[Theorem 8.12]{MR1231799} determines the simple representations of $\kring \sym_n$ that occur in $\lie_n$ and \cite[Corollary 8.10]{MR1231799} gives a combinatorial formula for their multiplicities. 

For the current cases of interest, one can reason directly using the natural surjection 
$ 
\lie_1(V) \otimes \lie_n (V)
\twoheadrightarrow 
\lie_{n+1} (V)
$ 
given by the Lie bracket, as follows. By construction, $\lie_1 (V)= V$, which identifies $\lie_1$; the anticommutativity of the Lie bracket gives $\lie_2 (V) \cong \Lambda^2 (V)$ and $\lie_2 \cong \sgn_2$ as $\sym_2$-modules.  It follows that $\triv_n$ cannot be a composition factor of $\lie_n$ for $n \geq 2$, by using Pieri's rule.

Likewise, the Jacobi relation  gives $\lie_3 \cong S^{(2,1)}$; this implies that $\sgn_n$ cannot be a composition factor of $\lie_n$ for $n \geq 3$.

Finally, \cite[Theorem 8.12]{MR1231799} ensures that $S^{(2,1^{n-2})}$ and $S^{(n-1,1)}$ are composition factors of $\lie_n$ for $n \geq 3$ (it is not hard to see this directly). Again by using Pieri's rule, the above results imply that these occur with multiplicity at most one, whence the result.
\end{proof}

\section{Ingredients from representation theory}
\label{sect:rep_symm}
 
This Section provides some basic information from representation theory that is used in the text. The material is of a slightly technical nature, hence  the reader may prefer to  consult it as necessary.

\subsection{Induction and restriction}

The interplay between induction and restriction for tensor products of group representations is standard:

\begin{prop}
\label{prop:induct_restr}
For $H \subset G$ a subgroup and $M$ (respectively $N$) a $G$- (resp. $H$)-module, there is a natural isomorphism of $G$-modules:
\[
M  \otimes (N\uparrow) 
\cong 
(M\downarrow \otimes N) \uparrow
\]
where $\uparrow$ (respectively $\downarrow$) denotes induction from $H$- to $G$-modules  
 (resp. restriction from $G$- to $H$-modules) and the tensor products are equipped with the diagonal structure.
\end{prop}

\begin{proof}
Working with right modules, the mutually inverse isomorphisms are given in terms of elements $x \in M$, $y \in N$ and $g \in G$ by 
\begin{eqnarray*}
x \otimes (y \otimes g) & \mapsto & (x g^{-1} \otimes y ) \otimes g 
\\
(x \otimes y)\otimes g & \mapsto & xg \otimes (y \otimes g)
\end{eqnarray*}
where $x \otimes (y \otimes g) \in M \otimes (N\uparrow)$ and $(x \otimes y) \otimes g \in (M\downarrow \otimes N) \uparrow$.

The verification that these give well-defined morphisms of $G$-modules and that they are mutually inverse is left to the reader.
\end{proof}

This will be applied via the following, in which $\eta_N : N \rightarrow N \uparrow \downarrow$ is the unit of the induction/ restriction adjunction:

\begin{cor}
\label{cor:mono_induction_argument}
 Let $H, G$ and $M, N$ be as in Proposition \ref{prop:induct_restr} and suppose that $N$ is flat as a $\kring$-module. 

 Suppose that $M'$ is an $H$-module equipped with a monomorphism of $H$-modules $ \phi : M' \hookrightarrow M\downarrow $, then the morphism
 of $G$-modules 
\[
(M' \otimes N)\uparrow  \rightarrow M \otimes (N\uparrow )
\] 
 adjoint to   $ \phi \otimes \eta_N : M' \otimes N \rightarrow  (M \downarrow \otimes N \uparrow \downarrow) = (M \otimes N\uparrow ) \downarrow  $ is a monomorphism. 
\end{cor}

\begin{proof}
By construction, the given morphism sends $(x' \otimes y) \otimes g$, for $x'\in M'$, $y \in N$ and $g \in G$, to 
$\phi(x') g \otimes (y \otimes g)$. 

Composing with the isomorphism $M  \otimes (N\uparrow) 
\cong 
(M\downarrow \otimes N) \uparrow$ gives
\[
(M' \otimes N)\uparrow  \rightarrow (M\downarrow \otimes N) \uparrow.
\]
It clearly suffices to show that this is injective.

The image of  $(x' \otimes y) \otimes g$ under this composite is $(\phi (x') \otimes y) \otimes g$. Otherwise put, the composite morphism identifies with that obtained by applying induction to the morphism of $H$-modules $\phi \otimes \id_N : M' \otimes N \rightarrow M\downarrow \otimes N$.  

Since  $N$ is $\kring$-flat, $\phi \otimes \id_N$ is injective, by the hypothesis upon $\phi$. The result follows since the induction functor is exact.
\end{proof}

\subsection{Skew representations of the symmetric groups}
\label{subsect:skew_reps}

In  this section we consider representations of symmetric groups over a field $\kring$ of characteristic zero. This hypothesis ensures that the module categories that are considered are semisimple. 

Fix natural numbers $a \leq n$ and a bijection $\mathbf{a} \amalg (\mathbf{n-a}) = \mathbf{n}$ that gives the Young subgroup $\sym_a \times \sym_{n-a} \subset \sym_n$. 
One has the basic:

\begin{lem}
\label{lem:vanish_skew}
For $\lambda \vdash n $ and $\alpha \vdash a$, the $\sym_{n-a}$-module 
 $ 
S^\lambda \otimes _{\sym_a} S^\alpha
$
 is zero unless $\alpha \preceq \lambda$.
\end{lem} 

\begin{nota}
\label{nota:skew_rep}
For $\alpha \preceq \lambda$, denote by $S^{\lambda /\alpha}$ the $\sym_{n-a}$-module 
 $ 
S^\lambda \otimes _{\sym_a} S^\alpha
$. (This is the skew representation associated to the skew partition $\lambda/\alpha$.)

By convention, if $\alpha \not \preceq \lambda$, $S^{\lambda/ \alpha}$ is understood to be zero (cf. Lemma \ref{lem:vanish_skew}).
\end{nota}

\begin{rem}
The significance of the skew representations is illustrated by the isomorphism of $\sym_a \times \sym_{n-a}$-modules:
\[
(S^\lambda)\downarrow^{\sym_n}_{\sym_a \times \sym_{n-a}}
\cong 
\bigoplus_{\substack{\alpha \vdash a \\ \alpha \preceq \lambda}}
S^\alpha \boxtimes S^{\lambda/\alpha}.
\] 
When $a = n-1$, this corresponds to Pieri's rule:
\[
(S^\lambda)\downarrow^{\sym_n}_{\sym_{n-1}}
\cong 
\bigoplus_{\substack{\alpha \vdash n-1 \\ \alpha \preceq \lambda}}
S^\alpha.
\] 

The skew representations can be identified using the Littlewood-Richardson coefficients. These give the isomorphism:
\[
(S^\lambda)\downarrow^{\sym_n}_{\sym_a \times \sym_{n-a}}
\cong 
\bigoplus_{\substack{\alpha \vdash a, \beta \vdash {n-a} \\ \alpha, \beta \preceq \lambda} }
(S^\alpha \boxtimes S^\beta)^{\oplus c^\lambda_{\alpha, \beta} }, 
\] 
where $c^\lambda_{\alpha, \beta}$ is the Littlewood-Richardson coefficient. Hence one identifies for $\alpha \preceq \lambda$, $\alpha \vdash a$:
\[
S^{\lambda/ \alpha} \cong \bigoplus_{\beta \vdash n-a, \beta \preceq \lambda}
(S^\beta) ^{\oplus c^\lambda_{\alpha, \beta} }. 
\]
\end{rem}

We now consider the relationship between skew representations $S^{\lambda/ \beta} $ and $S^{\lambda/\alpha}$ when $\beta \preceq \alpha \preceq \lambda$. For simplicity, we suppose that $|\beta |= |\alpha|-1 =a-1$. Thus, by Pieri's rule, there is a unique (up to non-zero scalar multiple) monomorphism $S^\beta \hookrightarrow S^\alpha\downarrow ^{\sym_a}_{\sym_{a-1}}$. By adjunction, this yields the surjection 
\[
S^\beta \uparrow^{\sym_a}_{\sym_{a-1}}
\twoheadrightarrow
S^\alpha.
\]
On applying the functor $S^\lambda \otimes _{\sym_a} - $, one obtains
\[
S^{\lambda / \beta }\downarrow ^{\sym_{n-a+1}}_{\sym_{n-a}}
\cong 
S^\lambda \otimes _{\sym_a}
(S^\beta \uparrow^{\sym_a}_{\sym_{a-1}})
\twoheadrightarrow 
S^{\lambda/ \alpha},
\]
using the Young subgroup $\sym_{a-1} \times \sym_{n-a+1}$ defined so that $\sym_{a-1}\subset \sym_a$ and $\sym_{n-a} \subset \sym_{n-a+1}$.

\begin{lem}
\label{lem:skew_morphism}
For partitions $\beta \preceq \alpha \preceq \lambda$ with $|\beta |= |\alpha|-1 =a-1$ and $|\lambda|=n$, the above construction gives a surjective morphism of ${\sym_{n-a}}$-modules
\[
S^{\lambda / \beta }\downarrow ^{\sym_{n-a+1}}_{\sym_{n-a}}
\twoheadrightarrow
S^{\lambda/ \alpha},
\]
uniquely determined up to non-zero scalar multiple.
\end{lem}

\subsection{Some explicit examples}

Throughout this subsection, $\kring$ is a field of characteristic zero. We will require an understanding of the simple representation associated to $(2,1^{N-2})$ (where  $N\geq 2$). Since there is an isomorphism
\[
S^{(2,1^{N-2})} \cong \sgn_N \otimes S^{(N-1,1)},
\]
where the right hand side is given the diagonal structure, it is sufficient to consider the simple $S^{(N-1,1)}$. For this we have the following standard result:

\begin{lem}
\label{lem:perm_N}
For $N\geq 2$, there are isomorphisms of $\sym_N$-modules:
\[
\kring \mathbf{N} \cong S^{(N-1,1)} \oplus \triv_N \cong \triv_{N-1} \uparrow_{\sym_{N-1}}^{\sym_N},
\]
where $\kring \mathbf{N}$ denotes the permutation module on the $\sym_N$-set $\mathbf{N}$, which identifies as $\kring [\sym_N/ \sym_{N-1}]$.

In particular, $S^{(N-1,1)}$ identifies as:
\begin{enumerate}
\item 
the quotient $(\kring \mathbf{N})/\triv_N$, where the submodule $\triv_N$ is generated by the element $\sum_{i \in \mathbf{N}} [i]$; 
\item 
the kernel of the map $\kring \mathbf{N} \twoheadrightarrow \triv_N$ induced by the $\kring$-linearization of $\mathbf{N} \twoheadrightarrow \{* \}$.   
\end{enumerate}
Hence $S^{(N-1,1)}$ identifies as the subspace of $\kring \mathbf{N}$ generated by the elements $[j] -[i]$, for $i< j \in \mathbf{N}$; restricting to $j=N$ gives a basis.
\end{lem} 

Now take $1 \leq m < N$, so that $(m) \preceq (N-1,1)$. One can thus consider the skew representation
\[
S^{(N-1,1) / (m)} \cong S^{(N-1,1)} \otimes_{\sym_m} \triv _m
\]
of $\sym_n$, where $N= m+n$.

For the following, we extend the canonical inclusion $\mathbf{m} \subset \mathbf{N}$ to the bijection $\mathbf{m} \amalg \mathbf{n} \cong \mathbf{N}$ and thus define the Young subgroup $\sym_m \times \sym_n \subset \sym_N$. 

\begin{lem}
\label{lem:skew_N-1_1}
For $m,n,N$ as above, there is an isomorphism of $\sym_n$-modules
$ S^{(N-1,1) / (m)} \cong \kring \mathbf{n}$.  Explicitly, the corresponding surjection 
\[
S^{(N-1,1)} \downarrow_{\sym_n}^{\sym_N}
\twoheadrightarrow 
S^{(N-1,1) / (m)}
\cong 
\kring \mathbf{n}
\]
sends $ [j] -[i] \in S^{(N-1,1)}$, where $i<j$, to 
\[
\left\{
\begin{array}{ll}
0 & j \leq m \\
{ [j-m]} \in \kring \mathbf{n} & i \leq m, j > m \\
{ [j-m] - [i-m]}\in \kring \mathbf{n} & i> m
\end{array}
\right.
\]
\end{lem}

\begin{proof}
That the skew representation identifies as the induced representation $\triv_{n-1} \uparrow _{\sym_{n-1}} ^{\sym_n}$ is standard (and can be checked easily); this is isomorphic to the permutation representation $\kring \mathbf{n}$. 

The given map is the restriction to $S^{(N-1,1)}$ of the $\kring$-linear map $\kring \mathbf{N} \twoheadrightarrow \kring \mathbf{n}$ corresponding, via the isomorphism $\mathbf{N} \cong \mathbf{m} \amalg \mathbf{n}$, to the projection sending $\kring \mathbf{m}$ to zero. This is clearly $\sym_m \times \sym_n$-equivariant and induces the required map.
\end{proof}

Consider the partitions $(m-1) \preceq (m) \preceq (N-1,1)$ and hence the associated surjection 
\[
S^{(N-1,1) / (m-1) } 
\twoheadrightarrow 
S^{(N-1,1) / (m)} \cong \kring \mathbf{n}, 
\]
using the appropriate restriction to $\sym_n$ for the domain (cf. Lemma \ref{lem:skew_morphism}). (In the case $m=1$, the domain is simply $S^{(N-1,1)}$. )

Lemma \ref{lem:skew_N-1_1} leads directly to the following technical Lemmas, that are exploited in Section \ref{sect:first_exam}:

\begin{lem}
\label{lem:skew_N-1_1_m-1_m}
In the above situation: 
\begin{enumerate}
\item 
the element $[m+1]-[m] \in S^{(N-1,1)}$ represents a non-zero element of $S^{(N-1,1) / (m-1) }$; 
\item 
for $\sym_2 \subset \sym_N$ the Young subgroup corresponding to $\{m, m+1 \} \subset \mathbf{N}$, the element $[m+1]-[m]$ generates a copy of $\sgn_2 \subset S^{(N-1,1)} \downarrow ^{\sym_N}_{\sym_2}$; 
\item 
the image of $[m+1]-[m]$ in $S^{(N-1,1) / (m)} \cong \kring \mathbf{n}$ identifies with $[1] \in \kring\mathbf{n}$, in particular generates $\kring \mathbf{n}$ as a $\sym_n$-module.
\end{enumerate}
Moreover, after relabelling so that $\sym_2 \subset \sym_{n+1}$ corresponds to the Young subgroup for $\{ m,m+1 \} \subset \{ m, \ldots , N \}$, 
\[
S^{(N-1,1) / (m-1) } \downarrow ^{\sym_{n+1}}_{\sym_2} \cong \sgn_2 \oplus \triv_2^{\oplus n-1}
\]
and $[m+1]-[m]$ represents a generator of the submodule $\sgn_2$.
\end{lem}

Passing from $S^{(N-1,1)}$ to $S^{(2,1^{N-2})}$, so that $(m)$ is replaced by $(1^m)$ and $(m-1)$ by $(1^{m-1})$, one considers the map of $\sym_n$-modules
\[
S^{(2,1^{N-2}) / (1^{m-1}) } 
\twoheadrightarrow 
S^{(2,1^{N-2}) / (1^{m}) } 
\]
(again omitting the restriction for the domain from the notation). 
 Lemma \ref{lem:skew_N-1_1_m-1_m} gives:

\begin{lem}
\label{lem:sym2_invt_generates}
For $N\geq 2$ and $1 \leq m < N$, for the Young subgroup $\sym_2 \times\sym_{n-1}  \subset \sym_{n+1}$, 
$$
(S^{(2,1^{N-2})/ (1^{m-1}) )})^{\sym_2} 
\cong \sgn_{n-1}. 
$$
The image of $(S^{(2,1^{N-2})/ (1^{m-1}) )})^{\sym_2} $ in $S^{(2,1^{N-2})/ (1^{m})} \cong \kring \mathbf{n} \otimes \sgn_n$ generates this $\sym_n$-module. Explicitly, the induced map 
\[
\big((S^{(2,1^{N-2})/ (1^{m-1}) )})^{\sym_2})\big) \uparrow_{\sym_{n-1}}^{\sym_n} 
\cong 
\sgn_{n-1} \uparrow_{\sym_{n-1}}^{\sym_n}
\rightarrow 
S^{(2,1^{N-2})/ (1^{m})}
\]
is an isomorphism.
\end{lem}

\subsection{The regular bimodule $\kring \sym_t$}
\label{subsect:bimodules}

We  require information on certain bimodules related to the bimodules $\kring \sym_t$, $t \in\nat$, (with the regular actions) and their relationships. These are established here by explicit, elementary arguments.

The group ring $\kring G$ is not free as a $G$-bimodule, in general. Indeed, one has the following characterization of bimodule maps out of $\kring G$, equipped with the regular actions:

\begin{lem}
\label{lem:regular_bimodule_hom}
For $G$ a group and $M$ a $G$-bimodule, there is a natural isomorphism 
$ 
\hom_{G-\mathrm{bimod}} (\kring G , M) 
\cong 
M^{[G]}$,  
where $\kring G$ is equipped with the regular actions and $M^{[G]}$ is the submodule $\{x \in M \ | \ g x = x g \   \forall g \in G\}$. 
Moreover, $M^{[G]}$ is a $Z (\kring G)$ submodule of $M$, where $Z (\kring G)$ is the centre of the group ring.
\end{lem}

For the group rings of the symmetric groups, one has the following description of the bimodule structure:

\begin{prop}
\label{prop:decompose_ksym_bimodule}
For $\kring$ a field of characteristic zero and $t \in \nat$, there is a canonical decomposition 
as bimodules 
$
\kring \sym_t
\cong 
\bigoplus_{\lambda \vdash t} 
c_\lambda \kring \sym_t
$, 
where $c_\lambda \in Z (\kring \sym_t)$ is a primitive central idempotent and there is an isomorphism of bimodules 
\[
c_\lambda \kring \sym_t \cong S^\lambda \boxtimes S^\lambda,
\]
and these are simple. Moreover, the idempotents are pairwise orthogonal: for $\lambda, \lambda' \vdash t$, $c_{\lambda}c_{\lambda'} = 0$ if $\lambda \neq \lambda'$.

Hence, for $M$ a $\sym_t$-bimodule, $\hom_{\sym_t-\mathrm{bimod}} (S^\lambda \boxtimes S^\lambda , M) \cong \hom_{\sym_t-\mathrm{bimod}} (c_\lambda \kring \sym_t , M) 
\cong 
c_\lambda (M^{[\sym_t]}).
$ 
\end{prop}

Now, consider the standard inclusion denoted here by $i_t :\sym_t \hookrightarrow \sym_{t+1}$; this induces a morphism of $\sym_t$-bimodules: $\kring \sym_t \hookrightarrow \kring\sym_{t+1} \downarrow$, where the codomain is given the restricted $\sym_t$-bimodule structure, denoted here simply by $\downarrow$.

\begin{lem}
\label{lem:morphisms_simple_bimodules}
For $\kring$ a field of characteristic zero, $t \in \nat$ and  $\lambda \vdash t$,  $\mu \vdash t+1$, 
\[
\hom_{\sym_t-\mathrm{bimod}} (S^\lambda \boxtimes S^\lambda ,(S^\mu \boxtimes S^\mu)\downarrow ) 
\cong 
\hom_{\sym_t-\mathrm{bimod}} (c_\lambda \kring \sym_t , (c_\mu \kring \sym_{t+1} ) \downarrow ) 
= 
\left\{ 
\begin{array}{ll}
\kring & \mbox{$\lambda \preceq \mu$} \\
0 & \mbox{otherwise}.
\end{array}
\right.
\]
When $\lambda \preceq \mu$, $i_t(c_\lambda) c_\mu= c_\mu i_t( c_\lambda)$ is a non-zero idempotent of $\kring \sym_{t+1}$.
\end{lem}

\begin{proof}
The first statement follows from the fact that $\hom_{\sym_t} (S^\lambda , S^\mu \downarrow^{\sym_{t+1}} _{\sym_t})$ is zero unless $\lambda \preceq \mu$, when it is $\kring$. 
The second statement then follows from Proposition \ref{prop:decompose_ksym_bimodule}. 
\end{proof}

\begin{rem}
Whilst $i_t(c_\lambda)$ is an idempotent in $\kring \sym_{t+1}$, it is not in general central. 
\end{rem}

From Lemma \ref{lem:morphisms_simple_bimodules}, one deduces:

\begin{prop}
\label{prop:bimod_inclusions}
For $\kring$ a field of characteristic zero, $t \in \nat$ and  $\lambda \vdash t$, $\mu \vdash t+1$, the composite
\[
c_\lambda \kring \sym_t
\hookrightarrow 
\kring \sym_t
\hookrightarrow 
\kring \sym_{t+1} \downarrow 
\twoheadrightarrow 
(c_\mu \kring \sym_{t+1}) \downarrow , 
\]
in which the outer morphisms are provided by the decomposition of Proposition \ref{prop:decompose_ksym_bimodule} and the remaining one by the inclusion $\sym_t \subset \sym_{t+1}$,  is non-trivial if and only if $\lambda \preceq \mu$. 
\end{prop}

\begin{proof}
The composite is determined by the image of $c_\lambda \in c_\lambda \kring \sym_t $ which is  $c_\lambda c_\mu$. This is non-zero if and only if $\lambda \preceq \mu$, by Lemma \ref{lem:morphisms_simple_bimodules}.
\end{proof}

\part{Truncating, beads and baby beads}
\label{part:truncating}

\section{$\catlie$, its truncation, $\cattlie$, and their representations}
\label{sect:catlie}

Throughout this Section, $\kring$ is taken to be a field of characteristic zero.   

We first introduce in section \ref{subsect:flie} the abelian categories $\flie$ and $\flie^\mu$ that are the key players in the theory; the truncation $\lie \twoheadrightarrow \lie_{\leq 2}$ is introduced and  discussed within this framework in section \ref{subsect:truncate_flie}. 

\subsection{The categories $\flie$ and $\flie^\mu$}
\label{subsect:flie}

To fix notation, this subsection reviews some material from \cite{2021arXiv211001934P} and \cite{2022arXiv220113307P}; the reader should consult these references for details. 

\begin{nota}
Denote by
\begin{enumerate}
\item 
$\catlie$ the $\kring$-linear category associated to the Lie operad $\lie$; 
\item 
$\flie$ the category of representations of $\catlie$, i.e., $\kring$-linear functors from $\catlie$ to $\kring$-vector spaces;
\item 
$\flie^\mu \subset \flie$ the full subcategory of representations on which the appropriate generalization of the adjoint action vanishes (see \cite{2022arXiv220113307P} for details).
\end{enumerate}
\end{nota}

The categories $\flie$ and $\flie^\mu$ are abelian and the (exact) inclusion $\flie^\mu \subset \flie$ admits a left adjoint:
\[
(-)^\mu : \flie \rightarrow \flie^\mu. 
\]

There is a forgetful functor $\flie \rightarrow \apsh{\fb}$, so that a $\catlie$-representation $F$ has an underlying sequence $F(\mathbf{n})$ of $\kring \sym_n$-representations, for $n \in \nat$. The functor $\flie \rightarrow \apsh{\fb}$ admits a section that factors across $\flie^\mu$:
$$
\apsh{\fb} \hookrightarrow \flie^\mu \subset \flie.
$$
The simple objects of $\flie$ (respectively $\flie^\mu$) lie in the image of this section.

\begin{exam}
For $n\in \nat$, $\kring\sym_n$ is a semi-simple object of $\flie^\mu$ and of $\flie$.
\end{exam}

The category $\flie$ has enough projectives as seen using the Yoneda lemma. Hence the category $\flie^\mu$ also has enough projectives; these can be constructed by applying the functor $(-)^\mu$ to the projective generators of $\flie$.

\begin{rem}
Since $\flie$ has enough projectives, one can do homological algebra in this category. This is relatively well-understood. Indeed, the fact that $\lie$ is a (quadratic) Koszul operad is equivalent to the fact that $\flie$ is a Koszul abelian category. In particular
\[
\ext^i_{\flie} (\kring \sym_m , \kring \sym_n)=0 \mbox{ if $i \neq m-n$ }
\] 
and the non-zero values identify with the category associated to the operadic suspension of the commutative operad, $\com$. (See \cite{2021arXiv211001934P} for details; the ideas go back to the seminal work of Ginzburg and Kapranov \cite{MR1301191}.)  
\end{rem}

The projective generators of $\flie$ and of $\flie^\mu$ are encoded in their associated Schur functors (see \cite[Section 6.4]{2022arXiv220113307P}). The Schur functor associated to the Lie operad $\lie$ is the functor $V \mapsto \lie (V)$, where $\lie (V)$ is the free Lie algebra on $V$ (which thus comes equipped with a natural Lie algebra structure).  Then there is a $\catlie$-representation with underlying $\fb$-module given by 
\[
\mathbf{n} \mapsto \lie(V)^{\otimes n},
\]
with $\sym_n$ acting via place permutations. The full $\catlie$-action is induced by the Lie algebra structure of $\lie(V)$. 

The tensor product of the adjoint action of $\lie (V)$ makes $\lie (V)^{\otimes n}$ into a $\lie (V)$-module. As in \cite[Section 6.4]{2022arXiv220113307P}, applying $(-)^\mu$ to the above object of $\flie$ gives 
\[
\mathbf{n}    \mapsto H_0 (\lie(V);\lie(V)^{\otimes n}),
\]
where the right hand side denotes zero degree Lie algebra homology. This is an object of $\flie^\mu$ which is a quotient $\catlie$-representation of $ \mathbf{n} \mapsto \lie(V)^{\otimes n}$.

Now, for a partition $\rho \vdash N$, with associated simple $\kring \sym_N$-module $S^\rho$, one can consider $S^\rho$ as an object of $\flie^\mu$ and of $\flie$. One also has the associated Schur functor $V \mapsto S^\rho (V)$. The following gives an explicit description of the projective cover of $S^\rho$ in $\flie$ (respectively $\flie^\mu$):

\begin{prop}
\label{prop:proj_cover}
For $\rho \vdash N$,
\begin{enumerate}
\item
the projective cover of $S^\rho$ in $\flie$ is the $\catlie$-representation 
\[
\mathbf{n} 
\mapsto 
\nt (S^{\rho}(V), \lie (V)^{\otimes n}),
\]
considered as a direct summand of $\mathbf{n} \mapsto \lie (V)^{\otimes n}$;
\item 
the projective cover of $S^\rho$ in $\flie^\mu$ is the $\catlie$-representation 
\[
\mathbf{n} 
\mapsto 
\nt (S^{\rho}(V), H_0(\lie(V); \lie (V)^{\otimes n})),
\]
considered as a direct summand of $\mathbf{n} \mapsto H_0(\lie(V);\lie (V)^{\otimes n})$.
\end{enumerate}
\end{prop}

\begin{proof} 
By the above discussion (cf. \cite[Section 6.4]{2022arXiv220113307P}), the projective cover of $\kring \sym_N$ in $\flie$ is given by $\mathbf{n} \mapsto \nt (V^{\otimes N}, \lie (V)^{\otimes n})$. The projective cover of $S^\rho$ is obtained by applying $- \otimes_{\sym_N} S^\rho$, which commutes with the formation of the natural transformations. The result follows, by definition of the Schur functor $S^\rho (V)$.

The proof of the second statement is similar, by using \cite[Theorem 6.25]{2022arXiv220113307P}, which describes the projective cover of $\kring \sym_N$ in $\flie^\mu$ as $ \mathbf{n} \mapsto \nt (V^{\otimes N}, H_0(\lie(V); \lie (V)^{\otimes n}))$.
\end{proof}

\begin{rem}
\label{rem:catlie_N}
The above can be treated without passage to Schur functors, as follows. For  $N \in \nat$, $\catlie (\mathbf{N}, -)$, defined by $\mathbf{n} \mapsto \catlie (\mathbf{N}, \mathbf{n})$, is an object of $\flie$. By Yoneda's lemma, it is the projective cover of $\kring \sym_N$ in $\flie$.
 On applying $(-)^\mu : \flie \rightarrow \flie^\mu$, one obtains $\catlie (\mathbf{N}, -)^\mu$, which is the projective cover of $\kring \sym_N$ in $\flie^\mu$. 
 
These can then be decomposed into isotypical components to obtain the analogue of Proposition \ref{prop:proj_cover}.
\end{rem}

\subsection{Truncating from $\lie$ to $\lie_{\leq 2}$}
\label{subsect:truncate_flie}

Denote by $\lie\twoheadrightarrow \tlie$ the truncation of operads defined by 
\[
\tlie (\mathbf{n}) = \left\{
\begin{array}{ll}
\lie (\mathbf{n}) & n \leq 2 \\
0 & \mbox{otherwise.}
\end{array}
\right.
\] 

On the associated Schur functors, this corresponds to the surjection of Lie algebras:
\[
\lie(V) 
\twoheadrightarrow 
\tlie (V) 
\cong 
\Lambda^1 (V) \oplus \Lambda^2 (V)
\]
defined by killing all commutators of length at least three. (Here the Lie algebra structure on $\Lambda^1 (V) \oplus \Lambda^2 (V)$ is the obvious one: on $\Lambda^1 (V) \otimes \Lambda^1 (V)$  the bracket $[-,-]$  is  the surjection $\Lambda^1 (V) \otimes \Lambda^1 (V)= V \otimes V \twoheadrightarrow \Lambda^2 (V)$ given by $v \otimes w \mapsto v \wedge w$; all other components are zero.)

\begin{rem}
This may appear to be a brutal truncation; however,  this paper shows that  it retains essential information when considering the analogous framework that leads to $\flie$ and $\flie^\mu$. 
\end{rem}

An additional reason for focussing upon the truncation $\tlie$ (rather than a less brutal truncation) is the following:

\begin{prop}
\label{prop:tlie}
\cite{MR3013090}
The operad $\tlie$ is (quadratic) Koszul, with Koszul dual the operad $\commag$, the commutative magmatic operad, i.e., the free binary quadratic operad with one commutative generator. The Koszul dual to the surjection $\lie \twoheadrightarrow \lie_{\leq 2}$ is the surjection 
\[
\commag \twoheadrightarrow \com
\]
that imposes the associativity relation.
\end{prop}

As for the case of $\lie$, one can form the categories $\ftlie$ and $\ftlie^{\mu}$; in particular, the general theory of \cite{2022arXiv220113307P} applies to the operad $\tlie$, yielding:

\begin{prop}
\label{prop:catlie_cattlie}
There is a commutative diagram of inclusions of full abelian subcategories:
\[
\xymatrix{
\ftlie^\mu
\ar@{^(->}[r]
\ar@{^(->}[d]
&
\ftlie 
\ar@{^(->}[d]
\\
\flie^\mu
\ar@{^(->}[r]
& 
\flie,
}
\]
in which the vertical inclusions are induced by restriction along $\catlie \twoheadrightarrow \cattlie$.

Moreover,
\begin{enumerate}
\item 
these functors all admit left adjoints;
\item 
the section $\apsh{\fb} \rightarrow \flie^\mu$ factors across $\ftlie^\mu$; in particular, all simple functors of $\flie$ belong to $\ftlie^\mu$. 
\end{enumerate}
\end{prop}

As for the case of $\flie$ and $\flie^\mu$, the categories $\ftlie$ and $\ftlie^\mu$ both have enough projectives. More precisely, one has the following analogue of Proposition \ref{prop:proj_cover}:

\begin{prop}
\label{prop:proj_cover_tlie}
For $\rho \vdash N$,
\begin{enumerate}
\item
the projective cover of $S^\rho$ in $\ftlie$ is the $\cattlie$-representation 
\[
\mathbf{n} 
\mapsto 
\nt (S^{\rho}(V), \tlie (V)^{\otimes n}),
\]
considered as a direct summand of $\mathbf{n} \mapsto \tlie (V)^{\otimes n}$;
\item 
the projective cover of $S^\rho$ in $\ftlie^\mu$ is the $\cattlie$-representation 
\[
\mathbf{n} 
\mapsto 
\nt (S^{\rho}(V), H_0(\lie(V); \tlie (V)^{\otimes n})),
\]
considered as a direct summand of $\mathbf{n} \mapsto H_0(\lie(V);\tlie (V)^{\otimes n})$.
\end{enumerate}

The canonical morphism from the projective covers of $S^\rho$ formed in $\flie$ and $\ftlie$ respectively   is induced by the natural surjection 
\[
\lie(V)^{\otimes n} 
\twoheadrightarrow 
\tlie(V)^{\otimes n}
\]
of $\catlie$-representations. Likewise, for the respective categories $\flie^\mu$ and $\ftlie^\mu$, is induced by the natural surjection 
\[
H_0(\lie(V); \lie(V)^{\otimes n} )
\twoheadrightarrow 
H_0(\lie(V); \tlie(V)^{\otimes n}).
\]
\end{prop}

The general Koszul duality framework (see \cite{2021arXiv211001934P}) in conjunction with Proposition \ref{prop:tlie} implies:

\begin{prop}
\label{prop:ftlie_koszul}
The category $\ftlie$ is Koszul abelian. Explicitly 
$$
\ext^i_{\ftlie} (\kring \sym_m , \kring \sym_n)=0 \mbox{ if $i \neq m-n$ }
$$
and the non-zero values are given by the category associated to the operadic suspension of $\commag$. 

The (exact) inclusion $\ftlie \hookrightarrow \flie$ induces surjections
\[
\ext^i_{\ftlie} (\kring \sym_m , \kring \sym_n)
\twoheadrightarrow 
\ext^i_{\flie} (\kring \sym_m , \kring \sym_n)
\]
induced by the surjection of operads $\commag \twoheadrightarrow \com$. In particular, for $i=1$, the map is an isomorphism.
\end{prop}

It is not known whether  the categories $\flie^\mu$ and $\ftlie^\mu$ are Koszul abelian. Without knowing this, one has no immediate analogue of Proposition \ref{prop:ftlie_koszul} for these categories. However,  one does have the following, which  is proved by standard techniques.

\begin{prop}
\label{prop:weak_koszul}
For $m, n \in \nat$,
\[
\ext^i_{\ftlie^\mu} (\kring \sym_m , \kring \sym_n)=
\ext^i_{\flie^\mu} (\kring \sym_m , \kring \sym_n)=0 
\mbox{ if $ i > m-n$}
\]
and, for $i\in \{0,1\}$, if $i \neq m-n$. 

The comparison morphism 
$ 
\ext^i_{\ftlie^\mu} (\kring \sym_m , \kring \sym_n)
\rightarrow 
\ext^i_{\flie^\mu} (\kring \sym_m , \kring \sym_n)
$ 
is an isomorphism for $i \in \{ 0, 1 \}$ and is surjective if $i=2$ and $m=n+2$.
\end{prop}

\begin{rem}
A fundamental problem  is to understand the (homological) behaviour of the functor:
$
\ftlie^\mu \hookrightarrow \flie^\mu$. 
Whereas Proposition \ref{prop:ftlie_koszul} shows that there is a great difference between the categories $\ftlie$ and $\flie$, upon restriction to $\ftlie^\mu $ and $ \flie^\mu$, this is  more subtle, as indicated by Theorem \ref{thm:2_1_N-2_isotypical}, for example.
\end{rem}

\section{Bead representations and baby beads}
\label{sect:tw}

Section \ref{subsect:TW_catlie} introduces the bead representations of Turchin and Willwacher  \cite{MR3982870} and uses the material of Section \ref{sect:catlie} to explain how they can be interpreted  intrinsically in $\flie^\mu$. Then, in Section \ref{subsect:relate_bbd}, the relationship with the {\em baby beads} of the title is sketched. 

\subsection{The Turchin-Willwacher representations and $\catlie$-representations}
\label{subsect:TW_catlie}

In this section, we explain the relationship between the work of Turchin and Willwacher, in particular their {\em bead representations}, and the framework provided by $\catlie$-representations.  To be  explicit, we are considering the  bead representations as  introduced in \cite[Section 2.5]{MR3982870}; these split into two components, namely the {\em type I} and {\em type II} components, in the terminology of {\em loc. cit.}.

These  representations \cite{MR3653316,MR3982870} arise in studying the higher Hochschild homology of a wedge of circles $\bigvee_r S^1$ with suitable coefficients. Since this wedge has the homotopy type of a classifying space $B\fg_r$ for the free group $\fg_r$ on $r$ generators, naturality yields functors on the category $\gr$ of finitely-generated free groups, i.e., objects of $\f(\gr)$. For the bead representations, the coefficients are basepoint independent; hence, upon restricting to the action of automorphisms $\mathrm{Aut}(\fg_r)$, the inner automorphisms act trivially, so that one obtains representations of the outer automorphism groups $\mathrm{Out}(\fg_r)$. 

So as to exploit the full functoriality on $\gr$, in  \cite{2018arXiv180207574P}, the authors introduced the full subcategory $\f^{\mathrm{Out}}(\gr) \subset \f(\gr)$ of {\em outer functors} from $\gr$ to $\kring$-vector spaces, namely those on which the inner automorphisms act trivially.  The interpretation of the above representations in $\f^{\mathrm{Out}}(\gr)$ is explained in \cite[Remark 16.17]{2018arXiv180207574P}; that work also explains why it is sufficient to focus upon the type I components. 

These  representations assemble to give  {\em polynomial} functors on $\gr$.
 Working with polynomial functors allows us to exploit an infinitesimal interpretation, as explained below. For this we pass to the contravariant setting, i.e., functors on $\gr\op$, by using vector space duality.
 
The category of polynomial functors of degree at most $d \in \nat$ is denoted $\f_d (\gr\op)$; this is a full subcategory of $\f(\gr\op)$ (see \cite[Section 7]{2021arXiv211001934P} for some details on polynomial functors). The category $\f_\omega (\gr\op)$ of {\em analytic} functors on $\gr\op$ is the full subcategory of functors that are the colimit of their polynomial subfunctors, so that, for each $ d\in \nat$, there are inclusions of full subcategories $  \f_d (\gr\op) \subset \f_\omega (\gr\op) \subset \f(\gr\op)$.  

The main result of  \cite{2021arXiv211001934P} shows that the category $\f_\omega (\gr\op)$  is equivalent to $\flie$. Under this equivalence, the polynomial functors of degree at most $d$ correspond to the objects $F \in \ob \flie$ such that $F(\mathbf{t})=0$ for $t >d$.  

One also  has the full subcategory $\f^{\mathrm{Out}} (\gr\op) \subset \f (\gr\op)$ of outer functors in the contravariant setting (see \cite{2018arXiv180207574P}). This restricts to polynomial and analytic functors respectively, giving the following inclusions of full subcategories (for all $d \in \nat$):
\[
\xymatrix{
\f^{\mathrm{Out}}_d (\gr\op)
\ar@{^(->}[r]
\ar@{^(->}[d]
&
\f^{\mathrm{Out}}_\omega (\gr\op)
\ar@{^(->}[r]
\ar@{^(->}[d]
&
\f^{\mathrm{Out}} (\gr\op)
\ar@{^(->}[d]
\\
\f_d (\gr\op)
\ar@{^(->}[r]
&
\f_\omega(\gr\op)
\ar@{^(->}[r]
&
\f (\gr\op).
}
\]

The inclusion $\f^{\mathrm{Out}}(\gr\op) \subset \f (\gr\op)$ admits a left adjoint,  denoted $\Omega$ in \cite{2018arXiv180207574P}. This restricts to give the left adjoint $\Omega : \f_\omega (\gr\op)\rightarrow \foutan$ to the inclusion $\foutan \subset \f_\omega(\gr\op)$. 

In \cite{2022arXiv220113307P}, it is  shown that, 
under the equivalence between $\f_\omega (\gr\op)$ and $\flie$,   the full subcategory $\foutan \subset \f_\omega (\gr\op)$ of analytic outer functors  corresponds to $\flie^\mu \subset \flie$. Moreover, via the above identifications, $\Omega: \f_\omega (\gr\op)\rightarrow \foutan$ corresponds to the left adjoint
$
(-)^\mu : \flie \rightarrow \flie^\mu
$
of the inclusion $\flie^\mu \subset \flie$.

\bigskip

We now provide an algebraic model for Turchin and Willwacher's bead representations. 
The reader should note that this model avoids the intricacies (as in \cite[Remark 16.17]{2018arXiv180207574P}) involved in passing between Hochschild homology and cohomology (dealt with by duality) and dealing with the Koszul signs. 

\begin{defn}
For a fixed $N \in \nat$ and  $c \in \nat$, let $\beads (\mathbf{N}, \mathbf{c})$ denote the set of arrangements of $N$ beads into $c$ ordered columns (where a column may be empty).

The group $\sym_N$ acts on $\beads (\mathbf{N}, \mathbf{c})$ by relabelling the beads and $\sym_c$ by permuting the columns.
\end{defn}

 An example of an element of $\beads (\mathbf{8}, \mathbf{5})$ is:
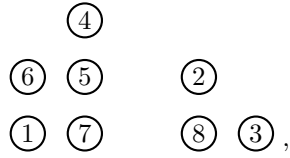
\begin{center}
 \begin{tikzpicture}[scale = 1.5]
 \draw [fill=white,thick] (.5,0) circle[radius = .15];
 \node at (.5,0) {$1$};
 \draw [fill=white,thick] (.5,.5) circle [radius = .15];
 \node at (.5,.5) {$6$}; 
 \draw [fill=white,thick] (1,0) circle [radius = .15];
 \node at (1,0) {$7$};
\draw [fill=white,thick] (1,1) circle [radius = .15];
 \node at (1,1) {$4$};
\draw [fill=white,thick] (2,0) circle [radius = .15];
 \node at (2,0) {$8$};
 \draw [fill=white,thick] (2.5,0) circle [radius = .15];
 \node at (2.5,0) {$3$};
 \draw [fill=white,thick] (1,.5) circle [radius = .15];
 \node at (1,.5) {$5$};
 \draw [fill=white,thick] (2,.5) circle [radius = .15];
 \node at (2,.5) {$2$};
 \node [below right] at (2.65,0) {,};
\end{tikzpicture}
\end{center}
in which the middle column is empty.

By the above,  $\kring \beads (\mathbf{N}, \mathbf{c})$ is a $\sym_N \times \sym_c$-module. There is more structure, corresponding to one of the key ingredients of \cite{2021arXiv211001934P}:

\begin{prop}
\cite{2021arXiv211001934P}
\label{prop:beads}
For $N \in \nat$,  
\begin{enumerate}
\item 
$\kring \beads (\mathbf{N}, -) $ is a  functor in $\apsh {\gr\op}$, given by   $\fg_c \mapsto \kring \beads (\mathbf{N}, \mathbf{c})$;
\item 
$\kring \beads (\mathbf{N}, -) $ has polynomial degree $N$; 
\item  
under the equivalence between $\f_\omega (\gr \op)$ and $\flie$, $\kring \beads (\mathbf{N}, -) $ corresponds to the projective cover of $\kring \sym_N$, namely $\catlie (\mathbf{N}, -)$. 
\end{enumerate}
\end{prop}

\begin{proof}
(Indications.)
 The category $\gr$ is symmetric monoidal with respect to the free product $\ast$ of groups. As in \cite[Lemma A.1]{2021arXiv211001934P}, as a symmetric monoidal category, morphisms are generated by the following: 
 \begin{enumerate}
\item 
$m_1 :  \fg_1 \rightarrow \fg_0 = \{e \}$;
\item 
 $m_2 : \fg_1 \rightarrow \fg_2$ sending 
the generator of $\fg_1$ to $x_1 x_2$;
\item 
$m_3 : \fg_0= \{e\} \rightarrow \fg_1$;
\item 
$m_4 : \fg_1 \stackrel{(-)^{-1}}{\rightarrow } \fg_1$;
\item 
$m_5 : \fg_2 \rightarrow \fg_1$ given by the fold map given by $x_i \mapsto x_1$, for $i \in \mathbf{2}$.
\end{enumerate}

To specify the structure of   $\kring \beads (\mathbf{N}, -)$ as a functor of $\gr\op$, 
together with the action of the symmetric groups $\sym_c$, it suffices to give the actions of the above morphisms and their extensions using the symmetric monoidal structure of $\gr$. (One must also check that these satisfy the requisite relations.)

The action of $m_1$ and $m_3$ (together with their extensions obtained using the symmetric monoidal structure of $\gr\op$) is given on generators by the operations of insertion and deleting empty columns. The reader is left to describe these for themselves. 

It remains to consider the action of $m_2$, $m_4$ and $m_5$ (and their extensions). 
For clarity, we treat only the case of $m_2$, $m_4$ and $m_5$ and not their extensions. These act on generators via the following operations:
\begin{enumerate}
\item 
stacking columns: $m_2$ induces  $\kring \beads (\mathbf{N}, \mathbf{2}) \rightarrow \kring \beads (\mathbf{N}, \mathbf{1})$, acting on generators  as represented by:
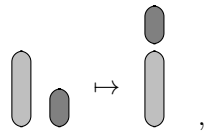
\begin{center}
 \begin{tikzpicture}[scale = 1]
\draw [rounded corners] [fill=lightgray] (0,0) rectangle (0.25,1); 
\draw [rounded corners] [fill=gray] (0.5,0) rectangle (0.75,.5); 
\node at (1.25,.5){$\mapsto$};
\draw [rounded corners] [fill=lightgray] (1.75,0) rectangle (2,1);
\draw [rounded corners] [fill=gray] (1.75,1.1) rectangle (2,1.6);
\node at (2.5,0) {,};
\end{tikzpicture}
\end{center}
where the shaded columns represent columns of beads; 
\item 
signed column flipping: $m_4$ induces $\kring \beads (\mathbf{N}, \mathbf{1}) \rightarrow \kring \beads (\mathbf{N}, \mathbf{1})$ defined by reversing the order of the column and multiplying by $(-1)^N$, eg.
\begin{center}
 \begin{tikzpicture}[scale = 1]
 \draw [fill=white,thick] (.5,0) circle[radius = .2];
 \node at (.5,0) {$1$}; 
  \draw [fill=white,thick] (.5,.5) circle[radius = .2];
 \node at (.5,.5) {$2$};
  \draw [fill=white,thick] (.5,1) circle[radius = .2];
 \node at (.5,1) {$3$};
\node at (1.5,.5) {$\mapsto$} ;
\node at (2.25,.5) {$-$};
 \draw [fill=white,thick] (3,0) circle[radius = .2];
 \node at (3,0) {$3$}; 
  \draw [fill=white,thick] (3,.5) circle[radius = .2];
 \node at (3,.5) {$2$};
  \draw [fill=white,thick] (3,1) circle[radius = .2];
 \node at (3,1) {$1$};
 \node [below right] at (3.25,0) {;};
\end{tikzpicture}
\end{center}
\item 
shuffle unstacking:  $m_5$ induces $\kring \beads (\mathbf{N}, \mathbf{1}) \rightarrow \kring \beads (\mathbf{N}, \mathbf{2})$, illustrated in the case $N=3$ by:
\begin{center}
 \begin{tikzpicture}[scale = 1]
 \draw [fill=white,thick] (.5,0) circle[radius = .2];
 \node at (.5,0) {$1$}; 
  \draw [fill=white,thick] (.5,.5) circle[radius = .2];
 \node at (.5,.5) {$2$};
  \draw [fill=white,thick] (.5,1) circle[radius = .2];
 \node at (.5,1) {$3$};
\node at (1.5,.5) {$\mapsto$} ;
 \draw [fill=white,thick] (2.5,0) circle[radius = .2];
 \node at (2.5,0) {$1$}; 
  \draw [fill=white,thick] (2.5,.5) circle[radius = .2];
 \node at (2.5,.5) {$2$};
  \draw [fill=white,thick] (2.5,1) circle[radius = .2];
 \node at (2.5,1) {$3$};
 \node at (3,0) {$\bullet$};
 \node at (3.5,.5) {$+$};
 \draw [fill=white,thick] (4,0) circle[radius = .2];
 \node at (4,0) {$1$}; 
  \draw [fill=white,thick] (4,.5) circle[radius = .2];
 \node at (4,.5) {$2$};
 \draw [fill=white,thick] (4.5,0) circle[radius = .2];
 \node at (4.5,0) {$3$};
   \node at (5,.5) {$+$};
   \draw [fill=white,thick] (5.5,0) circle[radius = .2];
 \node at (5.5,0) {$2$}; 
  \draw [fill=white,thick] (5.5,.5) circle[radius = .2];
 \node at (5.5,.5) {$3$};
 \draw [fill=white,thick] (6,0) circle[radius = .2];
 \node at (6,0) {$1$};
\node at (6.5,.5) {$+$};
 \draw [fill=white,thick] (7,0) circle[radius = .2];
 \node at (7,0) {$1$}; 
  \draw [fill=white,thick] (7,.5) circle[radius = .2];
 \node at (7,.5) {$3$};
 \draw [fill=white,thick] (7.5,0) circle[radius = .2];
 \node at (7.5,0) {$2$};
 \node at (8,.5) {$+$};
 \draw [fill=white,thick] (8.5,0) circle[radius = .2];
 \node at (8.5,0) {$1$}; 
  \draw [fill=white,thick] (9,.5) circle[radius = .2];
 \node at (9,.5) {$3$};
 \draw [fill=white,thick] (9,0) circle[radius = .2];
 \node at (9,0) {$2$};
 \node at (9.5,.5) {$+$}; 
 \draw [fill=white,thick] (10,0) circle[radius = .2];
 \node at (10,0) {$2$}; 
  \draw [fill=white,thick] (10.5,.5) circle[radius = .2];
 \node at (10.5,.5) {$3$};
 \draw [fill=white,thick] (10.5,0) circle[radius = .2];
 \node at (10.5,0) {$1$};
 \node at (11,.5) {$+$}; 
  \draw [fill=white,thick] (11.5,0) circle[radius = .2];
 \node at (11.5,0) {$3$}; 
  \draw [fill=white,thick] (12,.5) circle[radius = .2];
 \node at (12,.5) {$2$};
 \draw [fill=white,thick] (12,0) circle[radius = .2];
 \node at (12,0) {$1$};
 \node at (12.5,.5) {$+$};
 \node at (13,0) {$\bullet$};
  \draw [fill=white,thick] (13.5,0) circle[radius = .2];
 \node at (13.5,0) {$1$}; 
  \draw [fill=white,thick] (13.5,.5) circle[radius = .2];
 \node at (13.5,.5) {$2$};
  \draw [fill=white,thick] (13.5,1) circle[radius = .2];
  \node at (13.5,1) {$3$};
  \node [below right] at (14,0) {,};
\end{tikzpicture}
\end{center}
in which an empty column is indicated by $\bullet$.
\end{enumerate}
We claim that the above make $\fg_c \mapsto \kring \beads (\mathbf{N}, \mathbf{c})$ into a functor on $\gr\op$; this can be checked directly. 

\begin{rem}
This structure can also be deduced from  \cite[Proposition 9.13]{2021arXiv211001934P}, using that $\kring \beads (\mathbf{N}, \mathbf{c})$ is isomorphic as a $\sym_N \times \sym_c$-module to  $\cat \uass (\mathbf{N}, \mathbf{c})$, where $\cat \uass$ is the $\kring$-linear category associated to the operad $\uass$ for unital, associative algebras. The structure outlined above corresponds  to that exhibited in \cite[Appendix A]{2021arXiv211001934P}.
\end{rem}

The fact that $\kring \beads (\mathbf{N}, -)$ is polynomial of degree $N$ is a consequence of the following observation: if $c >N$, then any element of $\beads (\mathbf{N}, \mathbf{c})$ has at least one empty column. 

The final statement follows from  the explicit equivalence between $\f_\omega (\gr\op) $ and $\flie$ given in \cite{2021arXiv211001934P}, as follows. Under this equivalence $\kring \sym_N$ in $\flie$ corresponds to the functor $(\mathfrak{a}^\sharp)^{\otimes N}$ in $\f_\omega (\gr\op)$, where $\mathfrak{a}^\sharp$ is the functor $G \mapsto \hom (G, \kring)$, for $G \in \ob \gr\op$. Then the equivalence sends the projective cover $\cat \lie (\mathbf{N}, -)$ of $\kring \sym_N$ in $\flie$ to the projective cover of  $(\mathfrak{a}^\sharp)^{\otimes N}$ in $\f_\omega (\gr\op)$.  This projective cover is the functor $\cat \uass (\mathbf{N}, -)$, which identifies with $\kring \beads (\mathbf{N}, -)$  as indicated above.
\end{proof}

The functor $\kring \beads (\mathbf{N}, -) $ does not belong to $\foutan$ for $N>1$, i.e., inner automorphisms of $\fg_c$ need not act trivially on $\kring \beads (\mathbf{N}, \mathbf{c}) $. To obtain the algebraic model of Turchin and Willwacher's type I bead representations, we  apply the functor $\Omega : \f_\omega (\gr\op) \rightarrow \foutan$. The resulting functor $\Omega \kring \beads (\mathbf{N}, -) $ is our algebraic model for the type I bead representations on $N$ beads (this is justified in \cite[Remark 16.17]{2018arXiv180207574P}). 

By the description given in \cite{2018arXiv180207574P}, the values of  $\Omega \kring \beads (\mathbf{N}, -) $ have the following concrete description: for $c \in \nat$, there is an exact sequence 
\[
\bigoplus _{\substack{X \subset \mathbf{N} \\ |X| = N-1 }}
\kring \beads (X, \mathbf{c}) 
\rightarrow 
\kring \beads (\mathbf{N}, \mathbf{c}) 
\rightarrow 
\Omega \kring \beads (\mathbf{N}, -) (\mathbf{c})
\rightarrow 
0. 
\]
Here $\beads (X, \mathbf{c})$ is defined as above, with beads labelled by the elements of $X \subset \mathbf{N}$. The first map is defined as a sum of the conjugation actions on columns. This is  illustrated in the case $c=1$ and $X= \mathbf{N-1} \subset \mathbf{N}$ by the map
$\kring \beads (\mathbf{N-1}, \mathbf{1}) 
\rightarrow \kring \beads (\mathbf{N}, \mathbf{1})$ given on generators by
\begin{center}
 \begin{tikzpicture}[scale = 1]
\draw [rounded corners] [fill=lightgray] (0,0) rectangle (0.35,1); 
\node at (1,.5){$\mapsto$};
\draw [rounded corners] [fill=lightgray] (1.65,0) rectangle (2,1);
\draw [fill=white,thick] (1.825,1.25) circle[radius = .2];
 \node at (1.825,1.25) {$N$}; 
\node at (2.5,.5){$-$};
\draw [rounded corners] [fill=lightgray] (2.95,.4) rectangle (3.3,1.4);
\draw [fill=white,thick] (3.125,0.15) circle[radius = .2];
 \node at (3.125,.15) {$N$}; 
 \node  at (3.4,0) {,};
\end{tikzpicture}
\end{center}
where the shaded column represents an arrangement of beads labelled by $\mathbf{N-1}$.

The analogue of Proposition \ref{prop:beads} is the following, which gives our intrinsic  interpretation of the type I bead representations of Turchin and Willwacher. 

\begin{prop}
\label{prop:TW_reps}
For $N \in \nat$, under the equivalence between $\f_\omega (\gr \op)$ and $\flie$, $\Omega \kring \beads (\mathbf{N}, -)$  corresponds to the projective cover of $\kring \sym_N$ in $\flie^\mu$, which is given by $\catlie (\mathbf{N}, -)^\mu$. 
\end{prop}

\begin{proof}
This is a consequence of the main result of \cite{2022arXiv220113307P} applied in conjunction with Proposition \ref{prop:beads}.
\end{proof}

\begin{rem}
\label{rem:gadish-hainaut}
The Turchin-Willwacher bead representations also feature prominently in the work of Gadish and Hainaut \cite{2022arXiv220212494G}. Those authors have also developed techniques for calculating these representations, including implementing computer calculations.  
\end{rem}

\subsection{Relating to baby beads}
\label{subsect:relate_bbd}

The {\em baby beads} considered in this paper (see Section \ref{sect:exam_beads}) are  related to the bead arrangements $\beads (-,-)$ above. 

\begin{rem}
\label{rem:forewarn}
The reader is forewarned that this relationship is subtle. In particular, $\kring \beads (\mathbf{N}, -)$ lives in $\foutan$ whereas the baby beads $\kring \bbd (\mathbf{N}, -)$ introduced below give a first approximation to an object of $\flie$. In particular, the adjective {\em baby} has two simultaneous, different meanings: {\em baby} in the sense of limiting the number of beads in any column and, via the synonym {\em toy}, as an approximation to an object of $\flie$. 
\end{rem}

Explicitly (see Definition \ref{defn:bbd}), for $c, N \in \nat$ we consider the sub $\sym_N \times \sym_c$-set:
\[
\bbd (\mathbf{N}, \mathbf{c}) \subset \beads (\mathbf{N}, \mathbf{c})
\]
 given by imposing the condition that each column contains either one or two beads. On passage to the $\kring$-linearization, this should be viewed as giving  the {\em truncation}:
 \[
 \kring 
\beads (\mathbf{N}, \mathbf{c}) \twoheadrightarrow  \kring \bbd (\mathbf{N}, \mathbf{c}), 
\]
such that an arrangement containing an empty column or a column with at least three beads is sent to zero. The above inclusion induces a section to this truncation.
 
There is an important difference: $\kring \bbd (\mathbf{N}, -) $ is {\em not} given as   a functor on $\gr\op$. Rather, it is a `lift' of an object of $\flie$. To be explicit, it is a lift of $\bba (\mathbf{N}, -)$ of Section \ref{sect:exam_anti}, which can be equipped with a natural $\catlie$-module structure; the latter corresponds to $\cattlie (\mathbf{N}, -)$, considered as a $\catlie$-module. Remark \ref{rem:interpret_schur}, which interprets this in terms of the associated Schur functors, makes this more explicit.

This objects are related via the diagram 
\[
\xymatrix{
&
\kring \bbd (\mathbf{N}, -) 
\ar@{.>>}[d]
\\
\catlie (\mathbf{N}, -) 
\ar@{->>}[r]
&
\cattlie (\mathbf{N},-),
}
\]
in which:
\begin{enumerate}
\item 
the horizontal surjection, corresponding to the truncation, is a morphism of $\flie$;\item
the dotted vertical one {\em isn't} a morphism of $\flie$ ($\kring \bbd (\mathbf{N}, -)$ is not even given a $\catlie$-module structure);
\item  
via the equivalence between $\flie$ and $\f_\omega (\gr\op)$, $\catlie (\mathbf{N}, -) \in \ob \flie$ corresponds to $\kring \beads (\mathbf{N}, -) \in \ob \f_\omega (\gr\op)$ (see  Proposition \ref{prop:beads}).
\end{enumerate}
In particular, $\kring \beads (\mathbf{N}, -)$ and $\kring \bbd (\mathbf{N}, -)$ live in different worlds. This reinforces the warning of Remark \ref{rem:forewarn}.

\bigskip
Next one seeks the appropriate analogue of $\Omega \kring  \beads (\mathbf{N}, -)$ when 
 working with the {\em baby beads} $\kring \bbd (\mathbf{N},-)$. For this, the  `conjugation action' used in constructing $\Omega \kring  \beads (\mathbf{N}, -)$ (for $X=\mathbf{N-1}$) is replaced  by the map
\[
\kring \bbd (\mathbf{N-1}, \mathbf{c}) 
\rightarrow 
\kring \bbd (\mathbf{N}, \mathbf{c}) 
\]
induced by the `one-sided' operation given by adding beads to the top of columns. Illustrated as above for a single column, this is given by:
 
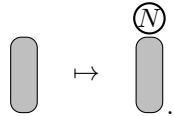
\begin{center}
 \begin{tikzpicture}[scale = 1]
\draw [rounded corners] [fill=lightgray] (0,0) rectangle (0.35,1); 
\node at (1,.5){$\mapsto$};
\draw [rounded corners]  [fill=lightgray] (1.65,0) rectangle (2,1);
\draw [fill=white,thick] (1.825,1.25) circle[radius = .2];
 \node at (1.825,1.25) {$N$}; 
\node  at (2.1,0) {.};
\end{tikzpicture}
\end{center}
Here, the gray column either contains one or two beads. Due to the {\em truncation}, this operation is only non-trivial if the gray column contains exactly one bead. 

On passing from $\kring \bbd (\mathbf{N},-)$ to $\cattlie (\mathbf{N},-)$, up to truncation, this gives the {\em infinitesimal} version of the conjugation action on $\kring \beads (\mathbf{N}, -)$. The latter corresponds to the structure that is used in \cite{2022arXiv220113307P} in the construction of $(-)^\mu : \flie \rightarrow \flie^\mu$. This explains why $\kring \bbd (\mathbf{N},-)$ equipped with the one-sided action can be taken as a  {\em toy model} for  $\cattlie (\mathbf{N},-)$. 

\begin{rem}
\label{rem:interpret_schur}
 Interpreted in terms of Schur functors, the above relationship corresponds to the commutative diagram (\ref{eqn:cx_non_anti}) given in the Introduction. 

For this, one fixes $c \in \nat$ and considers the Schur functors associated to $\mathbf{N} \mapsto \kring \bbd (\mathbf{N}, \mathbf{c})$ and $\mathbf{N} \mapsto \cattlie (\mathbf{N}, \mathbf{c})$. The surjection $\kring \bbd ( - , \mathbf{c}) 
\twoheadrightarrow \cattlie (-, \mathbf{c})$ induces the surjection of Schur functors 
\[
(V \oplus V^{\otimes 2})^{\otimes c} 
\twoheadrightarrow 
(V \oplus \lie_2 (V))^{\otimes c} = \tlie(V)^{\otimes c}
\]
(see Section \ref{sect:schur}). This is the right hand vertical map of diagram (\ref{eqn:cx_non_anti}). The horizontal maps of that diagram correspond to the actions. 
\end{rem}

\section{The behaviour of truncation on certain isotypical components} 
\label{sect:detrunc}

Throughout this Section, $\kring$ is taken to be a field of characteristic zero and $n$ is a  fixed positive integer. We work systematically with Schur functors: i.e.,  all constructions are functorial with respect to the $\kring$-vector space $V$. 

The truncation $\lie(V) \twoheadrightarrow \lie_{\leq 2} (V)$ induces a natural surjection:
\[
 H_0 (\lie (V); \lie (V)^{\otimes n})
\twoheadrightarrow 
 H_0 (\lie (V); \lie_{\leq 2}(V)^{\otimes n})
\]
of $\sym_n$-modules. This is not in general an isomorphism; however, on certain isotypical components it is.  For example: 

\begin{prop}
\label{prop:1^N}
For $N \in \nat$ and $\rho = (1^N)$, so that $S^\rho (V) = \Lambda^N (V)$,
\[
\nt \big(\Lambda^N(V) , H_0 (\lie (V); \lie (V)^{\otimes n}) \big)
\twoheadrightarrow 
\nt \big(\Lambda^N(V) , H_0 (\lie (V); \lie_{\leq 2}(V)^{\otimes n}) \big)
\]
is an isomorphism; i.e., $ H_0 (\lie (V); \lie_{\leq 2}(V)^{\otimes n})$ calculates the  $(1^N)$-isotypical component of $H_0 (\lie (V); \lie (V)^{\otimes n})$.
\end{prop}

\begin{proof}
(Indications.) 
This follows from the fact that $\sgn_t$ is not a composition factor of $\lie_t$ for $t \geq 3$ (see Proposition \ref{prop:lie_reps}).

Explicitly, this implies that applying $\nt (\Lambda^N(V), -)$ to the morphism of complexes given by truncation (represented by the vertical maps in the following commutative diagram):
\[
\xymatrix{
V \otimes \lie(V) ^{\otimes n}
\ar@{->>}[d]
\ar[r]
&
\lie(V) ^{\otimes n}
\ar@{->>}[d]
\\
V \otimes \tlie(V)^{\otimes n} 
\ar[r]
&
\tlie(V) ^{\otimes n},
}
\]
gives an isomorphism of complexes.
\end{proof}

\begin{rem}
\label{rem:isotyp_N}
This argument also applies to the case $\rho = (N)$, when the associated Schur functor is the $N$th symmetric power functor $S^N(V)$. The result in this case is much less striking, since $\nt \big(S^N(V) , H_0 (\lie (V); \lie (V)^{\otimes n}) \big)$ is zero unless $n=N$, when it yields the trivial representation $\triv_N$. 

This contrasts with the case $\rho = (1^N)$, where the result is far more interesting (see Theorem \ref{thm:isotyp_1N}). 
\end{rem}

The main result of this Section is  the following:

\begin{thm}
\label{thm:2_1_N-2_isotypical}
For $N\geq 3$ and $\rho = ( 2,1^{N-2})$, 
\[
\nt \big(S^{(2,1^{N-2})}(V) , H_0 (\lie (V); \lie (V)^{\otimes n}) \big)
\twoheadrightarrow 
\nt \big(S^{(2,1^{N-2})}(V) , H_0 (\lie (V); \lie_{\leq 2}(V)^{\otimes n}) \big)
\]
is an isomorphism; i.e.,  $ H_0 (\lie (V); \lie_{\leq 2}(V)^{\otimes n})$ calculates the  $(2, 1^{N-2})$-isotypical component of $H_0 (\lie (V); \lie (V)^{\otimes n})$.
\end{thm}

\begin{rem}
Contrary to the case $(1^N)$ of Proposition \ref{prop:1^N}, the truncation map  does not induce an isomorphism on the $( 2,1^{N-2})$-isotypical component at the level of complexes: for instance, $\lie_3$ is isomorphic to $S^{(2,1)}$ and this must be taken into account.
\end{rem}

\begin{rem}
The argument used to prove Theorem \ref{thm:2_1_N-2_isotypical} can also be applied to the case $\rho =(N-1,1)$. Once again (cf. Remark \ref{rem:isotyp_N}), this is of less interest, since this case can be handled directly as in \cite{2018arXiv180207574P}, where it is shown that 
$\nt \big(S^{(N-1,1)}(V) , H_0 (\lie (V); \lie (V)^{\otimes n}) \big)$
is zero unless $n=N$, when it yields the representation $S^{(N-1,1)}$.
\end{rem}

\subsection{An intermediate $\lie(V)$-module}

The proof of Theorem \ref{thm:2_1_N-2_isotypical} proceeds by considering a natural quotient $\lie(V)$-module of $\lie(V)^{\otimes n}$ that is slightly larger than $\lie_{\leq 2} (V)^{\otimes n}$. This is introduced in this subsection.

\begin{nota}
Denote by $Y$ the $\fb$-module given by $Y(\mathbf{t}) = 0$ for $t \leq 2$ and $S^{(2,1^{t-2})}$ for $t \geq 3$. The associated Schur functor is denoted $V \mapsto Y(V)$. 
\end{nota}

\begin{lem}
\label{lem:y}
There are natural surjections of $\lie(V)$-modules
\[
\lie (V) 
\twoheadrightarrow 
\lie_{\leq 2} (V) \oplus Y(V) 
\twoheadrightarrow 
\lie_{\leq 2}(V),
\]
where, as a functor, $\lie_{\leq 2}(V) \cong \Lambda^1 (V) \oplus \Lambda^2 (V)$, and the $\lie(V)$-module structure on $\lie_{\leq 2} (V) \oplus Y(V) $ is induced by that of $\lie (V)$ with the adjoint action.

The composition factors of the kernel of $\lie (V) 
\twoheadrightarrow 
\lie_{\leq 2} (V) \oplus Y(V) $ are of the form $S^\mu (V)$, where $\mu$ is a partition such that either $\mu_1 >2$ or $\mu_1 =2 =\mu_2$.

The subfunctor $Y(V) \subset \lie_{\leq 2} (V) \oplus Y(V) $ is a sub $\lie(V)$-module; the restriction of the  Lie action $\lie_1(V) \otimes Y(V) \rightarrow Y(V)$ identifies in polynomial degree $t+1$ (with $t \geq 3$)  as the (unique up to non-zero scalar multiple) natural surjection:
\[
V \otimes S^{(2,1^{t-2})}(V) 
\twoheadrightarrow 
S^{(2,1^{t-1})}(V).
\]
In particular $H_0 (\lie (V); Y(V) ) \cong S^{(2,1)}(V)$.
\end{lem}

\begin{proof}
This is a straightforward consequence of Proposition \ref{prop:lie_reps} and the methods indicated in its proof.
\end{proof}

On forming the $n$-fold tensor product, considered as $\lie(V)$-modules, this gives the natural, $\sym_n$-equivariant surjections:
\[
\lie (V) ^{\otimes n}
\twoheadrightarrow 
\big(\lie_{\leq 2} (V) \oplus Y(V)\big)^{\otimes n} 
\twoheadrightarrow 
\lie_{\leq 2}(V)^{\otimes n}.
\]
However, we are interested in working with a smaller quotient than $\big(\lie_{\leq 2} (V) \oplus Y(V)\big)^{\otimes n}$, obtained by discarding terms that contain more than one tensor factor from $Y(V)$:

\begin{defn}
For $1 \leq n \in \nat$, let $\ylin (n)(V)$ be the quotient:
\[
\big(\lie_{\leq 2} (V) \oplus Y(V)\big)^{\otimes n} / \big( \bigoplus_{i>1} Y(V)^{\otimes i} \otimes \lie_{\leq 2} (V)^{\otimes n-i} ) \uparrow _{\sym_i \times \sym_{n-i}}^{\sym_n} \big).
\]
\end{defn}

The following is clear from the construction:

\begin{lem}
\label{lem:structure_ylin}
The association  $V \mapsto \ylin (n) (V)$ is functorial and $\ylin (n) (V)$ is a quotient $\lie (V)$-module of $\lie (V)^{\otimes n}$ so that there are surjections of $\lie(V)$-modules 
\[
\lie (V)^{\otimes n}
\twoheadrightarrow 
\ylin (n) (V)
\twoheadrightarrow 
\lie_{\leq 2} (V)^{\otimes n}.
\]
The group $\sym_n$ acts naturally on $\ylin (n) (V)$  in $\lie (V)$-modules and the surjections are $\sym_n$-equivariant.

Moreover, $\ylin(n)(V)$ occurs in the $\sym_n$-equivariant short exact sequence of $\lie(V)$-modules:
\begin{eqnarray}
\label{eqn:ses_ylin}
0 \rightarrow 
(Y(V) \otimes \lie_{\leq 2}(V)^{\otimes n-1}) \uparrow_{\sym_{n-1}}^{\sym_n}
\rightarrow 
\ylin(n)(V) 
\rightarrow 
\lie_{\leq 2} (V)^{\otimes n}
\rightarrow 
0.
\end{eqnarray}
\end{lem}

The significance of $\ylin(n) (V)$ here is explained by the following:

\begin{lem}
\label{lem:ylin_isotypical}
For $1 \leq n \in \nat$ and $3 \leq N \in \nat$, the surjection $\lie (V)^{\otimes n}
\twoheadrightarrow 
\ylin (n) (V)$ induces an isomorphism of $\kring \sym_n$-modules:
\[
\nt (S^{(2,1^{N-2})} (V), \lie(V)^{\otimes n}) 
\stackrel{\cong}{\rightarrow} 
 \nt (S^{(2,1^{N-2})} (V),\ylin (n) (V)).
 \]
\end{lem}

\begin{proof}
That one obtains a $\sym_n$-equivariant map follows from the equivariance statement of Lemma \ref{lem:structure_ylin}. Hence it suffices to show that the underlying map is an isomorphism. Equivalently, writing $K(n) (V)$ for the kernel of $\lie (V)^{\otimes n}
\twoheadrightarrow 
\ylin (n) (V)$, we require to show that 
\[
 \nt (S^{(2,1^{N-2})} (V),K (n) (V))=0.
\]

Here, it suffices to restrict to the homogeneous component of $K(n) (V)$ of polynomial degree $N$ as a functor of $V$. This is a finite direct sum of functors of the form 
\[
S^{\lambda(1)} (V) \otimes \ldots \otimes S^{\lambda (n)}(V)
\]
for partitions $\lambda (1), \ldots, \lambda (n)$ such that $\sum_i |\lambda (i) | =N$. Here,  by construction of $Y(V)$ and  $\ylin(n) (V)$, at least one of the following holds:
\begin{enumerate}
\item 
 there exists $i$ such that $\lambda (i)_1 >2$ or $\lambda(i)_1=2= \lambda(i)_2$ (see Lemma \ref{lem:y}); 
\item 
there exists $i< j$ such that  $S^{\lambda(i) }(V)$ and $S^{\lambda (j)}(V)$ are both composition factors of $Y(V)$. 
\end{enumerate}

A straightforward application of the Littlewood-Richardson rule then shows that 
\[
\nt (S^{(2,1^{N-2})} (V), S^{\lambda(1)} (V) \otimes \ldots \otimes S^{\lambda (n)}(V))=0,
\]
as required.
\end{proof}

\subsection{Proof of Theorem \ref{thm:2_1_N-2_isotypical}}

By Lemma \ref{lem:ylin_isotypical},  Theorem \ref{thm:2_1_N-2_isotypical} is a consequence of:

\begin{prop}
\label{prop:reduction_ylin}
For $1 \leq n \in \nat$ and $3 \leq N \in \nat$, the surjection $\ylin(n) (V) \twoheadrightarrow \lie_{\leq 2} (V)^{\otimes n}$ induces an isomorphism
\[
\nt(S^{(2,1^{N-2})}(V) , H_0 (\lie (V); \ylin(n) (V))) 
\cong 
\nt(S^{(2,1^{N-2})}(V) ,H_0 (\lie (V); \lie_{\leq 2} (V)^{\otimes n}))
\]
of $\sym_n$-modules.
\end{prop}

The short exact sequence (\ref{eqn:ses_ylin}) induces an exact sequence 
\[
H_0 (\lie (V); (Y(V) \otimes \lie_{\leq 2}(V)^{\otimes n-1}) \uparrow_{\sym_{n-1}}^{\sym_n})
\rightarrow 
H_0 (\lie (V); \ylin(n)(V) ) 
\rightarrow 
H_0 (\lie(V); \lie_{\leq 2} (V)^{\otimes n})
\rightarrow 
0.
\]
Hence Proposition \ref{prop:reduction_ylin} follows from:

\begin{prop}
\label{prop:further_reduction_ylin}
For $1 \leq n \in \nat$ and $3 \leq N \in \nat$, the map 
\[
\nt (S^{(2,1^{N-2})} (V), H_0 (\lie (V); (Y(V) \otimes \lie_{\leq 2}(V)^{\otimes n-1}) \uparrow_{\sym_{n-1}}^{\sym_n})
\rightarrow 
\nt (S^{(2,1^{N-2})} (V),H_0 (\lie (V); \ylin(n)(V) ) )
\]
is zero.
\end{prop}

The remainder of this subsection is devoted to the proof of this Proposition.

There is an inclusion of Schur functors $S^{(2,1)} (V)\hookrightarrow Y(V)$ (not the  inclusion of a sub $\lie(V)$-module). This induces an inclusion 
\begin{eqnarray}
\label{eqn:inclusion_2,1}
(S^{(2,1)} (V) \otimes \lie_{\leq 2}(V)^{\otimes n-1}) \uparrow_{\sym_{n-1}}^{\sym_n}
\hookrightarrow 
(Y(V) \otimes \lie_{\leq 2}(V)^{\otimes n-1}) \uparrow_{\sym_{n-1}}^{\sym_n}.
\end{eqnarray}

\begin{lem}
\label{lem:restrict_2,1}
The inclusion (\ref{eqn:inclusion_2,1}) induces a surjection
\[
(S^{(2,1)} (V) \otimes \lie_{\leq 2}(V)^{\otimes n-1}) \uparrow_{\sym_{n-1}}^{\sym_n}
\twoheadrightarrow 
H_0 (\lie (V); (Y(V) \otimes \lie_{\leq 2}(V)^{\otimes n-1}) \uparrow_{\sym_{n-1}}^{\sym_n}).
\]
\end{lem}

\begin{proof}
Let $Y_{\leq t} (V)$ by the sub functor $\bigoplus_{i+2 \leq t} S^{(2,1^i)} (V)$ of $Y(V)$. Now, for $t \geq \dim (V)$, $Y_{\leq t} (V) = Y(V)$; so for fixed $V$, this gives a finite filtration of $Y(V)$.

The result is proved by showing by descending induction on $t$ that, for all $t \geq 3$, the inclusion $Y_{\leq t} (V) \subset Y(V) $ induces a surjection 
\[
(Y_{\leq t}  (V) \otimes \lie_{\leq 2}(V)^{\otimes n-1}) \uparrow_{\sym_{n-1}}^{\sym_n}
\twoheadrightarrow 
H_0 (\lie (V); (Y(V) \otimes \lie_{\leq 2}(V)^{\otimes n-1}) \uparrow_{\sym_{n-1}}^{\sym_n}),
\]
the case $t=3$ being the required result.

The initial case is given by taking $t= \dim V$, using that  $Y_{\leq t} (V) = Y(V)$ for $t \geq \dim (V)$, as above. For the inductive step with $t>3$, consider the (split) short exact sequence of functors:
\[
0
\rightarrow 
(Y_{\leq t-1}  (V) \otimes \lie_{\leq 2}(V)^{\otimes n-1}) \uparrow_{\sym_{n-1}}^{\sym_n}
\rightarrow 
(Y_{\leq t}  (V) \otimes \lie_{\leq 2}(V)^{\otimes n-1}) \uparrow_{\sym_{n-1}}^{\sym_n}
\rightarrow 
(S^{(2,1^{t-2})} (V) \otimes \lie_{\leq 2}(V)^{\otimes n-1}) \uparrow_{\sym_{n-1}}^{\sym_n}
\rightarrow 
0
\] 
given by the definition of $Y_{\leq t} (V)$. 

We require to show that an element in the image of $(S^{(2,1^{t-2})} (V) \otimes \lie_{\leq 2}(V)^{\otimes n-1}) \uparrow_{\sym_{n-1}}^{\sym_n}
\rightarrow 
H_0 (\lie (V); (Y(V) \otimes \lie_{\leq 2}(V)^{\otimes n-1}) \uparrow_{\sym_{n-1}}^{\sym_n})
$ lies in the image of $ (Y_{\leq t-1}  (V) \otimes \lie_{\leq 2}(V)^{\otimes n-1}) \uparrow_{\sym_{n-1}}^{\sym_n}$.

The action of $\lie (V)$ on $Y(V)$ restricts to 
\[
\lie_1 (V) \otimes Y_{\leq t-1} (V) \rightarrow Y_{\leq t} (V) = Y_{\leq t-1} (V) \oplus S^{(2,1^{t-2})}(V),
\]
by Lemma \ref{lem:y}, such that the composition with the projection to $ S^{(2,1^{t-2})}(V)$ is surjective (recall that $t>3$).

It follows that the Lie action induces:
\[
\lie_1 (V) \otimes (Y_{\leq t-1}  (V) \otimes \lie_{\leq 2}(V)^{\otimes n-1}) \uparrow_{\sym_{n-1}}^{\sym_n}
\rightarrow 
(Y_{\leq t}  (V) \otimes \lie_{\leq 2}(V)^{\otimes n-1}) \uparrow_{\sym_{n-1}}^{\sym_n}
\]
such that the composite with the projection 
\begin{eqnarray*}
(Y_{\leq t}  (V) \otimes \lie_{\leq 2}(V)^{\otimes n-1}) \uparrow_{\sym_{n-1}}^{\sym_n}
\cong 
(Y_{\leq t-1}  (V) \otimes \lie_{\leq 2}(V)^{\otimes n-1}) \uparrow_{\sym_{n-1}}^{\sym_n}
\oplus 
(S^{(2,1^{t-2})} (V) \otimes \lie_{\leq 2}(V)^{\otimes n-1}) \uparrow_{\sym_{n-1}}^{\sym_n}
\\
\twoheadrightarrow 
(S^{(2,1^{t-2})} (V) \otimes \lie_{\leq 2}(V)^{\otimes n-1}) \uparrow_{\sym_{n-1}}^{\sym_n}
\end{eqnarray*}
 is surjective. 
 
This shows that, modulo decomposables for the Lie action, an element in the image of $(S^{(2,1^{t-2})} (V) \otimes \lie_{\leq 2}(V)^{\otimes n-1}) \uparrow_{\sym_{n-1}}^{\sym_n}$ 
lies in the image of  $ (Y_{\leq t-1}  (V) \otimes \lie_{\leq 2}(V)^{\otimes n-1}) \uparrow_{\sym_{n-1}}^{\sym_n}$, as required, thus completing the proof of the inductive step. 
\end{proof}

By Lemma \ref{lem:restrict_2,1}, in order to prove Proposition \ref{prop:further_reduction_ylin}, it suffices to prove that, on applying $\nt (S^{(2,1^{N-2})}(V),-)$ to the composite
\[
(S^{(2,1)} (V) \otimes \lie_{\leq 2}(V)^{\otimes n-1}) \uparrow_{\sym_{n-1}}^{\sym_n}
\twoheadrightarrow 
H_0 (\lie (V); (Y(V) \otimes \lie_{\leq 2}(V)^{\otimes n-1}) \uparrow_{\sym_{n-1}}^{\sym_n})
\rightarrow 
H_0 (\lie (V); \ylin(n)(V) ), 
\]
the resulting  map 
\begin{eqnarray*}
\nt (S^{(2,1^{N-2})}(V), (S^{(2,1)} (V) \otimes \lie_{\leq 2}(V)^{\otimes n-1}) \uparrow_{\sym_{n-1}}^{\sym_n})
&\cong &
\nt (S^{(2,1^{N-2})}(V), S^{(2,1)} (V) \otimes \lie_{\leq 2}(V)^{\otimes n-1}) \uparrow_{\sym_{n-1}}^{\sym_n}
\\
 &\rightarrow &
\nt (S^{(2,1^{N-2})}(V), H_0 (\lie (V); \ylin(n)(V) )) \quad
\end{eqnarray*}
is zero.

Since this is $\sym_n$-equivariant, the induction/restriction adjunction implies that it suffices to show that the restriction:
\begin{eqnarray}
\label{eqn:reduction}
\quad \quad 
\nt (S^{(2,1^{N-2})}(V), S^{(2,1)} (V) \otimes \lie_{\leq 2}(V)^{\otimes n-1})
 \rightarrow 
\nt (S^{(2,1^{N-2})}(V), H_0 (\lie (V); \ylin(n)(V) ))
\end{eqnarray}
is zero.

Furthermore, since $\lie_{\leq 2} (V) \cong \Lambda^1 (V) \oplus \Lambda^2 (V)$ as a functor, there is a $\sym_{n-1}$-equivariant isomorphism:
\[
\lie_{\leq 2}(V)^{\otimes n-1}
\cong 
\bigoplus_{0 \leq i \leq n-1} (\Lambda^1 (V)^{\otimes i} \otimes \Lambda^2(V)^{ \otimes n-1 -i} ) \uparrow_{\sym_i \times \sym_{n-1-i}}^{\sym_{n-1}}.
\]
The summand indexed by $i$ is homogeneous polynomial of degree $2n-2 -i$. Again by exploiting the induction/ restriction adjunction, one reduces to considering the map (\ref{eqn:reduction})  restricted to 
\[
\nt (S^{(2,1^{N-2})}(V), S^{(2,1)} (V) \otimes \Lambda^1 (V)^{\otimes i} \otimes \Lambda^2(V)^{ \otimes n-1 -i})
\subset 
\nt (S^{(2,1^{N-2})}(V), S^{(2,1)} (V) \otimes \lie_{\leq 2}(V)^{\otimes n-1})
\]
where $N= 3 + (2n-2 -i)$ with $0 \leq i \leq n-1$, so that $i= 2n+1-N$. (If no such $i$ exists, then there is nothing to prove.)

\begin{lem}
\label{lem:phi}
Suppose that $0 \leq i \leq n-1$ with $i = 2n+1 -N$, then 
\[
\nt (S^{(2,1^{N-2})}(V), S^{(2,1)} (V) \otimes \Lambda^1 (V)^{\otimes i} \otimes \Lambda^2(V)^{ \otimes n-1 -i}) \cong \kring.
\]
\end{lem}

\begin{proof}
(Sketch.) 
By Pieri's rule, one has $\Lambda^{N-1} (V) \otimes \Lambda^1 (V) \cong \Lambda^N (V) \oplus S^{(2, 1^{N-2})}(V)$. Since $\nt (\Lambda^N (V) , S^{(2,1)} (V) \otimes \Lambda^1 (V)^{\otimes i} \otimes \Lambda^2(V)^{ \otimes n-1 -i})=0$ (as can be seen, for example, by using the Littlewood-Richardson rule), it suffices to show that 
\[
\nt (\Lambda^{N-1} (V) \otimes \Lambda^1 (V), S^{(2,1)} (V) \otimes \Lambda^1 (V)^{\otimes i} \otimes \Lambda^2(V)^{ \otimes n-1 -i}) \cong \kring.
\] 
This is shown by standard techniques, for example by using the fact that the exterior power functors are exponential functors.
\end{proof}

\begin{rem}
Lemma \ref{lem:phi} can also be proved directly  using the Littlewood-Richardson rule. 
\end{rem}

We construct a generator $0 \neq \phi \in \nt (S^{(2,1^{N-2})}(V), S^{(2,1)} (V) \otimes \Lambda^1 (V)^{\otimes i} \otimes \Lambda^2(V)^{ \otimes n-1 -i}) $  as follows. As in  Lemma \ref{lem:phi} and its proof, it suffices to exhibit a non-zero element of $\nt (\Lambda^{N-1} (V) \otimes \Lambda^1 (V), S^{(2,1)} (V) \otimes \Lambda^1 (V)^{\otimes i} \otimes \Lambda^2(V)^{ \otimes n-1 -i}) $. For this, we form the composite
\[
\Lambda^{N-1}(V) \otimes \Lambda^1 (V) \rightarrow \Lambda ^{N-3}(V) \otimes \Lambda ^2 (V)
\otimes \Lambda^1 (V) 
\cong (\Lambda ^2 (V)
\otimes \Lambda^1 (V)) 
\otimes 
\Lambda ^{N-3}(V)
\stackrel{\phi_1 \otimes \phi_2} {\rightarrow}
S^{(2,1)} (V) \otimes 
 \Lambda^1 (V)^{\otimes i} \otimes \Lambda^2(V)^{ \otimes n-1 -i}, 
\]
where the first map is given by the coproduct on the first tensor factor, the isomorphism permutes the tensor factors and 
\begin{eqnarray*}
\phi_1 &:& \Lambda^2(V) \otimes \Lambda^1 (V)  \twoheadrightarrow  S^{(2,1)} (V) \\
\phi_2 &:& \Lambda ^{N-3}(V)  \rightarrow  \Lambda^1 (V)^{\otimes i} \otimes \Lambda^2(V)^{ \otimes n-1 -i}.
 \end{eqnarray*}
Here, $\phi_1$ is the projection (unique up to non-zero scalar multiple) and $\phi_2$ is the iterated coproduct.

The $\lie (V)$-action on $\ylin (n) (V)$ restricts to 
\[
\lie_2 (V) \otimes \ylin(n) (V)  \rightarrow \ylin(n) (V).
\]
In particular we may restrict this action to $\Lambda^1 (V) \otimes (\Lambda^1 (V)^{\otimes i} \otimes \Lambda^2(V)^{ \otimes n-1 -i}) \subset \ylin(n) (V)$, where $i$ is as above, which gives 
\[
\lie_2 (V) \otimes \big( 
\Lambda^1 (V) \otimes (\Lambda^1 (V)^{\otimes i} \otimes \Lambda^2(V)^{ \otimes n-1 -i})
\big) 
 \rightarrow \ylin(n) (V).
\]

Define $0 \neq \psi \in \nt \big(S^{(2,1^{N-2})}(V), \lie_2 (V) \otimes ( 
\Lambda^1 (V) \otimes (\Lambda^1 (V)^{\otimes i} \otimes \Lambda^2(V)^{ \otimes n-1 -i} ))
\big)  $ as the composite of the inclusion 
\[
S^{(2,1^{N-2})} (V) \hookrightarrow   \Lambda^1 (V) \otimes \Lambda^{N-1}(V)
\]
with the map induced by the iterated coproduct together with permutation of tensor factors:
\[
\Lambda^1 (V) \otimes \Lambda^{N-1}(V)
\rightarrow 
\Lambda^1 (V) \otimes (\Lambda ^2 (V) \otimes \Lambda^1 (V)^{\otimes i} \otimes \Lambda^2(V)^{ \otimes n-1 -i} ) 
\cong 
\Lambda ^2 (V) \otimes \big(\Lambda^1 (V) \otimes \Lambda^1 (V)^{\otimes i} \otimes \Lambda^2(V)^{ \otimes n-1 -i} \big), 
\]
where the first factor of $\Lambda ^2 (V) $ in the final term  identifies with $\lie_2 (V)$.

\begin{lem}
\label{lem:phi_act_psi}
The following diagram commutes up to possible non-zero scalar multiple:
\[
\xymatrix{
S^{(2,1^{N-2})} (V)
\ar[d]_\psi 
\ar[r]^\phi 
&
S^{(2,1)} (V) \otimes \Lambda^1 (V)^{\otimes i} \otimes \Lambda^2(V)^{ \otimes n-1 -i}
\ar@{^(->}[d]
\\
\lie_2 (V) \otimes \big( 
\Lambda^1 (V) \otimes (\Lambda^1 (V)^{\otimes i} \otimes \Lambda^2(V)^{ \otimes n-1 -i} )
\big)
\ar[r]
&
\ylin (n) (V),
}
\]
where the right hand vertical arrow is the inclusion $S^{(2,1)} (V) \otimes \Lambda^1 (V)^{\otimes i} \otimes \Lambda^2(V)^{ \otimes n-1 -i} \subset \ylin (n) (V)$ and the bottom arrow is given by the $\lie(V)$-action on $\ylin(n)(V)$.
\end{lem}

\begin{proof}
It is straightforward to see that the composite around the bottom of the diagram is non-trivial. Hence, by Lemma \ref{lem:phi}, it suffices to show that this composite factors across the inclusion given by the right hand vertical arrow. 

Using the coassociativity of the coproduct on exterior powers, this is a consequence of the fact that the composites 
\begin{eqnarray*}
&&\Lambda^3 (V) \hookrightarrow 
\Lambda^{2} (V) \otimes \Lambda^1 (V) 
= \lie_2 (V) \otimes \lie_1 (V) 
\rightarrow 
\lie_3 (V) 
\\
&&
\Lambda^4 (V) \hookrightarrow 
\Lambda^{2} (V) \otimes \Lambda^2 (V) 
= \lie_2 (V) \otimes \lie_2 (V) 
\rightarrow 
\lie_4 (V) 
\end{eqnarray*}
are zero, where the first map is the coproduct and the second is the Lie action; this holds since $\Lambda^n (V)$ is not a composition factor of $\lie_n(V)$ for $n>2$ (see Proposition \ref{prop:lie_reps}).

This implies that the only non-trivial contribution from the Lie action in the composite is that coming from $\lie_2 (V) \otimes \Lambda^1 (V) \rightarrow \lie_3 (V)$ applied to the distinguished (first) factor of $\Lambda^1 (V)$, as required.  
\end{proof}

\begin{proof}[Proof of Proposition \ref{prop:further_reduction_ylin}]
As explained above, it suffices to show that the map (\ref{eqn:reduction}) is zero; more specifically, it suffices to show that the image of the generator $\phi$ corresponding to Lemma \ref{lem:phi} is sent to zero in $\nt (S^{(2,1^{N-2})}(V), H_0 (\lie (V); \ylin(n)(V) ))$.
 This follows from Lemma \ref{lem:phi_act_psi}. 
\end{proof}

\subsection{A non-example}

That Theorem \ref{thm:2_1_N-2_isotypical} does not generalize to all isotypical components can be illustrated by considering the case $\rho = (2,2)$ (so that $N=4$) and  taking $n=2$. 

The homogeneous component of polynomial degree $4$ of $\lie(V) ^{\otimes 2}$ is 
\[
S^{(2,1)} (V) \otimes \Lambda^1 (V)  \  \oplus \  \Lambda^1 (V) \otimes  S^{(2,1)} (V) \ \oplus \ \Lambda^2 (V) \otimes \Lambda^2 (V), 
\]
where only the final factor arises from $\lie_{\leq 2}(V)^{\otimes 2}$.  

Now consider the Lie action, which gives 
$
\lie_1 (V) \otimes \lie(V) ^{\otimes 2}
\rightarrow 
\lie(V) ^{\otimes 2}$. 
The homogeneous component of polynomial degree $4$ of $\lie_1 (V) \otimes \lie(V) ^{\otimes 2}$ is:
\[
\lie_1 (V) \otimes (\Lambda^2 (V) \otimes \Lambda^1 (V) \ \oplus \ \Lambda^1 (V) \otimes \Lambda^2 (V) ),
\]
in which the terms  $\Lambda^2 (V) \otimes \Lambda^1 (V)$ and  $\Lambda^1 (V) \otimes \Lambda^2 (V) $ both arise from $\lie_{\leq 2}(V)^{\otimes 2}$.

Consider the commutative diagram for the respective actions:
\[
\xymatrix{
\lie_1 (V) \otimes \lie(V) ^{\otimes 2}
\ar[r]
\ar@{->>}[d]
&
\lie(V) ^{\otimes 2}
\ar@{->>}[d]
\ar@{->>}[r]
&
H_0 (\lie(V);  \lie(V) ^{\otimes 2})
\ar@{->>}[d]
\\
\lie_1 (V) \otimes \lie_{\leq 2}(V) ^{\otimes 2}
\ar[r]
&
\lie_{\leq 2}(V) ^{\otimes 2}
\ar@{->>}[r]
&
H_0 (\lie(V);  \lie_{\leq 2}(V) ^{\otimes 2}),
}
\]
in which the rows are exact and the  vertical maps are induced by the quotient map $\lie(V) \twoheadrightarrow \lie_{\leq 2}(V)$.

On applying the functor $\nt (S^{(2,2)} (V), -)$, by the above identifications one calculates:
\[
\xymatrix{
\nt (S^{(2,2)} (V), H_0 (\lie(V);  \lie(V) ^{\otimes 2})
\ar@{}[d]|\cong 
\ar[r]
&
\nt (S^{(2,2)} (V), H_0 (\lie(V);  \lie_{\leq 2} (V) ^{\otimes 2})
\ar@{}[d]|\cong 
\\
\kring 
\ar[r]
&
0 ,
}
\]
so that the projection does not induce an isomorphism on the $(2,2)$-isotypical component.

\part{The categorical framework}
\label{part:cat_framework}

\section{Generalized walled categories}
\label{sect:cat}

The purpose of this Section is to introduce two variants $\cc$ and $\cd$ of the upper walled Brauer category that underpin the methods of the remainder of the  paper. The reader is encourage to read it in parallel with the material of Part \ref{part:applications}, where the motivating examples are presented.

\subsection{$\cc$, a generalized upper walled Brauer category}
\label{subsect:cc}

We consider the following generalization of the upper walled Brauer category  that is of importance in representation theory (cf. \cite{MR3376738} for instance and Remark \ref{rem:invcc} below).

\begin{defn}
\label{defn:cc}
Let $\cc$ denote the category with objects pairs of finite sets $(X,Y)$ and in which a morphism $(X,Y) \rightarrow (U,V)$ is given by a triple $(i,j, \alpha)$ with $i \in \hfi (X,U)$, $j \in \hfi (Y, V) $ and $\alpha \in \hfi (V\backslash j(Y), U \backslash i(X))$.  

The composite with $(U,V ) \rightarrow (W,Z)$ given by the triple $(f, g, \beta)$ is the morphism $(X,Y) \rightarrow (W,Z)$ 
 given by the triple $(f  i \in \hfi (X,W), g j \in \hfi (Y, Z) , \gamma \in \hfi (Z \backslash gj (Y), W \backslash fi(X)) )$ where, with respect to the decompositions $Z \backslash gj (Y) = Z \backslash g(V) \amalg g (V \backslash j(Y))$ and $W \backslash fi(X) = W \backslash f(U) \amalg f (U \backslash i(X))$, $\gamma = \beta \amalg f \alpha g^{-1}$ where $g^{-1}$ is the isomorphism $g (V\backslash j(Y)) \cong V\backslash j(Y)$ inverse to $g$. 

The wide subcategory of $\cc$ with morphisms $(i, j, \alpha)$ for which $\alpha$ is a bijection is denoted $\cc_0$.
\end{defn}

\begin{rem}
The category $\cc_0$ is the upper walled Brauer category.
\end{rem}

The following is clear, in which  $\init$ denotes the unique map in $\hfi (\emptyset, Z)$ for any finite set $Z$:

\begin{lem}
\label{lem:subcat_cc}
\ 
\begin{enumerate}
\item 
The category $\cc$ is an $EI$-category (i.e., all endomorphisms are isomorphisms); the maximal subgroupoid is isomorphic to $\fb^{\times 2}$.
\item 
There is a forgetful functor $\cc \rightarrow \finj^{\times 2}$ that is the identity on objects and sends a morphism $(i, j , \alpha)$ of $\cc$ to $(i,j)$. 
\item 
There is an inclusion of $\finj \times \fb$ as the wide subcategory of $\cc$ with morphisms $(i, j, \init)$ such that $j$ is a bijection. 
\item 
The intersection of the wide subcategories $\cc_0$ and $\finj \times \fb$ in $\cc$ is the maximal subgroupoid $\fb ^{\times 2}$, so that there is a commutative diagram of inclusions of wide subcategories:
\[
\xymatrix{
\fb^{\times 2}
\ar@{^(->}[r]
\ar@{^(->}[d]
&
\cc_0
\ar@{^(->}[d]
\\
\finj \times \fb 
\ar@{^(->}[r]
&
\cc.
}
\]
\item 
The category $\cc_0$ decomposes into connected components $\cc_0 = \amalg_{n \in \zed} \cc_0^{(n)}$, where $\cc_0^{(n)}$ is the full subcategory with objects $(X,Y)$ such that $|X|- |Y|=n$. 
\end{enumerate}
\end{lem}

The morphisms of the  wide subcategories $\cc_0$ and $\finj \times \fb$ generate those of $\cc$, as follows from:

\begin{prop}
\label{prop:factor_cc}
Let $(i,j, \alpha)$ be a morphism in $\hom_\cc ((X,Y), (U,V))$. 
There are canonical factorizations:
\begin{enumerate}
\item 
$(X,Y) \rightarrow (U',V) \rightarrow (U,V)$, where $U'\subset U$ is $i(X) \amalg \alpha (V \backslash j(Y))$,  the first morphism belongs to $\cc_0$ and is given by $(i,j,\alpha)$ considered as mapping to $(U',V)$ and the second belongs to $\finj \times \fb$ and is given by $(U' \subset U, \id, \init)$; 
\item 
$(X,Y) \rightarrow (U'',Y) \rightarrow (U,V)$, where $U'' \subset U $ is $U \backslash (\alpha (V \backslash j(Y)))$, the first morphism belongs to $\finj \times \fb$ and is given by $(i, \id, \init)$ and the second belongs to $\cc_0$ and is given by $(U'' \subset U, j, \alpha)$. 
\end{enumerate}
These give the commutative diagram in $\cc$:
\[
\xymatrix{
(X,Y) 
\ar[r]^{\in \cc_0}
\ar[d]_{\in \finj \times \fb}
&
(U',V) 
\ar[d]^{\in \finj \times \fb}
\\
(U'', Y) 
\ar[r]_{\in \cc_0}
&
(U, V).
}
\]
\end{prop}

\begin{proof}
This is an immediate verification.
\end{proof}

\begin{rem}
\label{rem:structure_cc}
\
\begin{enumerate}
\item 
Proposition \ref{prop:factor_cc} does not assert that the factorizations of morphisms as pairs from $\cc_0$ and $\finj \times \fb$ are unique. The sub groupoid $\fb^{\times 2}$ accounts for the non-unicity. 
\item 
One can go further and present the morphisms of  $\cc$ by generators and relations, starting from $\fb^{\times 2}$. Namely, for each $(X,Y)\in \ob \fb^{\times 2}$, it suffices to add the following generators (given by the canonical inclusions):
\begin{eqnarray*}
(X,Y) & \rightarrow & (X \amalg \mathbf{1},Y) \\
(X,Y ) &  \rightarrow & (X \amalg \mathbf{1},Y\amalg \mathbf{1})
\end{eqnarray*}
that lie respectively in $\finj \times \fb$ and in $\cc_0$. These satisfy commutation relations that are left to be formulated by the reader. 
\end{enumerate}
\end{rem}

The subcategory $\cc_0$ has an obvious involution:

\begin{prop}
\label{prop:involution_cc0}
The category $\cc_0$ has involution $\invcc : \cc_0 \rightarrow \cc_0$ given on objects by $(X,Y) \mapsto (Y,X)$ and on morphisms 
by $(i,j,\alpha) \mapsto (j,i,\alpha^{-1})$. 
\end{prop}

\begin{rem}
\label{rem:invcc}
\ 
\begin{enumerate}
\item
The involution $\invcc$ does not extend to the category $\cc$, due to the asymmetric nature of the definition of the latter. 
\item
The symmetry exhibited by the involution $\invcc$ of $\cc_0$ is also abundantly clear in its applications in representation theory, as in \cite{MR991410} for example, where mixed representations of the general linear groups are considered. 

For example, for $\kring$ a field and $V$ a finite-dimensional $\kring$-vector space, there is a natural functor from $\cc_0 \op$ to $\kring$-vector spaces given by 
$ 
(\mathbf{m}, \mathbf{n}) \mapsto V^{\otimes m} \otimes (V^\sharp)^{\otimes n}$.
 Under this functor, $(\mathbf{0}, \mathbf{0}) \rightarrow (\mathbf{1}, \mathbf{1})$ gives the  evaluation map $V \otimes V^\sharp \rightarrow \kring$. 

Vector space duality induces an involution of the groupoid of finite-dimensional $\kring$-vector spaces and their automorphisms, which reflects the involution $\invcc$ on $\cc_0$ through the above structure, identifying the evaluation map as $(V^\sharp) ^\sharp \otimes V^\sharp \rightarrow \kring$ (using that $V$ is finite-dimensional) and transposing tensor factors.
\end{enumerate}
\end{rem}

\subsection{A category of elements for $\cc$}

The additional structure of the morphisms of $\cc$ provided by the injection $\alpha$ of the triple $(i, j ,\alpha)$ leads to the following:

\begin{prop}
\label{prop:cc_to_sets}
There is a functor $\hfi (-,-) : \cc \rightarrow \sets$ given on objects by $(X,Y) \mapsto \hfi (Y,X)$. Given $(i,j,\alpha) \in \hom_\cc ((X,Y), (U,V))$, the image of $\kappa \in \hfi (Y, X)$ is the unique $\tilde{\kappa} \in \hfi (V,U)$ that fits into the commutative diagram
\[
\xymatrix{
Y
\ar[d]_\kappa 
\ar@{^(->}[r]^j
&
V 
\ar[d]^{\tilde{\kappa}}
&
V \backslash j(Y)
\ar[d]^\alpha
\ar@^{_(->}[l]
\\
X 
\ar@{^(->}[r]^i
&
U
&
U\backslash i(X).
\ar@^{_(->}[l]
}
\]
\end{prop}

\begin{proof}
It is a straightforward verification that the above definition is compatible with the composition of morphisms; it clearly respects the identity morphisms.
\end{proof}

Using Proposition \ref{prop:cc_to_sets}, one can consider the category of elements $\cc_{\hfi(-,-)}$, equipped with the  canonical  projection 
$
\cc_{\hfi(-,-)} \rightarrow \cc.
$ . The category $\cc_{\hfi(-,-)}$ has objects  $((X,Y), \kappa)$, where $(X,Y) \in \ob \cc$ and $\kappa \in \hfi (Y,X)$.  A morphism from  $((X,Y), \kappa)$ to $((U, V), \lambda)$ is given by a morphism $(i, j , \alpha) \in \hom_\cc ((X,Y), (U,V))$ such that $\lambda$ is the image of $\kappa$ in $\hfi (V,U)$ (i.e., $\lambda = \tilde{\kappa}$ in the notation of Proposition \ref{prop:cc_to_sets}).

\begin{cor}
\label{cor:cc_elements_to_fi}
There is a functor 
$ 
\cc_{\hfi(-,-)} \rightarrow \finj
$ 
given on objects by $((X,Y), \kappa)\mapsto X \backslash \kappa (Y)$.
\end{cor}

\begin{proof}
One first checks the behaviour on morphisms. For this, consider a morphism from  $((X,Y), \kappa)$ to $((U, V), \tilde{\kappa})$, using the notation of Proposition \ref{prop:cc_to_sets}. It suffices to show that $i : X \hookrightarrow U$ restricts to 
\[
X \backslash \kappa (Y)
\hookrightarrow 
U \backslash \tilde{\kappa}(V).
\]
Now, by construction, the image of $\tilde{\kappa}$ is contained in $i\kappa(Y) \amalg U \backslash i(X)$. In particular,   $i(X \backslash \kappa (Y))$ is disjoint from the image of $\tilde{\kappa}$, as required. That this defines a functor is then a straightforward verification.
\end{proof}

\subsection{$\cd$, the half-doubled variant of $\cc$}

Recall that $\mathbf{2}$ denotes the finite set $\{1, 2\}$. 

\begin{nota}
\label{nota:epsilon}
For $\psi : \mathbf{2} \times X \rightarrow Y$ a map of finite sets and $\epsilon \in \{1, 2\}$, denote by $\psi_\epsilon : X \rightarrow Y$  the map given by the composite
\[
X \cong \{\epsilon \} \times X \subset \mathbf{2} \times X \stackrel{\psi}{\rightarrow} Y.
\]
\end{nota}

\begin{defn}
\label{defn:cd}
Let $\cd$ be the category with objects pairs of finite sets $(X,Y)$ and in which a morphism $(X,Y) \rightarrow (U,V)$ is given by a triple $(i, j, \zeta)$, where $i \in \hfi (X,U)$, $j \in \hfi (Y, V)$ and 
$$\zeta \in \hfi (\mathbf{2} \times (V \backslash j(Y)), U \backslash i(X)).
$$  Composition is defined analogously to that of $\cc$. 

Let $\cd_0 \subset \cd$ be the wide subcategory with morphisms $(i, j , \zeta)$ in which $\zeta$ is a bijection.
\end{defn}

The composition of $\cd$ can be described explicitly using:

\begin{lem}
\label{lem:cd_to_cc}
For $\epsilon \in \{1, 2\}$, the map $(i, j , \zeta) \mapsto (i, j, \zeta_\epsilon)$, defines a functor 
$
(-)_\epsilon : \cd \rightarrow \cc
$ 
that is the identity on objects. The  functors $(-)_1, (-)_2 $ induce a faithful embedding $\cd \rightarrow \cc \times \cc$.
\end{lem}

\begin{proof}
To give a morphism $(i, j , \zeta)$ of $\cd$ is equivalent to specifying a pair of morphisms $(i, j, \zeta_1)$ and $(i, j, \zeta_2)$ of $\cc$ that satisfy the condition that the images of $\zeta_1$ and $\zeta_2$ are disjoint. 
\end{proof}

Using Lemma \ref{lem:cd_to_cc}, it is immediate that Lemma \ref{lem:subcat_cc} has  counterpart:

\begin{lem}
\label{lem:subcat_cd}
\ 
\begin{enumerate}
\item 
The category $\cd$ is an $EI$-category with maximal subgroupoid isomorphic to $\fb^{\times 2}$.
\item 
There is a forgetful functor $\cd \rightarrow \finj^{\times 2}$ that is the identity on objects and sends a morphism $(i, j , \zeta)$ of $\cd$ to $(i,j)$. 
\item 
There is an inclusion of $\finj \times \fb$ as the wide subcategory of $\cc$ with morphisms $(i, j, \init)$ such that $j$ is a bijection. 
\item 
The intersection of the wide subcategories $\cd_0$ and $\finj \times \fb$ in $\cc$ is the sub groupoid $\fb ^{\times 2}$, so that there is a commutative diagram of inclusions of wide subcategories:
\[
\xymatrix{
\fb^{\times 2}
\ar@{^(->}[r]
\ar@{^(->}[d]
&
\cd_0
\ar@{^(->}[d]
\\
\finj \times \fb 
\ar@{^(->}[r]
&
\cd.
}
\]
\item 
The category $\cd_0$ decomposes into connected components $\cd_0 = \amalg_{n \in \zed} \cd_0^{(n)}$, where $\cd_0^{(n)}$ is the full subcategory with objects $(X,Y)$ such that $|X|- 2 |Y|=n$. 
\item 
The functors of Lemma \ref{lem:cd_to_cc} restrict to $(-)_1, (-)_2 : \cd_0 \rightrightarrows \cc_0$ and these are compatible with the connected component decompositions: for $n \in \zed$, they restrict to 
\[
(-)_1, (-)_2 : \cd_0^{(n)} \rightrightarrows \cc_0^{(n)}.
\] 
\end{enumerate}
\end{lem}

Likewise, Proposition \ref{prop:factor_cc} has the counterpart:

\begin{prop}
\label{prop:factor_cd}
Let $(i,j, \zeta)$ be a morphism in $\hom_\cd ((X,Y), (U,V))$. 
There are canonical factorizations in $\cd$:
\begin{enumerate}
\item 
$(X,Y) \rightarrow (U',V) \rightarrow (U,V)$, where $U'\subset U$ is $i(X) \amalg \zeta (\mathbf{2} \times (V \backslash j(Y)))$, the first morphism belongs to $\cd_0$ and is given by $(i,j,\zeta)$ considered as mapping to $(U',V)$ and the second belongs to $\finj \times \fb$ and is given by $(U' \subset U, \id, \init)$; 
\item 
$(X,Y) \rightarrow (U'',Y) \rightarrow (U,V)$, where $U'' \subset U $ is $U \backslash \zeta (\mathbf{2} \times (V \backslash j(Y)))$, the first morphism belongs to $\finj \times \fb$ and is given by $(i, \id, \init)$ and the second belongs to $\cd_0$ and is given by $(U'' \subset U, j, \zeta)$. 
\end{enumerate}
These give the commutative diagram in $\cd$:
\[
\xymatrix{
(X,Y) 
\ar[r]^{\in \cd_0}
\ar[d]_{\in \finj \times \fb}
&
(U',V) 
\ar[d]^{\in \finj \times \fb}
\\
(U'', Y) 
\ar[r]_{\in \cd_0}
&
(U, V).
}
\]
\end{prop}

\begin{rem}
\label{rem:structure_cd}
As for $\cc$ in Remark \ref{rem:structure_cc}, one can present the morphisms of $\cd$ by generators and relations, starting from  $\fb^{\times 2}$. In this case, for each $(X,Y)\in \ob \fb^{\times 2}$, it suffices to add the following generators induced by the canonical inclusions:
\begin{eqnarray*}
(X,Y) & \rightarrow & (X \amalg \mathbf{1},Y) \\
(X,Y ) &  \rightarrow & (X \amalg \mathbf{2},Y\amalg \mathbf{1})
\end{eqnarray*}
that lie respectively in $\finj \times \fb$ and in $\cd_0$ (for the second, there is a canonical choice for the map $\zeta$). 
\end{rem}

\subsection{Categories of elements for $\cd$}

\begin{nota}
For $\epsilon \in \{1, 2\}$, denote by $\hfi (-,-)_\epsilon : \cd \rightarrow \sets$ the composite functor $\hfi (-,-) \circ (-)_\epsilon$, where $(-)_\epsilon  : \cd \rightarrow \cc$ is the functor of Lemma \ref{lem:cd_to_cc}. 
\end{nota}

We choose to focus upon $\hfi (-,-)_1 : \cd \rightarrow \sets$. As in Section \ref{subsect:cc}, one can form the category of elements $\cd_{\hfi(-,-)_1}$. Corollary \ref{cor:cc_elements_to_fi} has the counterpart:

\begin{prop}
\label{prop:cd_elements_to_cc}
There is a functor 
$ 
\contract
:
\cd_{\hfi(-,-)_1}
\rightarrow 
\cc
$ 
given on objects by $((X,Y), \kappa) \mapsto (X \backslash \kappa (Y), Y)$. 
\end{prop}

\begin{proof}
Using the functor given by Corollary \ref{cor:cc_elements_to_fi}, it is clear that one obtains a functor 
$ 
\cd_{\hfi(-,-)_1}
\rightarrow 
\finj ^{\times 2}
$ 
given on objects as above. It remains to show that this factors canonically across the forgetful functor $\cc \rightarrow \finj ^{\times 2}$.

For this, consider a morphism $((X,Y),\kappa) \rightarrow ((U,V), \tilde{\kappa})$ in $\cd_{\hfi(-,-)_1}$ with underlying morphism in $\cd$ given by the triple $(i,j,\zeta)$. Then, as in Proposition \ref{prop:cc_to_sets}, $\tilde{\kappa}$ is determined by the commutative diagram:
\[
\xymatrix{
Y
\ar[d]_\kappa 
\ar@{^(->}[r]^j
&
V 
\ar[d]^{\tilde{\kappa}}
&
V \backslash j(Y)
\ar[d]^{\zeta_1}
\ar@^{_(->}[l]
\\
X 
\ar@{^(->}[r]^i
&
U
&
U\backslash i(X),
\ar@^{_(->}[l]
}
\]
noting that the right hand vertical morphism is given by $\zeta_1$, reflecting the usage of the functor $(-)_1$.

We require to construct the injection $V \backslash j(Y) \hookrightarrow (U \backslash \tilde{\kappa}(V)) \backslash i (X \backslash \kappa (Y))$. This is provided by $\zeta_2$;  it suffices to check that $\zeta_2 : V \backslash j(Y) \hookrightarrow U \backslash i(X)$ maps to $(U \backslash \tilde{\kappa}(V)) \backslash i (X \backslash \kappa (Y)) \subset U \backslash i(X)$. This follows, using the above identification of $\tilde{\kappa}$, since the images of $\zeta_1$ and $\zeta_2$ are disjoint. 

That this defines a functor is a straightforward verification.
\end{proof}

\begin{rem}
\label{rem:transfer_cc_to_cd}
The functors 
\[
\xymatrix{
\cd 
&
\cd_{\hfi(-,-)_1}
\ar@{->>}[l]
\ar[r]^(.7)\contract
&
\cc,
}
\]
(in which the left hand functor is the canonical projection) provide a method for transferring functors on $\cc$ to functors on $\cd$ when restriction along $\cd_{\hfi(-,-)_1} \rightarrow \cd$ admits a left adjoint.
\end{rem}

To illustrate this, consider the following:

\begin{prop}
\label{prop:cc_op_cd_bifunctor}
There is a functor 
\begin{eqnarray*}
\cc \op \times \cd &\rightarrow & \sets \\
((X,Y), (U,V)) & \mapsto & \amalg _{\kappa \in \hfi (V,U)}  \hom_\cc ((X,Y),(U\backslash \kappa(V) , V)).  
\end{eqnarray*}
\end{prop}

\begin{proof}
Using  $\contract : \cd_{\hfi(-,-)_1} \rightarrow \cc$ of Proposition \ref{prop:cd_elements_to_cc}, one has the functor 
\begin{eqnarray*}
\cc \op \times \cd_{\hfi(-,-)_1} &\rightarrow & \sets \\
((X,Y), ((U,V),\kappa)) & \mapsto &  \hom_\cc ((X,Y),(U\backslash \kappa(V) , V)).  
\end{eqnarray*}
The dependency on $\kappa$ is removed by forming the disjoint union (this corresponds to a left Kan extension). This yields the required functor.
\end{proof}

 \section{Further structure for $\cc$ and $\cd$}
 \label{sect:cat_more}

This Section continues on from Section \ref{sect:cat}, introducing further structure  that is important in relating the categories $\cc$ and $\cd$ and functors defined on them.

\subsection{Symmetric monoidal structures for $\cc$ and $\cd$}

The disjoint union $\amalg$ of finite sets clearly induces symmetric monoidal structures as follows:

\begin{prop}
\label{prop:sym_monoidal}
There are symmetric monoidal structures $(\cc, \amalg , (\mathbf{0}, \mathbf{0}))$ and $(\cd, \amalg , (\mathbf{0}, \mathbf{0}))$ defined on objects by $(X,Y) \amalg (U,V)= (X \amalg U, Y\amalg V)$. 

These restrict to $(\cc_0, \amalg , (\mathbf{0}, \mathbf{0}))$ and $(\cd_0, \amalg , (\mathbf{0}, \mathbf{0}))$ respectively, respecting the  connected components in the following sense: for $s,t \in \nat$,
\begin{eqnarray*}
\cc_0^{(s)} \times \cc_0^{(t)} & \stackrel{\amalg}{\rightarrow } & \cc_0 ^{(s+t)}
\\
\cd_0^{(s)} \times \cd_0^{(t)} & \stackrel{\amalg}{\rightarrow } & \cd_0 ^{(s+t)}.
\end{eqnarray*}
\end{prop}

\begin{rem}
\label{rem:00_not_initial}
Note that $(\mathbf{0}, \mathbf{0})$ is not initial in any of these categories. For instance, there is a natural isomorphism of $\aut(X) \times \aut (Y)$-sets:
\[
\hom_{\cc} ((\mathbf{0}, \mathbf{0}) , (X,Y) ) 
\cong 
\hfi (Y, X). 
\]

On restricting to $\cc_0$, this gives
\[
\hom_{\cc_0} ((\mathbf{0}, \mathbf{0}) , (X,Y) ) 
\cong 
\mathrm{Iso} (Y, X),
\]
(which is empty if $|X|\neq |Y|$). In particular, for $n \in \nat$, 
$$\hom_{\cc} ((\mathbf{0}, \mathbf{0}) , (\mathbf{n},\mathbf{n}) )=\hom_{\cc_0} ((\mathbf{0}, \mathbf{0}) , (\mathbf{n},\mathbf{n}) ) \cong \sym_n,$$
 with the regular bimodule structure.  
\end{rem}

The symmetric monoidal structure on $\cc$ provides shift functors $ - \amalg (\mathbf{m}, \mathbf{n})$ for $m,n \in \nat$. By composing such shift functors, one can reduce to studying the `generators' $ - \amalg  (\mathbf{0}, \mathbf{1})$ and $ - \amalg  (\mathbf{1}, \mathbf{0})$. If working with $\cc_0$, using the involution $\invcc$, one may restrict further to considering just one of these.

 For the functor $ - \amalg  (\mathbf{0}, \mathbf{1}) : \cc \rightarrow \cc$,  one has the following:

\begin{lem}
\label{lem:shift_cc}
For $(U, V) \in \ob \cc$, there is a natural isomorphism (with respect to $(X,Y) \in \ob \cc$):
\begin{eqnarray*}
\hom _\cc ((U,V), (X, Y \amalg \mathbf{1}))
&\cong &
\hom_\cc ((U\amalg\mathbf{1}, V) , (X,Y) ) 
\amalg
\coprod_{v \in V} 
\hom_\cc ((U, V \backslash \{v \}), (X,Y))
\\
&\cong &
\hom_\cc ((U\amalg\mathbf{1}, V) , (X,Y) ) 
\amalg
\Big( \aut(V) \times _{\aut (V')} \hom_\cc ((U, V'), (X,Y))
\Big), 
\end{eqnarray*}
where $V' \subset V$ with $|V'|=|V|-1$.

Moreover, this is equivariant with respect to the action of $\aut(U) \times \aut (V)$.  
\end{lem}

\begin{proof}
An element of $\hom _\cc ((U,V), (X, Y \amalg \mathbf{1}))$ is a triple $(i: U \hookrightarrow X, j: V \hookrightarrow Y \amalg \mathbf{1}, \alpha : (Y \amalg \mathbf{1}) \backslash j(V) \hookrightarrow X \backslash i (U) )$. One distinguishes the cases $j (V) \subset Y$ and $\mathbf{1} \subset j(V)$.  

In the first case, giving such a datum is equivalent to giving the triple $(i' : U \amalg \mathbf{1}\hookrightarrow X, j' : V \hookrightarrow Y, \alpha': Y \backslash j(V) \hookrightarrow X \backslash i (U \amalg \mathbf{1}) )$, where $i'$ is the extension of $i$ by the restriction of $\alpha$ to $\mathbf{1}$, $j'=j$ and $\alpha'$ is the restriction of $\alpha$  

In the second case, set $v := j^{-1} (1)$. Then, such a datum is equivalent to giving the triple 
$(i'': U \hookrightarrow X, j'' : V \backslash \{v \} \hookrightarrow Y, \alpha'' : Y \backslash j'' (V \backslash \{v \}) \hookrightarrow X \backslash i(U))$, where $i''=i$, $j''$ is  the restriction of $j$ and $\alpha''$ identifies with $\alpha$.

One checks that the bijection thus obtained is natural with respect to $(X,Y)\in \ob \cc$ and with respect to the action of $\aut(U) \times \aut (V)$.  
\end{proof}

\begin{rem}
If one considers $(X\amalg \mathbf{1}, Y )$ in place of $(X, Y \amalg \mathbf{1})$, then the argument in the proof of Lemma \ref{lem:shift_cc} does not go through  (unless one restricts to $\cc_0$). The difficulty is that, for a given morphism $(i, j, \alpha)$ from $(U, V)$ to $(X\amalg \mathbf{1}, Y )$, in the case $i(U) \subset X$, the element $1 \in \mathbf{1}$ need not be in the image of $\alpha$, hence one cannot in general construct the required   morphism $\mathbf{1} \hookrightarrow X$ naturally. 
\end{rem}

One can go further in the context of Lemma \ref{lem:shift_cc} and analyse the full naturality with respect to $(U,V) \in \ob \cc$. 

\begin{nota}
\label{nota:shift_morphisms_cc}
For $\Phi \in \hom_\cc ((A,B), (U,V))$ given by the triple $(f, g, \gamma)$, denote by:
\begin{enumerate}
\item 
$\Phi_{\mathbf{1}} \in \hom_\cc ((A\amalg \mathbf{1},B), (U\amalg \mathbf{1},V))$ the morphism   obtained from $\Phi$ by
applying the functor $- \amalg (\mathbf{1}, \mathbf{0})$; 
\item
for $v \in V \backslash g(B)$, $\Phi_{v \not \in g(B)} \in \hom _\cc ((A \amalg \mathbf{1}, B) , (U, V\backslash \{ v\} )$ the morphism determined by $(f', g, \gamma')$ where $f'$ extends $f$ by mapping $1 \in \mathbf{1}$ to $\gamma (v)$ and $\gamma'$ is induced by $\gamma$; 
 \item 
 for $v \in g(B) \subset V$, $\Phi_{v  \in g(B)} \in \hom _\cc ((A, B\backslash \{ g^{-1} (v) \}) , (U, V\backslash \{ v\} )$ the morphism determined by the triple $(f, g', \gamma)$, where $g'$ is the restriction of $g$.
\end{enumerate}
\end{nota}

The morphisms of Notation \ref{nota:shift_morphisms_cc} induce the respective natural transformations:
\begin{enumerate}
\item 
$
\Phi_{\mathbf{1}}^* : 
\hom_\cc ((U\amalg\mathbf{1}, V) , -) ) 
\rightarrow 
\hom_\cc ((A\amalg\mathbf{1}, B) , - )
$;
\item 
if $v \in V \backslash g(B)$, $\Phi_{v \not \in g(B)} ^* : 
\hom_\cc ((U, V \backslash \{v \}), -)
 \rightarrow  
\hom_\cc ((A\amalg\mathbf{1}, B) , - )$;
\item 
if $v \in g(B) \subset V$, $\Phi_{v \in g(B)} ^* : 
\hom_\cc ((U, V \backslash \{v \}), -)
 \rightarrow 
\hom_\cc ((A, B \backslash \{g^{-1}(v) \}), -)$
.
\end{enumerate}

\begin{nota}
\label{nota:Phi_natural}
For $\Phi \in \hom_\cc ((A,B), (U,V))$ given by the triple $(f, g, \gamma)$, denote by 
\[
\Phi^\natural : 
\hom_\cc ((U\amalg\mathbf{1}, V) , - ) 
\amalg
\coprod_{v \in V} 
\hom_\cc ((U, V \backslash \{v \}), -)
\rightarrow 
\hom_\cc ((A\amalg\mathbf{1}, B) , - ) 
\amalg
\coprod_{b \in B} 
\hom_\cc ((A, B \backslash \{b \}), -)
\]
the natural transformation induced by the natural morphisms $\Phi_{\mathbf{1}}^*$, $\Phi_{v \not \in g(B)} ^*$,  $\Phi_{v \in g(B)} ^*$, using that $\amalg$ is the coproduct in sets.
\end{nota}

One has the following refinement of Lemma \ref{lem:shift_cc} identifying the full functoriality:

\begin{prop}
\label{prop:shift_cc_full}
The associations $\Phi \mapsto \Phi^\natural $ make 
\[
(U,V) \mapsto \hom_\cc ((U\amalg\mathbf{1}, V) , - ) 
\amalg
\coprod_{v \in V} 
\hom_\cc ((U, V \backslash \{v \}), -)
\]
into a functor on $\cc\op$.

With respect to this structure, the isomorphisms of  Lemma \ref{lem:shift_cc} define a natural isomorphism of functors on $\cc\op \times \cc$:
\begin{eqnarray*}
\hom _\cc ((U,V), (X, Y \amalg \mathbf{1}))
&\cong &
\hom_\cc ((U\amalg\mathbf{1}, V) , (X,Y) ) 
\amalg
\coprod_{v \in V} 
\hom_\cc ((U, V \backslash \{v \}), (X,Y))
\end{eqnarray*}
\end{prop}

\begin{proof}
The proof is a direct verification, using the explicit definitions of $\Phi \mapsto \Phi^\natural$ and the isomorphism given in Lemma \ref{lem:shift_cc}.
\end{proof}

\begin{exam}
\label{exam:hom0D_cc0}
Fix $d \in \nat$ and consider the functor $\hom_{\cc_0} ((\mathbf{0}, \mathbf{d}), -)$. Since we have restricted to $\cc_0$, there is a  natural isomorphism:
\[
\hom_{\cc_0} ((\mathbf{0}, \mathbf{d}), (X,Y))
\cong 
\mathrm{Iso} (X  \amalg \mathbf{d}, Y)
\]
(This may be compared with the case $d=0$ given in Remark \ref{rem:00_not_initial}; the two expressions are equivalent via  the involution $\invcc$.) This isomorphism is $\sym_d$-equivariant with respect to the obvious actions. 

Then, by Lemma \ref{lem:shift_cc}, one has the $\sym_d$-equivariant isomorphisms:
\[
\mathrm{Iso} (X  \amalg \mathbf{d}, Y\amalg \mathbf{1})
\cong 
\hom_{\cc_0} ((\mathbf{0}, \mathbf{d}), (X,Y \amalg \mathbf{1}))
\cong 
\hom_{\cc_0} ((\mathbf{1}, \mathbf{d}), (X,Y) )
\amalg 
\Big (\sym_d \times _{\sym_{d-1}} \mathrm{Iso} (X  \amalg (\mathbf{d-1}), Y)\Big).
\]
\end{exam}

\subsection{Relating $\cc$ and $\cd$}

We start by exhibiting a section to the functor $\cd_{\hfi(-,-)_1}
\rightarrow \cc$ of Proposition \ref{prop:cd_elements_to_cc}.

\begin{lem}
\label{lem:dblelt}
There is a functor $\dblelt : \cc \rightarrow \cd_{\hfi(-,-)_1}$ given on objects by 
\[
(X,Y) \mapsto ((X \amalg Y, Y), \kappa_{X,Y})
\]
where $\kappa_{X,Y} : Y \hookrightarrow X \amalg Y$ is the canonical inclusion. 

For $(i, j , \alpha) \in \hom_\cc ((X,Y), (U,V))$, the underlying morphism of $\dblelt (i,j , \alpha)$  in $ \hom_\cd ((X \amalg Y, Y),(U \amalg V, V))$ is given by the triple $(i \amalg j, j, \zeta)$, where 
\[
\zeta : \mathbf{2} \times (V \backslash j (Y))
= \{1\} \times (V \backslash j (Y))\amalg \{2\} \times (V \backslash j (Y))
 \hookrightarrow 
(U \amalg V) \backslash ((i \amalg j) (X \amalg Y)) \cong (V \backslash j(Y)) \amalg (U \backslash i(X)) 
\]
is given by  $\id$ on $\{1\} \times (V \backslash j (Y))$ and $\alpha$ on $\{2\} \times (V \backslash j (Y))$ with respect to the above identification.

The composite $\cc \stackrel{\dblelt}{\rightarrow} \cd_{\hfi(-,-)_1} \stackrel{\contract}{\rightarrow} \cc$ is naturally isomorphic to the identity functor.
\end{lem}

\begin{proof}
One first checks that, after composing with the projection $\cd_{\hfi(-,-)_1} \rightarrow \cd$, one obtains a functor $\cc \rightarrow \cd$. 
 This is a direct verification. 
 
To establish that $\dblelt$ is a functor as claimed, it remains to check that this behaves correctly on the canonical injections $\kappa_{X,Y}$. This follows from the construction of the functor $\hfi(-,-)_1$. The details are left to the reader.

The statement concerning the composite is clear.
\end{proof}

\begin{defn}
\label{defn:dbl}
Let $\dbl : \cc \rightarrow \cd$ be the composite $\cc \stackrel{\dblelt}{\rightarrow } \cd_{\hfi(-,-)_1} \rightarrow \cd$. 
\end{defn}

The following Lemma records some immediate properties of $\dbl$.

\begin{lem}
\label{lem:functor_dbl}
The functor $\dbl : \cc \rightarrow \cd$ is a faithful embedding.  This restricts to a faithful embedding $\dbl : \cc_0 \rightarrow \cd_0$, i.e., the following diagram commutes:
\[
\xymatrix{
\cc_0 \ar[r]^{\dbl}
\ar@{^(->}[d]
&
\cd_0
\ar@{^(->}[d]
\\
\cc 
\ar[r]_{\dbl}
&
\cd.
}
\]

Moreover, $\dbl$ is compatible with the respective decompositions of $\cc_0$ and $\cd_0$ into connected components: for $n \in \zed$, the functor $\dbl$ restricts to 
$ 
\dbl : \cc_0^{(n)} 
\rightarrow 
\cd_0^{(n)}.
$ 
\end{lem}

\begin{rem}
\ 
\begin{enumerate}
\item 
The functor $\dbl$ depends on the choice of which component of $\mathbf{2} \times (V \backslash j (Y))
= \{1\} \times (V \backslash j (Y))\amalg \{2\} \times (V \backslash j (Y))$ one should apply $\alpha$.  This corresponds to the choice to work with $\cd_{\hfi(-,-)_1}$.
\item 
Using the functors of Lemma \ref{lem:cd_to_cc}, one has the composites:
\[
\xymatrix{
\cc 
\ar[r]^\dbl 
&
\cd 
\ar@<.5ex>[r]^{(-)_1} 
\ar@<-.5ex>[r]_{(-)_2} 
&
\cc.
}
\]
At the level of objects, both are given by $(X,Y)\mapsto (X \amalg Y, Y)$. For $(i, j , \alpha) \in \hom_\cc ((X,Y), (U,V))$, $(-)_\epsilon \circ \dbl (i,j , \alpha) \in \hom_\cd ((X \amalg Y, Y),(U \amalg V, V))$ is given by the triple $(i \amalg j, j, \alpha_\epsilon)$, where 
\begin{enumerate}
\item 
$\alpha_1: V \backslash j(Y) \hookrightarrow V \subset U \amalg V$ is the canonical inclusion (thus independent of $\alpha$);
\item 
$\alpha_2$ is the composite of $\alpha : V \backslash j(Y) \hookrightarrow U\backslash i(X)$ with the inclusion $U\backslash i(X) \subset U \subset U \amalg V$.
\end{enumerate}
\item 
The  map on objects $\dbl : \ob \cc \rightarrow \ob \cd$ is not surjective. The image is given by the set of pairs of finite sets $(U, V)$ such that $|U |\geq |V|$. 
\item 
After restriction to the subcategory of $\cd_{\hfi (-,-)_1}$ corresponding to $\cd_0 \subset \cd$, the functors $\dblelt$ and $\contract$ are also compatible with the decompositions into connected components.
\end{enumerate}
\end{rem}

The functor $\dblelt$ has the following  interpretation:

\begin{prop}
\label{prop:dblelt_left_adj}
The functor $\dblelt : \cc \rightarrow \cd_{\hfi(-,-)_1}$ is left adjoint to the functor $\contract : \cd_{\hfi(-,-)_1}
\rightarrow \cc$ of Proposition \ref{prop:cd_elements_to_cc}.
\end{prop}

\begin{proof}
For $(X,Y) \in \ob \cc$ and $((U,V),\kappa) \in \ob \cd_{\hfi(-,-)_1}$, we first define the (natural) map 
\[
\hom_{\cd_{\hfi(-,-)_1}} (\dblelt(X,Y) , ((U,V),\kappa) ) 
\rightarrow 
\hom_\cc ((X,Y), (U\backslash \kappa (V), V)).
\]

A morphism of the domain has underlying morphism in $\cd$ given by a triple $(i,j, \zeta)$, where $i :X \amalg Y \hookrightarrow U$, $j: Y \hookrightarrow V$ and $\zeta : \mathbf{2} \times V \backslash j(Y) \hookrightarrow U \backslash i(X \amalg Y)$. The compatibility with $\kappa$ is given by the commutativity of the diagram:
\begin{eqnarray}
\label{eqn:zeta_1_vs_kappa}
\xymatrix{
Y \ar@{^(->}[r]^j
\ar@{_(->}[d]
\ar[rd]|{i_Y}
&
V
\ar[d]^\kappa
&
V \backslash j(Y) 
\ar[d]^{\zeta_1} 
\ar@{_(->}[l]
\\
X \amalg Y
\ar[r]_i 
&
U 
&
U \backslash i(X \amalg Y),
\ar@{_(->}[l]
}
\end{eqnarray}
in which $i_Y$ denotes the restriction of $i$ to $Y$. Writing $i_X$ for the restriction of $i$ to $X$, one also has that 
 $\zeta_2$ factors as the inclusion $\zeta_2 : V \backslash j(Y) 
\hookrightarrow (U\backslash \kappa (V))\backslash i_X (X)$. 

The image of the morphism corresponding to $(i,j,\zeta)$ in $\hom_\cc ((X,Y), (U\backslash \kappa (V), V))$ is defined to be the triple $(i_X , j, \zeta_2)$. This construction is clearly functorial. 

To see that the map is a bijection, one uses the fact that the commutative diagram (\ref{eqn:zeta_1_vs_kappa}) determines $\zeta_1$ and $i_Y$ in terms of $\kappa$. 
\end{proof}

The proof of  Proposition \ref{prop:dblelt_left_adj} shows more, since the definition of the category of elements implies that, for an element of  $\hom_\cd (\dbl (X,Y), (U,V))$, there exists a unique $\kappa \in \hfi (V,U)$ for which the element is in the image of
\[
\hom_{\cd_{\hfi(-,-)_1}} (\dblelt(X,Y) , ((U,V),\kappa) ) 
\rightarrow 
\hom_\cd (\dbl(X,Y), (U,V)).
\]

Hence, from Proposition \ref{prop:dblelt_left_adj}, one  deduces:

\begin{cor}
\label{cor:pseudo_adjunction}
There is a natural isomorphism of functors $\cc \op \times \cd \rightarrow  \sets $:
\[
\hom_\cd (\dbl (X,Y), (U,V))
\stackrel{\cong}{\rightarrow}
\amalg _{\kappa \in \hfi (V,U)}  \hom_\cc ((X,Y),(U\backslash \kappa(V) , V)).
\]
for $(X,Y) \in \ob \cc$ and $(U,V) \in \ob \cd$, where the codomain is the functor of Proposition \ref{prop:cc_op_cd_bifunctor}.
\end{cor}

\section{Modules on $\cc$ and $\cd$}
\label{sect:psh}

This Section first introduces $\apsh{\cc}$ and $\apsh{\cd}$, their subcategories $\apsh{\cc_0}$ and $\apsh{\cd_0}$ and the related left adjoints $\hcc$ and $\hcd$
 (see Section \ref{subsect:extensions}). Then the functor $\ind : \apsh{\cc} \rightarrow \apsh{\cd}$ is introduced, so as to relate these categories in Section \ref{subsect:int}. 

\subsection{Extension of modules}
\label{subsect:extensions}

To motivate the constructions occurring below, we first recall the relationship between $\apsh{\finj}$ and $\apsh{\fb}$. In addition to restriction from $\apsh {\finj}$  to  $\apsh {\fb}$, one has the extension by zero functor $\apsh {\fb} \rightarrow \apsh {\finj}$. The left adjoint to the extension functor  gives the functor $H_0^\finj$, whose left derived functors are, by definition, the $\finj$-homology functors (cf. \cite{MR3654111}, for example).

For $F \in \ob \apsh{\finj}$, the values of $H^\finj_0 F$ are given explicitly by
\[
H^\finj_0 F (X) = 
\mathrm{Coker} 
\Big( 
\bigoplus_{\substack{ i: Y \hookrightarrow X\\ |Y|= |X|-1}}
F(Y) 
\stackrel{\bigoplus F(i)}{\rightarrow}
F (X)
\Big).
\]
(The condition $|Y|= |X|-1$ can be replaced by $|Y|< |X|$, but the above is sufficient.)

This generalizes to the inclusions of wide subcategories that occur in Lemmas \ref{lem:subcat_cc} and \ref{lem:subcat_cd}.

\begin{prop}
\label{prop:extn_left_adj_cc}
The inclusions of wide subcategories of Lemma \ref{lem:subcat_cc} induce a commutative diagram of extension functors: 
\[
\xymatrix{
\apsh{\fb^{\times 2}}
\ar@{^(->}[r]
\ar@{^(->}[d]
&
\apsh{\cc_0}
\ar@{^(->}[d]
\\
\apsh{\finj \times \fb} 
\ar@{^(->}[r]
&
\apsh{\cc}.
}
\]
These inclusions admit left adjoints. In particular, the left adjoint $\hcc : \apsh{\cc} \rightarrow \apsh{\cc_0}$
to the extension $\apsh{\cc_0}\rightarrow \apsh{\cc}$ 
has values given on $G \in \ob \apsh{\cc}$ by
\[
\hcc G (W,Z) = H^\finj_0 G (-,Z) (W).
\] 
\end{prop}

\begin{proof}
The extension functors are defined as for the case $\fb \subset \finj$. The left adjoints admit explicit descriptions, analogous to that of $H^\finj_0$. 

The left adjoint $\hcc$ to $\apsh{\cc_0}\rightarrow \apsh{\cc}$ has values given by 
\[
\hcc G (W,Z) 
=
\mathrm{Coker} 
\Big(
\bigoplus_{\substack{ i: U \hookrightarrow W\\ |U|= |W|-1}}
G(U,Z) 
\stackrel{\bigoplus G(i,\id,\init)}{\longrightarrow}
G(W,Z)
\Big).
\]
That this defines a functor with values in $\apsh{\cc_0}$ is verified using Proposition \ref{prop:factor_cc}.
 By inspection, it is isomorphic to $H^\finj_0 G (-,Z) (W)$.

Likewise, the left adjoint to $\apsh{\finj \times \fb}\rightarrow \apsh{\cc}$ has values given by 
\[
(W,Z) \mapsto
\mathrm{Coker} 
\Big(
\bigoplus_{\substack{ (i,j,\alpha)\in \hom_\cc( (U,V),(W,Z))
\\ |U|= |W|-1, |V|= |Z| -1}}
G(U,V) 
\stackrel{\bigoplus G(i,j, \alpha)}{\longrightarrow}
G(W,Z)
\Big).
\]
Again, that this defines a functor with values in $\apsh{\cc_0}$ is verified using Proposition \ref{prop:factor_cc}.
\end{proof}

Similarly, for $\cd$:

\begin{prop}
\label{prop:extn_cd}
The inclusions of wide subcategories of Lemma \ref{lem:subcat_cd} induce a commutative diagram of extension functors: 
\[
\xymatrix{
\apsh{\fb^{\times 2}}
\ar@{^(->}[r]
\ar@{^(->}[d]
&
\apsh{\cd_0}
\ar@{^(->}[d]
\\
\apsh{\finj \times \fb} 
\ar@{^(->}[r]
&
\apsh{\cd}.
}
\]
These inclusions admit left adjoints. In particular, the left adjoint  $\hcd : \apsh{\cd} \rightarrow \apsh{\cd_0}$ to the extension $\apsh{\cd_0}\rightarrow \apsh{\cd}$ 
has values given on $F \in \ob \apsh {\cd}$ by
\[
\hcd F(W,Z) = H^\finj_0 F (-,Z) (W).
\] 
\end{prop}

\begin{proof}
The proof is analogous to that of Proposition \ref{prop:extn_left_adj_cc}. For example, the left adjoint  to 
$\apsh{\finj \times \fb}
\rightarrow \apsh{\cd}$ has values given by:
\[
(W,Z) 
\mapsto
\mathrm{Coker} 
\Big(
\bigoplus_{\substack{ (i,j,\zeta)\in \hom_\cd( (U,V),(W,Z))
\\ |U|= |W|-2, |V|= |Z| -1}}
F(U,V) 
\stackrel{\bigoplus F(i,j, \zeta)}{\longrightarrow}
F(W,Z)
\Big).
\]
\end{proof}

\subsection{Relating $\apsh{\cc}$ and $\apsh{\cd}$} 
\label{subsect:int}

Precomposition with the functor $\dbl$ of Definition \ref{defn:dbl} gives  the (exact) restriction functor $\dbl^* : \apsh{\cd} \rightarrow \apsh{\cc}$.  The strategy of Remark \ref{rem:transfer_cc_to_cd} can be applied to construct  an explicit left adjoint to $\dbl^*$.

First we note:

\begin{lem}
\label{lem:restrict_dbl*}
The functor $\dbl^* : \apsh{\cd} \rightarrow \apsh{\cc}$ restricted to $\apsh{\cd_0} \subset \apsh{\cd}$ takes values in $\apsh{\cc_0} \subset \apsh{\cc}$, i.e., the following diagram commutes:
\[
\xymatrix{
\apsh{\cd_0}
\ar[r]^{\dbl^*}
\ar@{^(->}[d]
&
\apsh{\cc_0}
\ar@{^(->}[d]
\\
\apsh{\cd}
\ar[r]_{\dbl^*}
&
\apsh{\cc}
}
\]
in which the vertical functors are the canonical inclusions.
\end{lem}

The following is an immediate consequence of Proposition \ref{prop:dblelt_left_adj}: 

\begin{prop}
\label{prop:dblelt*_adjunction}
The adjunction $\dblelt : \cc \rightleftarrows \cd_{\hfi(-,-)_1} : \contract$ induces an adjunction:
\[
\contract^* : 
\apsh {\cc}
\rightleftarrows 
\apsh {\cd_{\hfi(-,-)_1}} 
: 
\dblelt^*.
\]
\end{prop}

Restriction along the projection $\cd_{\hfi(-,-)_1} \rightarrow \cd$ induces the exact functor 
\[
\apsh{\cd} \rightarrow \apsh {\cd_{\hfi(-,-)_1}}. 
\]
Left Kan extension gives the left adjoint as follows:

\begin{prop}
\label{prop:rest_cd_elements}
The restriction functor $\apsh{\cd} \rightarrow \apsh {\cd_{\hfi(-,-)_1}}$ has left adjoint 
\[
\apsh {\cd_{\hfi(-,-)_1} } \rightarrow \apsh{\cd} 
\]
that sends a functor $F \in \ob \apsh {\cd_{\hfi(-,-)_1}}$ to the functor:
\[
(X,Y) \mapsto \bigoplus _{\kappa \in \hfi (Y,X) } F ((X,Y), \kappa).
\]
\end{prop}

\begin{defn}
\label{defn:int}
Let $\ind : \apsh{\cc} \rightarrow \apsh{\cd}$ denote the composite:
\[
 \apsh{\cc} \stackrel{\contract^*} {\rightarrow } \apsh {\cd_{\hfi(-,-)_1} } \rightarrow \apsh{\cd}, 
\]
where the second functor is the left adjoint given in  Proposition \ref{prop:rest_cd_elements}.
\end{defn}

The following is immediate:
 
 \begin{prop}
 \label{prop:properties_int}
  The functor $\ind : \apsh{\cc} \rightarrow \apsh{\cd}$ satisfies the following properties:
 \begin{enumerate}
 \item 
it is exact;
 \item 
it restricts to $\apsh{\cc_0} \rightarrow \apsh{\cd_0} \subset \apsh{\cd}$ (i.e., restricted to $\apsh{\cc_0} \subset \apsh{\cc}$, it takes values in $\apsh{\cd_0}\subset \apsh{\cd}$). 
\end{enumerate} 
The value of the  functor $\ind F$, for $F \in \ob \apsh{\cc}$,  on $(U,V) \in \ob \cd$ is given by
\[
\ind F (U,V) 
:= 
\bigoplus_{\kappa \in \hfi (V,U) }
F (U \backslash \kappa (V), V).
\]
 \end{prop}

 Putting together Propositions \ref{prop:dblelt*_adjunction} and \ref{prop:rest_cd_elements}, one obtains:

 \begin{thm}
 \label{thm:int_left_adjoint}
 The functor $\ind : \apsh{\cc} \rightarrow \apsh{\cd}$ is left adjoint to the restriction functor   $ \dbl^* : \apsh{\cd} \rightarrow \apsh{\cc}$.
 
 Moreover, restricting to the full subcategories $\apsh{\cc_0} \subset \apsh{\cc}$ and $\apsh{\cd_0} \subset \apsh{\cd}$ gives the adjunction:
 \[
 \ind[0] : \apsh{\cc_0} \rightleftarrows \apsh{\cd_0} : \dbl^*.
 \]
 \end{thm}
 
\begin{proof}
The first statement follows from Propositions \ref{prop:dblelt*_adjunction} and \ref{prop:rest_cd_elements} and the second is a formal  consequence, using Lemma \ref{lem:restrict_dbl*} and Proposition \ref{prop:properties_int} for the restrictions of $\dbl$ and $\ind$ respectively.
\end{proof}

\begin{cor}
\label{cor:compat_int_H} 
The following diagram commutes up to natural isomorphism:
 \[
 \xymatrix{
 \apsh{\cc}
 \ar[r]^\ind
 \ar[d]_{\hcc}
 &
\apsh{\cd} 
 \ar[d]^{\hcd}
 \\
 \apsh{\cc_0}
 \ar[r]_{\ind[0]}
 &
 \apsh{\cd_0}.
 }
 \]
\end{cor}

\begin{proof}
This is a diagram of left adjoint functors. To establish the commutativity, it suffices to show that the associated diagram of right adjoint functors commutes; this is given by Lemma \ref{lem:restrict_dbl*}.
\end{proof}

\begin{rem}
Corollary \ref{cor:compat_int_H} can also be proved directly,  by using the explicit construction of the left adjoints $\hcc : \apsh{\cc} \rightarrow \apsh{\cc_0}$ and $\hcd : \apsh{\cd} \rightarrow \apsh{\cd_0}$, that kill the images of non-automorphisms of $\finj \times \fb$ in the respective cases (see Propositions \ref{prop:extn_left_adj_cc} and \ref{prop:extn_cd}). The key ingredients in the proof are the  commutation properties given by Propositions \ref{prop:factor_cc} and \ref{prop:factor_cd} respectively.
\end{rem}

\subsection{The representations underlying $\ind$}

Consider the functor $\ind$ as described in Proposition \ref{prop:properties_int}. If one restricts to considering only the $\aut (U) \times \aut (V)$-module structure of $\ind F (U,V)$, then this can be described using induction functors as follows.

For $m,n \in \nat$, fix a bijection $\mathbf{m+ n} \cong \mathbf{m} \amalg  \mathbf{n}$, giving the inclusion $\mathbf{m} \subset \mathbf{m+n}$ and the Young subgroup $\sym_m \times \sym_n \subset \sym_{m+n}$.

\begin{prop}
\label{prop:int_values}
For  $F \in \ob \apsh{\cc}$ and $m,n \in \nat$, there is a natural isomorphism of $\sym_{m+n}\times \sym_n$-modules:
\[
\big(\ind F\big) (\mathbf{m+n},\mathbf{n}) 
\cong 
\kring \sym_{m+n} \otimes_{\sym_m} F(\mathbf{m},\mathbf{n}), 
\]
where $\sym_n$ acts diagonally, using the right action of $\sym_n$ on $\sym_{m+n}$.  

In particular, this only depends on the $\sym_m\times \sym_n$-module $F(\mathbf{m},\mathbf{n})$.
\end{prop}

\begin{proof} 
Write $U = \mathbf{m+n}$ and $V= \mathbf{n}$ and take $\kappa_0 : V \hookrightarrow U$ to be the canonical inclusion. Then the inclusion of $F(\mathbf{m}, \mathbf{n})$ as the direct summand of $\bigoplus_{\kappa \in \hfi (V,U) }
F (U \backslash \kappa (V), V)$ indexed by $\kappa_0$ is $\sym_m$-equivariant and hence induces  a morphism of $\sym_{m+n}$-modules 
\[
\kring \sym_{m+n} \otimes_{\sym_m} F(\mathbf{m},\mathbf{n})
\rightarrow 
\ind F (U,V).
\]
It is straightforward to check that this is an isomorphism. 

It remains to verify that it is $\sym_n$-equivariant. This is left to the reader.
\end{proof}

\begin{exam}
\label{exam:calculate_int}
If $\kring$ is a field of characteristic zero, then the categories of modules over $\sym_m$, $\sym_n$, $\sym_{m+n}$ that intervene, together with the relevant bimodule categories, are semi-simple. Moreover, up to isomorphism, the simple $\kring [\sym_m \times \sym_n]$-modules are given by the exterior tensor products $S^\mu \boxtimes S^\nu$, for $\mu \vdash m$ and $\nu \vdash n$. Hence, to understand the induction functor occurring in Proposition \ref{prop:int_values}, it suffices to consider:
\[
\kring \sym_{m+n} \otimes_{\sym_m} (S^\mu \boxtimes S^\nu)
\]
with the appropriate $\sym_{m+n} \times \sym_n$-module structure. 

As recalled in Section \ref{subsect:bimodules}, $\kring \sym_{m+n} \cong \bigoplus_{\rho \vdash m+n} S^\rho \boxtimes S^\rho$   as a bimodule. From this, one deduces the isomorphism of 
$\sym_{m+n} \times \sym_n$-modules:
\[
\kring \sym_{m+n} \otimes_{\sym_m} (S^\mu \boxtimes S^\nu)
\cong 
\bigoplus_{\rho \vdash m+n}
\bigoplus_{\lambda \vdash n} 
S^\rho \boxtimes (S^\nu \otimes S^\lambda)^{\oplus c^{\rho}_{\mu \lambda}},
\]
where  $c^{\rho}_{\mu \lambda} $ is the Littlewood-Richardson coefficient and  $S^\nu \otimes S^\lambda$ is given the diagonal structure. 

The Kronecker coefficients $k^{\gamma}_{\nu \lambda}$ describe the tensor product,
 i.e., $S^\nu \otimes S^\lambda= \bigoplus_{\gamma \vdash n} (S^\gamma)^{\oplus  k^{\gamma}_{\nu \lambda}}$; this gives 
\[
\kring \sym_{m+n} \otimes_{\sym_m} (S^\mu \boxtimes S^\nu)
\cong 
\bigoplus_{\rho \vdash m+n}
\bigoplus_{\gamma \vdash n}
(S^\rho \boxtimes S^\gamma)^{\oplus \sum_{\lambda \vdash n} c^{\rho}_{\mu \lambda} 
k^{\gamma}_{\nu \lambda}},
\] 
This reduces the calculation to that of fundamental coefficients from the representation theory of the symmetric groups.
\end{exam}

For applications, $m+n$ is usually fixed, say $m+n =N \in \nat$ and one focuses upon an `isotypical component' corresponding to a fixed $\rho \vdash N$. In this case, one has the following, using the definition of the skew representation $S^{\rho/\mu}$ (see Section \ref{subsect:skew_reps}).

\begin{prop}
\label{prop:isotypical_rho}
For $m, n \in \nat$ such that $m+n = N$ and partitions $\mu \vdash m$, $\nu \vdash n$ and $\rho \vdash N$, the isotypical component of  $\kring \sym_{N} \otimes_{\sym_m} (S^\mu \boxtimes S^\nu) $ corresponding to $\rho$ is isomorphic to 
\[
\left\{
\begin{array}{ll}
0 & \mu \not \preceq \rho \\
S^\rho \boxtimes (S^{\rho/\mu} \otimes S^\nu) & \mu \preceq \rho
\end{array}
\right.
\]
as a $\sym_N \times \sym_n$-module, with diagonal $\sym_n$-module structure on 
$S^{\rho/\mu} \otimes S^\nu$.
\end{prop}

\section{Structure results for $\cc$-modules}
\label{sect:psh_more}

This Section considers the structure of $\apsh{\cc}$ and $\apsh{\cc_0}$ in greater detail, introducing techniques that are useful for analysing the principal objects of interest in the paper. This also serves to present some basic examples.

\subsection{Examples for $\cc$ and $\cc_0$}
\label{subsect:exam_cc0}

The structure of the standard projectives on $\cc$ and $\cc_0$ is of importance in the applications. Explicit information on the related $\cc_0$-modules can be derived from the material of Section \ref{subsect:bimodules}.

One has the following identification, which should be compared with that given in Example \ref{exam:proj_finj}:

\begin{prop}
\label{prop:cc_00}
There is an isomorphism in $\apsh{\cc}$:
\[
P^\cc_{(\mathbf{0}, \mathbf{0})} 
\cong 
\Big[
(X,Y) \mapsto \kring \hom_\finj (Y, X)
\Big],
\]
where the right hand side is the  $\kring$-linearization of the functor of Proposition \ref{prop:cc_to_sets}. 

In particular, for $n \in \nat$, $P^\cc_{(\mathbf{0}, \mathbf{0})} (\mathbf{n},\mathbf{n})  \cong \kring \sym_n$ and the canonical morphism $(\mathbf{n}, \mathbf{n}) \rightarrow (\mathbf{n+1}, \mathbf{n+1}) $ corresponds to the morphism of $\sym_n$-bimodules 
$
\kring \sym_n \rightarrow \kring \sym_{n+1}$  
(using the restricted structure on the codomain), induced by $\mathbf{n} \subset \mathbf{n+1}$.
\end{prop}

\begin{proof}
This is a direct verification, extending the partial identification given in Remark \ref{rem:00_not_initial}. 
\end{proof}

On restricting to $\cc_0$, one can identify the standard projectives of the form $P^{\cc_0}_{(\mathbf{0}, \mathbf{d})} $, for $d \in \nat$, 
 in terms of $P^{\cc_0}_{(\mathbf{0}, \mathbf{0})}$, using precomposition with the functor $(- \amalg (\mathbf{d} , \mathbf{0})) : \cc_0 \rightarrow \cc_0$ given by the symmetric monoidal structure of $\cc_0$.
 
\begin{prop}
\label{prop:P_cc0_Od}
For $d\in \nat$, there is a natural isomorphism of functors on $\cc_0$:
\[
P^{\cc_0}_{(\mathbf{0}, \mathbf{d})} 
\cong 
P^{\cc_0}_{(\mathbf{0}, \mathbf{0})} \circ (- \amalg (\mathbf{d} , \mathbf{0})).
\]
\end{prop}

\begin{proof}
This is again a straightforward verification (cf. Example \ref{exam:hom0D_cc0}).
\end{proof}

\begin{rem}
\ 
\begin{enumerate}
\item 
This Proposition can be paraphrased by stating that the projective $P^{\cc_0}_{(\mathbf{0}, \mathbf{d})}$ is obtained from $P^{\cc_0}_{(\mathbf{0}, \mathbf{0})}$ by restriction. 
\item 
The involution $\invcc$ of $\cc_0$ (see Proposition \ref{prop:involution_cc0}) allows $P^{\cc_0}_{(\mathbf{s}, \mathbf{t})}$ to be related to $P^{\cc_0}_{(\mathbf{t}, \mathbf{s})}$, for any $(s,t) \in \nat^{\times 2}$. In particular, Proposition \ref{prop:P_cc0_Od} can be applied to understand  $P^{\cc_0}_{(\mathbf{d}, \mathbf{0})} $. 
\end{enumerate}
\end{rem}

It is interesting to understand the behaviour of the standard projectives of $\apsh{\cc}$ upon restriction to $\cc_0$, or more precisely, on restricting to a connected component of $\cc_0$. For this we fix $(A,B) \in \ob \cc$ and consider the projective $P^\cc_{(A,B)}$.  

The following is immediate:

\begin{lem}
\label{lem:non_triviality}
For $(X,Y) \in \ob \cc$, $P^\cc_{(A,B)}(X,Y) \neq 0$ if and only if $|X|- |Y | \geq |A |- |B|$. 
\end{lem}

Henceforth we restrict to considering the connected component of $\cc_0$ corresponding to $|X|-|Y|= |A|-|B|+d$, for fixed $d \in \nat$. 
That is, we restrict to $\cc_0^{(|A|-|B|+d)}$. 

\begin{prop}
\label{prop:restrict_proj_cc}
For $d \in \nat$, there is an isomorphism:
\[
P^{\cc_0} _{(A \amalg \mathbf{d}, B) } 
/ 
\sym_d 
\stackrel{\cong}{\rightarrow} 
P^\cc_{(A,B)} \downarrow ^{\cc}_{\cc_0^{(|A|-|B|+d)}},
\]
where the coinvariants are formed with respect to the action arising from that of $\sym_d$ on $A \amalg \mathbf{d}$ in $\cc_0$. 

In particular, if $\kring$ is a field of characteristic zero, then  $P^\cc_{(A,B)} \downarrow ^{\cc}_{\cc_0^{(|A|-|B|+d)}}$ is projective in $\apsh{\cc_0}$. 
\end{prop}

\begin{proof}
First we exhibit a morphism 
\[
P^{\cc_0} _{(A \amalg \mathbf{d}, B) } 
\rightarrow  
P^\cc_{(A,B)} \downarrow ^{\cc}_{\cc_0^{(|A|-|B|+d)}}.
\]
This is constructed at the level of the morphism sets of the categories $\cc_0$ and $\cc$. Namely, for $(X,Y) \in \ob \cc_0^{(|A|-|B|+d)}$, there is a natural map
\[
\hom_{\cc_0} ((A \amalg \mathbf{d}, B) ,(X,Y))
\rightarrow 
\hom_{\cc} ((A, B), (X,Y))
\]
that sends a triple $(i : A \amalg \mathbf{d} \hookrightarrow X, j : B \hookrightarrow Y , \alpha)$ to the triple 
$(i|_A : A  \hookrightarrow X, j : B \hookrightarrow Y , \tilde{\alpha})$. Here, for $\alpha : Y \backslash \mathrm{im} (j) \stackrel{\cong}{\rightarrow} X \backslash \mathrm{im} (i)$, $\tilde{\alpha}$ is obtained by composing with the canonical inclusion $X \backslash \mathrm{im} (i) \subset X \backslash \mathrm{im} (i|_A)$.

One checks that, after passing to the $\kring$-linearization, one obtains a natural transformation as claimed. 

The map between $\hom$ sets corresponds to forgetting $i|_{\mathbf{d}}$, whilst retaining the information on the image as a subset of $X$. In particular, it is clear that the map is surjective. It clearly factors across the $\sym_d$-coinvariants and this gives the required isomorphism. 
 
Finally, if $\kring$ is a field of characteristic zero, then the quotient map:
\[
P^{\cc_0} _{(A \amalg \mathbf{d}, B) } 
\twoheadrightarrow 
P^{\cc_0} _{(A \amalg \mathbf{d}, B) } 
/ 
\sym_d 
\]
admits a section (given by the transfer), hence $P^{\cc_0} _{(A \amalg \mathbf{d}, B) } 
/ 
\sym_d $ is a direct summand of a projective, whence it is projective.
\end{proof}

\begin{rem}
If $\kring$ is a field of characteristic zero, Proposition \ref{prop:restrict_proj_cc} shows that the projective $P^\cc_{(A,B)}$ admits an infinite descending filtration in $\apsh{\cc}$ of which the subquotients are projectives of $\apsh{\cc_0}$ (considered as objects of $\apsh{\cc}$ by extension). The top quotient of this filtration is 
\[
P^\cc_{(A,B)} 
\twoheadrightarrow 
P^{\cc_0}_{(A,B)}.
\]
\end{rem}

\subsection{The functor $\circ (- \amalg (\mathbf{0}, \mathbf{1}) )$ on $\apsh{\cc}$}

Precomposition with $- \amalg (\mathbf{0}, \mathbf{1})$ on $\cc$ induces an exact endofunctor of $\apsh{\cc}$. The purpose of this Section is to identify its right adjoint. These functors provide fundamental tools for analysing the structure of the objects that interest us (see Section \ref{sect:further}). However, this material is not required for the proof of the  main results of the paper, so this subsection can be omitted on first reading. 

\begin{lem}
\label{lem:precomp_projectives_cc}
The functor $\circ (- \amalg (\mathbf{0}, \mathbf{1}) ) : \apsh{\cc} \rightarrow \apsh{\cc}$ preserves projectives. More precisely, for $(U, V) \in \ob \cc$, there is an isomorphism:
\[
P^\cc_{(U,V)} \circ (- \amalg (\mathbf{0}, \mathbf{1}) )
\cong 
P^\cc_{(U \amalg \mathbf{1},V)}
\oplus 
\bigoplus_{v \in V}
 P^\cc_{(U,V\backslash\{v\})}. 
\]
\end{lem}

\begin{proof}
This follows from Lemma \ref{lem:shift_cc} by passing to the $\kring$-linearization. 
\end{proof}

\begin{rem}
More is true: as in Proposition \ref{prop:shift_cc_full}, one can identify the functoriality with respect to $(U,V) \in \ob \cc$. This is exploited in the following Definition, which uses the morphisms  given in Notation \ref{nota:shift_morphisms_cc}. (The reader should note that Proposition \ref{prop:shift_cc_full} concerned {\em contravariant} functoriality, whereas the following is {\em covariant}.)
\end{rem}

\begin{defn}
\label{defn:radj}
Let $\radj : \apsh{\cc} \rightarrow \apsh{\cc}$ be the functor defined on $F \in \ob \apsh{\cc}$ by 
\[
\radj F (U,V) 
:=
F (U \amalg \mathbf{1}, V) 
\oplus 
\bigoplus_{v \in V} 
F (U, V \backslash \{ v \} ).
\]

For a morphism $\Phi \in  \hom_\cc ( (A,B), (U,V)  )$ given by the triple   $(f,g,\gamma)$, the induced morphism
\[
F (A \amalg \mathbf{1}, B) 
\oplus 
\bigoplus_{b \in B} 
F (A, B \backslash \{ b \} )
\rightarrow 
F (U \amalg \mathbf{1}, V) 
\oplus 
\bigoplus_{v \in V} 
F (U, V \backslash \{ v \} )
\]
is given by the following:
\begin{enumerate}
\item 
on the summand $F (A, B \backslash \{ b \} )$ is  $F (\Phi_{g(b) \in g(B)}) : F(A, B \backslash\{ b \}) \rightarrow F (U, V\backslash\{ g (b)\})$; 
\item 
on the summand $F (A \amalg \mathbf{1}, B) $ is the morphism $F (A \amalg \mathbf{1}, B)  \rightarrow F (U \amalg \mathbf{1}, V) 
\oplus 
\bigoplus_{v \in V\backslash g(B)} 
F (U, V \backslash \{ v \} )
$ that has components 
\begin{enumerate}
\item 
$F (\Phi_{\mathbf{1}}) : F (A \amalg \mathbf{1}, B)  \rightarrow F (U \amalg \mathbf{1}, V) $;
\item 
 $F (\Phi_{v \not \in g(B)}) : F (A \amalg \mathbf{1}, B)  \rightarrow 
F (U, V \backslash \{ v \} )$.
\end{enumerate}
\end{enumerate}
\end{defn}

\begin{rem}
The verification that $\radj$ defines a functor $\apsh{\cc} \rightarrow \apsh{\cc}$ is left to the reader. 
\end{rem}

\begin{thm}
\label{thm:radj_adjoint}
The functor $\radj : \apsh{\cc} \rightarrow \apsh{\cc}$ is exact and is right adjoint to precomposition with $- \amalg (\mathbf{0}, \mathbf{1})$. 
\end{thm}

\begin{proof}
From its explicit description, it is immediate that $\radj$  is exact. 

The fact that it is the stated right adjoint follows from Lemma \ref{lem:precomp_projectives_cc} and Proposition \ref{prop:shift_cc_full}, taking into account the full functoriality. 
\end{proof}

By restriction, one obtains a functor $\radj :  \apsh{\cc_0} \rightarrow \apsh{\cc_0}$ and this is the right adjoint to the restricted functor $\circ ( - \amalg (\mathbf{0}, \mathbf{1})) : \apsh {\cc_0} \rightarrow \apsh{\cc_0}$. This is  compatible with the  connected components:

\begin{prop}
\label{prop:radj_grading}
The functors $\circ ( - \amalg (\mathbf{0}, \mathbf{1}))$ and $\radj$ restrict to adjoint functors 
\begin{eqnarray*}
\circ ( - \amalg (\mathbf{0}, \mathbf{1})) : \apsh{\cc_0^{(n)} } \rightleftarrows \apsh{\cc_0^{(n+1)} } \ : \ 
\radj 
\end{eqnarray*}
for each $n \in \zed$. 
\end{prop}

\subsection{Relating $\apsh{\cc}$ and $\apsh{\finj^{\times 2}}$}

The forgetful functor $\cc \rightarrow \finj^{\times 2}$ induces an exact restriction functor $\apsh{\finj^{\times 2}} \rightarrow \apsh{\cc}$. Functors in the essential image of this functor have very special properties. For instance, one has:

\begin{prop}
\label{prop:l_adj_applied_to_finj2}
The composite $\apsh{\finj^{\times 2}} \rightarrow \apsh{\cc} \stackrel{\hcc}{ \rightarrow} \apsh{\cc_0}$ takes values in the subcategory $\apsh{\fb^{\times 2}} \subset \apsh{\cc_0}$ of functors on which non-isomorphisms act trivially.
\end{prop}

\begin{proof}
Recall the generators for the morphisms of the  category $\cc$ from Remark \ref{rem:structure_cc}. 
 
In the category $\finj^{\times 2}$, a morphism of the form $(X,Y) \rightarrow (X\amalg \mathbf{1}, Y\amalg \mathbf{1})$ factors as the composite $(X,Y) \rightarrow (X\amalg \mathbf{1}, Y) \rightarrow (X\amalg \mathbf{1}, Y\amalg \mathbf{1})$. From this one deduces readily that, after applying the left adjoint to a functor in the image of $\apsh{\finj^{\times 2}} \rightarrow \apsh{\cc}$, the morphism $(X,Y) \rightarrow (X\amalg \mathbf{1}, Y\amalg \mathbf{1})$ acts trivially. The result follows.
\end{proof}

Below we are  interested in functors in $\apsh{\cc_0}$ that arise from $\apsh{\finj^{\times 2}}$ by restricting along $\cc_0 \hookrightarrow \cc$. 
 For example,  one can consider the external tensor product 
\[
\boxtimes : \apsh{\finj} ^{\times 2} \rightarrow \apsh{\finj^{\times 2}},
\]
compose this with the functor to $\apsh{\cc}$ and then restrict to $\apsh{\cc_0}$.

\begin{exam}
\label{exam:proj_cc0}
The following should be compared with Example \ref{exam:proj_finj}.

For  $(X,Y) \in \ob \cc_0$, one has the projective $P^{\cc_0} _{(X,Y)} \in \ob \apsh{\cc_0}$ and the projectives $P^\finj_X$ and $P^\finj_Y$ of $\apsh{\finj}$. Forming the exterior tensor product and restricting to $\apsh{\cc_0}$ gives the functor $P^\finj_X\boxtimes P^\finj_Y$ (omitting the restriction from the notation).

Yoneda's lemma gives a morphism in $\apsh{\cc_0}$:
\[
P^{\cc_0} _{(X,Y)}
\rightarrow 
P^\finj_X\boxtimes P^\finj_Y
\]
that is an isomorphism when evaluated on $(X,Y)$.

This is not surjective: for instance, $P^{\cc_0} _{(X,Y)}(U,V)$ is zero if $|U|-|V| \neq |X|- |Y|$, by the decomposition of $\cc_0$ into connected components; this property does not hold for  $P^\finj_X\boxtimes P^\finj_Y$ evaluated on $(U,V)$. The solution is to restrict further to $\cc_0^{(n)}$, where $n= |X|-|Y|$. One then obtains the surjection:
\[
P^{\cc_0} _{(X,Y)}
\twoheadrightarrow 
P^\finj_X\boxtimes P^\finj_Y
\downarrow_{\cc_0^{(n)}}.
\]

This is not bijective in general. More precisely, for $(U,V) \in \ob \cc_0^{(n)}$, 
$$P^{\cc_0} _{(X,Y)}(U,V) 
\twoheadrightarrow 
P^\finj_X\boxtimes P^\finj_Y
(U,V)
$$
is an isomorphism if and only if $|U|-|X| \in \{0,1 \}$. This gives a precise sense in which $P^\finj_X\boxtimes P^\finj_Y\downarrow_{\cc_0^{(n)}}$ approximates $P^{\cc_0} _{(X,Y)}$.
\end{exam}

\begin{exam}
\label{exam:sgn_cc0}
In this example, we take $\kring$ a field of characteristic  $\neq 2$. One has the `orientation' $\fb$-module $\ori$ given by 
\[
X \mapsto \ori (X) := \Lambda^{|X|} (\kring X),
\]
the top exterior power evaluated on the $\kring$-linearization of $X$. Thus $\ori (\mathbf{n})$ identifies with the signature representation of $\sym_n$, for $n \in \nat$.

This is {\em not} the restriction of an object of $\apsh{\finj}$; this is because the natural  map $\ori (X) \rightarrow \ori(X \amalg \mathbf{2})$  does not map to the $\aut (\mathbf{2})$-invariants (since $\kring$ is not of characteristic two). As a result, one cannot simply define an object of $\apsh{\cc}$ by forming $\ori \boxtimes \ori$ and restricting. 

However, there is an object of $\apsh {\cc_0}$ with underlying object in $\apsh{\fb \times \fb}$ given by
$
(X,Y) \mapsto \ori(X) \otimes \ori(Y).
$ 
As in Remark \ref{rem:structure_cc},  the additional structure over $\cc_0$ is specified by the behaviour of the morphisms $(X, Y) \rightarrow (X \amalg \mathbf{1} , Y \amalg \mathbf{1})$. This is given by the map 
\[
\ori(X) \otimes \ori(Y)
\rightarrow 
\ori(X\amalg \mathbf{1})  \otimes \ori(Y \amalg \mathbf{1})
\]
corresponding to the identifications $\ori(X) = \ori(X\amalg \mathbf{1})   \downarrow ^{\aut (X\amalg \mathbf{1})} _{\aut (X)}$ and respectively for $Y$.

That this defines an object of $\cc_0$ follows  since the commutativity condition now involves the action of $\aut(\mathbf{2})$ on $(X \amalg \mathbf{2}, Y \amalg \mathbf{2})$, acting on {\em both} factors, giving the required $\aut (\mathbf{2})$-invariance.
\end{exam}

 \section{The signed variant $\ce$}
 \label{sect:signed}
 
This Section introduces the structure that allows us to treat {\em antisymmetrization}.  
  Throughout the Section, $\kring$ is a commutative unital ring.  For $\calc$ a category, $\kring \calc$ denotes the associated $\kring$-linear category.

\subsection{The $\kring$-linear category $\ce$} 
This Section introduces a signed variant $\ce$  of $\cd$ by imposing an anticommutativity condition. This requires working with categories enriched in $\kring$-modules.
 
 \begin{nota}
 For an object $(X,Y)$ of $\cd$ (i.e., pair of finite sets) denote by 
 \[
 \iota_{(X,Y)} : (X,Y) \hookrightarrow (X \amalg \mathbf{2}, Y \amalg \mathbf{1}) 
 \]
 the  morphism in $\cd$ given by the triple $(i, j, \alpha)$ where $i$, $j$ are the canonical inclusions and $\alpha$ is the canonical isomorphism $\mathbf{2} \times \mathbf{1} \cong \mathbf{2}$. 
 
 Likewise denote by $\tau_{(X,Y)} : (X,Y) \hookrightarrow (X \amalg \mathbf{2}, Y \amalg \mathbf{1})$ the morphism given by $(i, j , \tau \alpha)$, where $\tau$ is the non-identity automorphism of $\mathbf{2}$, namely $\tau = (12)$. 
 \end{nota}
 
 \begin{rem}
 \label{rem:iota_tau}
 \ 
\begin{enumerate}
\item
 The morphisms $\iota_{(X,Y)} $ and $\tau_{(X,Y)}$ are related by the commutative diagram in $\cd$:
 \[
 \xymatrix{
 (X,Y)
 \ar[r]^(.4){\iota_{(X,Y)}}
 \ar[rd]_{\tau_{(X,Y)}}
 &
 (X \amalg \mathbf{2}, Y \amalg \mathbf{1}) 
 \ar[d]^\cong
 \\
 &
 (X \amalg \mathbf{2}, Y \amalg \mathbf{1}) 
 }
 \]
 in which the vertical isomorphism is induced by $\tau : \mathbf{2} \stackrel{\cong}{\rightarrow} \mathbf{2}$.
\item 
The morphisms $\iota_{(X,Y)} $ and $\tau_{(X,Y)}$ are the only elements of $\hom_{\cd} ((X,Y), (X \amalg \mathbf{2}, Y \amalg \mathbf{1}))$ with underlying map in $\finj^{\times 2}$ given by the canonical inclusions $(i,j)$.
\end{enumerate} 
 \end{rem}
 
The following is clear (cf. Remark \ref{rem:structure_cd}):

 \begin{lem}
\label{lem:gen_cd}
 The smallest subcategory of $\cd$ containing $\finj \times \fb \subset \cd$ and the morphisms $\iota_{(X,Y)}$ for all $(X,Y) \in \ob \cd$ is $\cd$ itself.
 \end{lem}

 \begin{defn}
 \label{defn:ce}
 Let $\ce$ be the $\kring$-linear category with objects pairs of finite sets and with $\hom_{\ce} (-, -)$ defined by the quotient of $\hom_{\kring \cd}(-, -)$ by the two-sided ideal  generated by the elements 
 $ 
 \iota_{(X,Y)}  + \tau_{(X,Y)}
 $
 for $(X,Y) \in \ob \cd$, which is equivalent to imposing the relation 
 $\iota_{(X,Y)}  \equiv - \tau_{(X,Y)}$. 
 
 Let $\ce_0 \subset \ce$ be the wide subcategory of $\ce$ given by the image of $\kring \cd_0$. 
 \end{defn}
 
 \begin{rem}
 By construction, there is a commutative diagram of $\kring$-linear functors:
 \[
 \xymatrix{
 \kring \cd_0
 \ar@{^(->}[r]
 \ar[d]
 &
 \kring \cd
 \ar[d]
 \\
 \ce_0
 \ar@{^(->}[r]
 &
 \ce
 }
 \]
 in which the horizontal functors are inclusions of wide subcategories and the vertical functors are the quotient functors (that are the identity on objects).
 
By Lemma \ref{lem:subcat_cd} there is a commutative diagram of $\kring$-linear inclusions of wide subcategories:
  \begin{eqnarray}
  \label{eqn:wide_ce}
 \xymatrix{
 \kring (\fb^{\times 2})
 \ar@{^(->}[r]
 \ar@{^(->}[d]
 &
 \ce_0
 \ar@{^(->}[d]
 \\
 \kring(\finj \times \fb)
 \ar@{^(->}[r]
 &
 \ce.
 }
 \end{eqnarray}
 \end{rem}
 
 \begin{nota}
 Denote by $\ce\dash\modules$ (respectively $\ce_0\dash\modules$)  the category of $\kring$-linear functors from $\ce$ (respectively $\ce_0$) to $\kring$-modules.  
 \end{nota}

 \begin{prop}
 \label{prop:ce_mod_full_subcat}
 The category $\ce\dash\modules$ (respectively $\ce_0\dash\modules$) identifies with the  full subcategory of $\apsh{\cd}$ (resp. $\apsh {\cd_0}$) of functors $F$ for which 
  $
 F (\iota_{(X,Y)}) = - F(\tau_{(X,Y)})
  $ 
 for all $(X,Y) \in \ob \cd$. 
 \end{prop}
 
 \begin{proof}
 This is an immediate consequence of the construction of the categories $\ce$ and $\ce_0$ from $\kring \cd$ and $\kring\cd_0$ respectively.
 \end{proof}

\begin{prop}
\label{prop:ce_adjunctions}
The inclusions of wide subcategories of diagram (\ref{eqn:wide_ce}) induce extension functors that fit into the  commutative diagram:
\begin{eqnarray}
\label{eqn:extn_ce}
\xymatrix{
 \apsh{\fb^{\times 2}}
 \ar@{^(->}[r]
 \ar@{^(->}[d]
 &
 \ce_0\dash\modules
 \ar@{^(->}[d]
   \ar@{^(->}[r]
 &
 \apsh{\cd_0} 
 \ar@{^(->}[d]
 \\
 \apsh{\finj \times \fb}
 \ar@{^(->}[r]
 &
 \ce\dash\modules
  \ar@{^(->}[r]
  &
  \apsh{\cd},
 }
\end{eqnarray}
in which the horizontal composites are the 
 extension functors of Proposition \ref{prop:extn_cd}.

Moreover, the functors $\ce_0\dash\modules \hookrightarrow \ce\dash\modules$ and $\apsh{\finj \times \fb} \hookrightarrow \ce\dash\modules$ 
of diagram (\ref{eqn:extn_ce})  admit left adjoints respectively:
\begin{eqnarray*}
\hce &:&  \ce\dash\modules \rightarrow \ce_0\dash\modules
\\
&& \ce\dash\modules \rightarrow  \apsh{\finj \times \fb}.
\end{eqnarray*}

The values of $\hce$ are given on $F \in \ob \ce \dash\modules$ by
\[
\hce F (W,Z) = H^{\finj}_0 F (- , Z) (W).
\]
\end{prop} 
 
 \begin{proof}
The proof is analogous to those of Propositions \ref{prop:extn_left_adj_cc} and \ref{prop:extn_cd}.
 \end{proof}
 
 \subsection{Passing to $\ce\dash\modules$}

\begin{prop}
\label{prop:lad}
The inclusion $\ce\dash\modules \hookrightarrow \apsh{\cd}$ admits a left adjoint
\[
\lad :  \apsh{\cd} \rightarrow \ce\dash\modules.
\]
Moreover, the following properties are satisfied:
\begin{enumerate}
\item
The functor $\lad$ is right exact and preserves projectives.
\item 
The functor $\lad$ restricts to $\lad : \apsh{\cd_0} \rightarrow \ce_0 \dash \modules$ that is left adjoint to the inclusion $\ce_0\dash\modules \hookrightarrow \apsh{\cd_0}$.
\item 
The adjunction unit induces a canonical surjection $F \twoheadrightarrow \lad F$ in $\apsh {\cd}$. 
\end{enumerate}
\end{prop} 
 
 \begin{proof}
 There is a $\kring$-linear bifunctor 
 $
 B : \kring\cd\op \otimes \ce \rightarrow \kring\dash\modules
 $
given as the composite:
\[
\kring\cd\op \otimes \ce
\rightarrow 
(\ce) \op \otimes \ce 
\stackrel{\hom_{\ce} (-,-)}{\longrightarrow}
 \kring\dash\modules,
\] 
 where the first functor is given by the quotient  $\kring \cd \rightarrow \ce$. 
 
 Using the $\kring$-linearization $\cd \rightarrow \kring \cd$, this gives a functor $\cd \op \rightarrow \ce\dash\modules$, denoted again by $B$ (abusing notation). Then, by standard arguments, the left adjoint $\lad$ is given by the functor 
 \[
 B \otimes _\cd - : \apsh{\cd} \rightarrow \ce\dash\modules.
 \] 
 It is clear that this restricts as stated and that the adjunction unit is surjective. 
 
The remaining property follows since $\lad$ is left adjoint to an exact functor. 
 \end{proof}

\begin{rem}
 The notation $\lad$ is chosen to reflect the fact that $\lad F$ is a natural {\em quotient} of $F$ obtained by imposing {\em antisymmetry} (in the precise sense exhibited in the proof of Proposition \ref{prop:lad}).
\end{rem}

\begin{exam}
\label{exam:F_cc}
Consider the functor $F \in \ob \apsh {\cc}$ defined by 
\[
F (X,Y) = 
\left\{
\begin{array}{ll}
\kring & |X|=|Y|=0 \mbox{ or } |X|=|Y|=1
\\
0 & \mbox{otherwise},
\end{array}
\right.
\]
with all possible non-trivial structure morphisms the identity map of $\kring$. Applying the functor $\ind : \apsh{\cc} \rightarrow \apsh{\cd}$ yields $G:=\ind F$ that is supported on the objects $(\mathbf{2}, \mathbf{1})$ and $(\mathbf{0}, \mathbf{0})$, with 
$
G (\mathbf{2}, \mathbf{1}) = \kring \sym_2$ and 
$G (\mathbf{0}, \mathbf{0}) = \kring$.
 Then $\lad G (\mathbf{0}, \mathbf{0}) = \kring$ and $\lad G (\mathbf{2}, \mathbf{1}) \cong \kring$, corresponding (as an $\aut_\cd(\mathbf{2}, \mathbf{1})$-module) to the signature representation of $\sym_2 \cong \aut_\cd(\mathbf{2}, \mathbf{1})$.

There is a short exact sequence in $\apsh{\cd}$:
\[
0
\rightarrow 
G(\mathbf{2},\mathbf{1})
\rightarrow 
G
\rightarrow 
G(\mathbf{0},\mathbf{0})
\rightarrow 
0
\]
where $G(\mathbf{2},\mathbf{1})$ and $G(\mathbf{0},\mathbf{0})$ are considered as functors by extension in the obvious way.
 In particular $\lad G(\mathbf{2},\mathbf{1}) = G(\mathbf{2},\mathbf{1})$ and $\lad G(\mathbf{0},\mathbf{0})= G(\mathbf{0},\mathbf{0})$ (cf. Corollary \ref{cor:F_ladF_iso} below). This shows that the functor $\lad$ is not exact since the sequence 
\[
\lad G(\mathbf{2},\mathbf{1}) = G(\mathbf{2},\mathbf{1})
\rightarrow 
\lad G
\rightarrow 
\lad G(\mathbf{0},\mathbf{0})= G(\mathbf{0},\mathbf{0})
\rightarrow 
0
\] 
is not short exact.
\end{exam}

 The left adjoint $\lad :  \apsh{\cd} \rightarrow \ce\dash\modules$ has a more concrete description, given using the following: 
 
 \begin{lem}
 \label{lem:iota+tau}
 Suppose that $(X,Y) \in \ob \cd$ and $X'\subset X$, $Y'\subset Y'$ with $|X'|= |X|-2$ and $|Y'|= |Y|-1$. 
 Choose an isomorphism $(X,Y)=(X' \amalg \mathbf{2}, Y' \amalg \mathbf{1})$ extending the inclusions. The morphism
 \[
 \iota_{(X',Y')} + \tau_{(X',Y')} \in \hom_{\kring \cd} ((X',Y'), (X,Y) )
 \]
obtained by composition, is independent of the choice of isomorphism.
 \end{lem}
 
 \begin{proof}
 As in Remark \ref{rem:iota_tau}, there are two possible choices for the isomorphism and these are related by $\aut(\mathbf{2})$. It follows that the sum $\iota_{(X',Y')} + \tau_{(X',Y')} $ is independent of the choice. 
 \end{proof}
 
 \begin{prop}
 \label{prop:lad_explicit}
 For $F \in \ob \apsh{\cd}$ and $(X,Y) \in \ob \ce$, 
 \[
 \lad F (X,Y) 
 = 
 \mathrm{Coker} 
 \Big(
 \bigoplus_{\substack{X' \subset X \\ Y' \subset Y} }
 F (X',Y') 
 \stackrel{\iota + \tau} {\rightarrow }
 F(X,Y) 
 \Big),
 \]
 where the sum is over subsets such that $|X'|= |X|-2$ and $|Y'|= |Y|-1$ and the morphism $\iota + \tau$ is given by  Lemma \ref{lem:iota+tau}.
 \end{prop}
 
 \begin{proof}
 It suffices to show that the right hand side defines a functor from $\apsh{\cd}$ to itself, since it then clearly takes values in $\ce\dash\modules$ and is the universal such construction.
 
 This follows by using properties of the category $\cd$, in particular Lemma \ref{lem:gen_cd}. Using the commutation properties given by Proposition \ref{prop:factor_cd} one reduces to working with $\cd_0$. The result then follows by standard arguments. 
 \end{proof}
 
\begin{cor}
\label{cor:F_ladF_iso}
Suppose that $F \in \ob \apsh {\cd}$ is in the image of the extension functor $\apsh {\finj\times \fb} \hookrightarrow \apsh {\cd}$. Then the canonical surjection 
$
F \twoheadrightarrow \lad F
$ 
is an isomorphism. 

In particular, this applies if $F$ is supported in a single bidegree (i.e., $F(\mathbf{m}, \mathbf{n})$ is zero for all but one  pair $(m,n) \in \nat^{\times 2}$). 
\end{cor} 
 
\begin{proof}
The hypothesis implies that all the morphisms $\iota + \tau$ are zero (using the notation employed in Proposition \ref{prop:lad_explicit}), whence the result.
\end{proof}

\begin{rem}
\label{rem:lad_values}
The value $\lad F (X,Y)$ as an $\aut(X) \times \aut(Y)$-module depends only upon 
\begin{enumerate}
\item 
the $\aut (X) \times \aut(Y)$-module $F(X,Y)$; 
\item 
the  $\aut (X') \times \aut(Y')$-module $F(X',Y')$ for some $X' \subset X$ and $Y' \subset Y$ such that $|X'|= |X|-2$ and $|Y'|= |Y|-1$;
\item 
the structure morphism $\iota + \tau : F(X',Y') \rightarrow F(X,Y)$, which is $\aut (X') \times \aut(Y')$-equivariant, with respect to the restricted module structure on $F(X,Y)$.
\end{enumerate}
\end{rem}

In the application of these results, we apply  $\hce : \ce \dash\modules \rightarrow \ce_0\dash\modules$ of Proposition \ref{prop:ce_adjunctions} to a functor in the image of  $\lad : \apsh{\cd} \rightarrow \ce\dash\modules$.   The following Corollary shows that it is equivalent to first apply the left adjoint $\apsh{\cd} \rightarrow \apsh{\cd_0}$ of Proposition \ref{prop:extn_cd} and then to apply $\lad$. This is an important step in reducing the complexity of calculations.

 \begin{cor}
 \label{cor:lad_cd0}
 The following diagram of left adjoint functors commutes up to natural isomorphism: 
 \[
 \xymatrix{
 \apsh {\cd} 
 \ar[r]^\lad 
\ar[d]_{\hcd}
&
\ce\dash\modules 
\ar[d]^{\hce}
\\
 \apsh {\cd_0} 
 \ar[r]_\lad 
 &
 \ce_0\dash\modules.
   }
 \]
 \end{cor}

 \begin{proof}
 This follows from the restriction assertion in Proposition \ref{prop:lad} and the unicity of left adjoints. 
 \end{proof}
 
 \begin{rem}
 \ 
 \begin{enumerate}
 \item 
  An alternative proof of Corollary \ref{cor:lad_cd0} is to use the explicit description of $\lad$ given by Proposition \ref{prop:lad_explicit} in conjunction with the commutation properties given by Proposition \ref{prop:factor_cd}. The  details are left to the reader. 
 \item 
 This result will be applied in conjunction with Corollary \ref{cor:compat_int_H}, to study $\hce$ applied to an object in the image of the composite $\lad \ind : \apsh{\cc} \rightarrow \ce\dash\modules$. These results allow one to first calculate $\hcc$ and then apply $\lad \ind$. 
 \end{enumerate}
 \end{rem}

\section{Composing $\ind$ and $\lad$}
\label{sect:compose}

In this Section,  the composite $\lad \ind$ is investigated. 

\subsection{The composition} 
For the main  application, we need to understand  the composite:
\[
\apsh{\cc} 
\stackrel{\ind}{\rightarrow}
\apsh {\cd}
\stackrel{\lad}{\rightarrow}
\ce\dash\modules
\]
together with its restriction 
$
\apsh{\cc_0} 
\stackrel{\ind[0]}{\rightarrow}
\apsh {\cd_0}
\stackrel{\lad}{\rightarrow}
\ce_0\dash\modules$.
  This relies upon combining the explicit description of $\ind$ as given in Proposition \ref{prop:properties_int} with that of $\lad$ given by Proposition \ref{prop:lad_explicit}.

 For this, we fix $(X,Y) \in \ob \cd$; the key is the following Lemma:

\begin{lem}
\label{lem:lambda_epsilon}
Let $X' \subset X$ and $Y' \subset Y$ be such that $|X'| = |X| -2$ and $|Y'|= |Y|-1$, with a fixed choice of bijection $\mathbf{2} \cong X \backslash X'$. Then, given $\lambda \in \hfi (Y', X')$, for $\epsilon \in \{1, 2 \}$ there exist unique maps $\lambda^\epsilon \in \hfi (Y,X)$ that fit into the commutative diagram:
\[
\xymatrix{
Y' \ar[d]_{\lambda}
\ar@{^(->}[r]
&
Y 
\ar[d]^{\lambda^\epsilon}
&
Y \backslash Y' \cong \mathbf{1}
\ar@{_(->}[l]
\ar[d]^\epsilon 
\\
X' 
\ar@{^(->}[r]
&
X 
&
X \backslash X' \cong \mathbf{2}
\ar@{_(->}[l],
}
\]
in which the map labelled $\epsilon$ corresponds to $1 \mapsto \epsilon \in \mathbf{2}$ with respect to the given isomorphisms. These maps are related by 
$
\lambda^1 = \tau \lambda^2
$, 
where $\tau \in \aut (X)$ is the unique non-identity automorphism of $X$ fixing $X'$.

Moreover there is a unique morphism in 
$
\iota^\epsilon : 
(X' \backslash \mathrm{im} \lambda , Y') 
\rightarrow 
(X \backslash \mathrm{im} \lambda^\epsilon  , Y)
$ in $\cc_0$ 
with underlying injective maps $ X' \backslash \mathrm{im} \lambda  \hookrightarrow X \backslash \mathrm{im} \lambda^\epsilon$ (induced by $X' \subset X$) and $Y' \subset Y$. 

These fit into a commutative diagram in $\cc$:
\[
\xymatrix{
(X' \backslash \mathrm{im} \lambda , Y') 
\ar[rr]^{\iota^1}
\ar[d]_{\iota^2}
&&
(X \backslash \mathrm{im} \lambda^1  , Y)
\ar[d]
\\
(X \backslash \mathrm{im} \lambda^2  , Y)
\ar[r]
&
(X,Y) 
\ar[r]_{\tau}
&
(X,Y)
}
\]
in which the unlabelled morphisms are induced by the inclusions $X \backslash \mathrm{im} \lambda^\epsilon  \subset X$.
\end{lem}

\begin{proof}
It is clear that the first commutative diagram uniquely determines the morphisms $\lambda^\epsilon$ (the labelling depends upon the choice of the isomorphism $\mathbf{2} \cong X \backslash X'$). The relation $\lambda^1 = \tau \lambda^2$ follows immediately from the construction.

It remains to check that one obtains the morphisms $\iota^1$ and $\iota^2$  of $\cc_0$ as claimed. Fix $\epsilon$. First one observes that $X' \subset X$ restricts to an inclusion $X' \backslash \mathrm{im} \lambda  \hookrightarrow X \backslash \mathrm{im} \lambda^\epsilon$, since the only difference between the images of $\lambda$ and $\lambda^\epsilon$ is an element of $X \backslash X'$, by construction. Then, since the two complements are both of cardinal one, the given injections uniquely determine a morphism of $\cc$.

That these fit into the given commutative diagram in $\cc$ is a direct verification. 
\end{proof}

In the situation of the Lemma, for any $F \in \ob \apsh{\cc}$, one has the induced morphisms of $\kring$-modules:
\begin{eqnarray}
\label{eqn:F_iota}
F (X \backslash \mathrm{im} \lambda^1  , Y)
\stackrel{F(\iota^1)}{\leftarrow}
F(X' \backslash \mathrm{im} \lambda , Y')
\stackrel{F(\iota^2)}{\rightarrow} 
F (X \backslash \mathrm{im} \lambda^2  , Y).
\end{eqnarray}

\begin{prop}
\label{prop:lad_int}
For $F \in \ob \apsh{\cc}$ and $(X,Y) \in \ob \ce$, there is a natural isomorphism between $\Big( \lad \ind F \Big) (X,Y)$ and the cokernel of 
\[
\bigoplus_{\substack{
X' \subset X \\
Y' \subset Y \\
|X'| = |X| -2\\
|Y'|= |Y|-1
}}
\bigoplus _{\lambda \in \hfi (Y',X') }
F (X' \backslash \mathrm{im} \lambda , Y')
\rightarrow 
\bigoplus _{\kappa \in \hfi (Y, X) }
F (X \backslash \mathrm{im} \kappa  , Y),
\]
where the   map on the summand $F (X' \backslash \mathrm{im} \lambda , Y')$  
is given by (\ref{eqn:F_iota}).

The naturality with respect to $(X,Y)$ is determined by that of $\ind F$.
\end{prop}

\begin{proof}
This follows from the description of $\ind$ given in Proposition \ref{prop:properties_int} and that of $\lad$ given by Proposition \ref{prop:lad_explicit}, using the above identifications.
\end{proof}

\subsection{Underlying geometric structure and first examples}

It is useful to identify the `geometric' structure underlying the colimit appearing in Proposition \ref{prop:lad_int};  this is derived from Lemma \ref{lem:lambda_epsilon}.

Namely, as above take $(X,Y)\in \ob \ce$ and consider the bipartite graph $\gph_{X,Y}$ with 
\begin{enumerate}
\item 
white vertices given by the elements of $\hfi (Y, X)$; 
\item 
black vertices gives by triples $(X', Y', \lambda)$ where $\lambda \in \hfi (Y', X')$ and  $X' \subset X$, $Y' \subset Y$ such that $|X'|= |X|-2$ and $|Y'|= |Y|-1$. 
\end{enumerate}
Each black vertex is connected by edges to precisely two white vertices, namely those corresponding to $\lambda^1$, $\lambda^2$ in Lemma \ref{lem:lambda_epsilon}

The bipartite graph $\gph_{X,Y}$ is empty if $|X |< |Y|$. Moreover,  $\gph_{X,Y}$ is connected unless $|X|=|Y|>1$. The group $\aut (X) \times \aut (Y)$ acts on the graph in the obvious way. 

\begin{rem}
The black vertices index the summands of the domain and the white vertices those of the  codomain  of the map appearing in Proposition \ref{prop:lad_int}.
\end{rem}

\begin{exam}
\label{exam:gph}\ 
\begin{enumerate}
\item 
$\gph _{\mathbf{1}, \mathbf{1}}$ has a single vertex and this is white; 
\item 
$\gph_{\mathbf{2}, \mathbf{1}}$ has one black vertex and two white vertices:
 
 \begin{tikzpicture}[scale = 1]
 \draw (0,0) -- (1,0) -- (2,0);
 \draw [fill=white] (0,0) circle [radius = .1];
  \draw [fill=black] (1,0) circle [radius = .1];
  \draw [fill=white] (2,0) circle [radius = .1];
\end{tikzpicture};
\item 
$\gph_{\mathbf{3}, \mathbf{1}}$ has three black vertices and three white vertices:
 
  
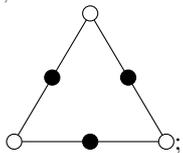
\begin{tikzpicture}[scale = 1]
 \draw (0,0) -- (1,0) -- (2,0);
  \draw (0,0)--(.5,.85) -- (1, 1.7);
    \draw (1, 1.7)--(1.5,.85) -- (2,0);
    \draw [fill=black] (1.5,.85) circle [radius = .1]; 
    \draw [fill=white] (0,0) circle [radius = .1];
  \draw [fill=black] (1,0) circle [radius = .1];
  \draw [fill=white] (2,0) circle [radius = .1];
  \draw [fill=white] (1,1.7) circle [radius = .1];
   \draw [fill=black] (.5,.85) circle [radius = .1];
\end{tikzpicture};
\item 
in general, for $s\geq 1$, $\gph_{\mathbf{s}, \mathbf{1}}$ corresponds to an obvious notion of bipartite $1$-skeleton of a $(s-1)$-simplex.
\end{enumerate}
Life becomes more complicated for $\gph_{\mathbf{s},\mathbf{t}}$ with $s> t>1$. As an example, the reader is encouraged to check that  $\gph_{\mathbf{4}, \mathbf{2}}$ is a bipartite  version of the $1$-skeleton of the cuboctahedron. 
\end{exam}

Since the black vertices of $\gph_{X,Y}$ are in bijection with the `edges' between two white vertices on which they sit, these can be omitted:

\begin{nota}
For $(X,Y)\in \ob \cc$ let  $\gbar_{X,Y}$ be the graph obtained from $\gph_{X,Y}$ by replacing each  \begin{tikzpicture}[scale = .5]
 \draw (0,0) -- (1,0) -- (2,0);
 \draw [fill=white] (0,0) circle [radius = .1];
  \draw [fill=black] (1,0) circle [radius = .1];
  \draw [fill=white] (2,0) circle [radius = .1];
\end{tikzpicture} by 
 \begin{tikzpicture}[scale = .5]
 \draw (0,0) -- (2,0);
 \draw [fill=white] (0,0) circle [radius = .1];
  \draw [fill=white] (2,0) circle [radius = .1];
\end{tikzpicture}.
\end{nota}

By construction, one has:

\begin{lem}
\label{lem:vertices_gbar}
If $s \geq t$, the vertices of $\gbar_{\mathbf{s},\mathbf{t}}$ are in bijection with $\sym_s / \sym_{s-t}$.
\end{lem}

This establishes a relation between $\gbar_{\mathbf{s}, \mathbf{t}}$ and a subgraph of the Cayley graph of $\sym_s$ that is introduced below.  For this, the following standard elementary result is used:

\begin{lem}
\label{lem:transpositions}
Let $2 \leq s$ be a natural number and $\tau = (ij) \in \sym_s$. Then, for $\phi \in \sym_s$, 
\[
\tau \phi = \phi \tau^\phi, 
\]
where $\tau^\phi$ is the transposition $(\phi^{-1}(i) \ \phi^{-1}(j) )$.
\end{lem}

If $t \leq s$, the canonical inclusion $\mathbf{t} \subset \mathbf{s}$ gives the decomposition $\mathbf{s} = \mathbf{t} \amalg (\mathbf{s-t})$ and hence the inclusion of the associated Young subgroup: $\sym_t \times \sym_{s-t} \subset \sym_s$.

\begin{nota}
For $s \geq t \in \nat$, denote by
\begin{enumerate}
\item 
$\ghat_{\mathbf{s},\mathbf{t}}$ the graph with vertex set $\sym_s$ and an edge between $\phi$ and $\psi$ if and only if $ \psi \phi^{-1}$ is a transposition; this  corresponds to a pair $(\tau, \phi)$ such that $\psi = \tau \phi$; 
\item 
$\gtilde_{\mathbf{s},\mathbf{t}}$ the subgraph of $\ghat_{\mathbf{s},\mathbf{t}}$ containing all vertices, with edges given by pairs $(\tau , \phi)$ such that $\tau^\phi \not \in \sym_t \times \sym_{s-t}$. 
\end{enumerate}
\end{nota}

The following is clear: 

\begin{lem}
\label{lem:gtilde_right_action}
For $s \geq  t$, the right action of $\sym_{s-t}$ on $\sym_s$ induces a right action of $\sym_{s-t}$ on $\ghat_{\mathbf{s},\mathbf{t}}$ and this restricts to a right action on $\gtilde_{\mathbf{s},\mathbf{t}}$.
\end{lem}

Using this, one obtains the following:

\begin{prop}
\label{prop:sym_t_cover_gbar}
For $s \geq t$, there is a $\sym_{s-t}$-principal covering 
$
\gtilde_{\mathbf{s},\mathbf{t}}
\twoheadrightarrow 
\gbar_{\mathbf{s},\mathbf{t}}
$ 
with map on vertices given by the canonical surjection $\sym_s \twoheadrightarrow \sym_s / \sym_{s-t}$. 
\end{prop}

\begin{proof}
It is clear that the action of $\sym_{s-t}$ on $\gtilde_{\mathbf{s},\mathbf{t}}$ is free, hence it suffices to prove that the quotient graph 
$\gtilde_{\mathbf{s},\mathbf{t}}/ \sym_{s-t}$ is isomorphic to $\gbar_{\mathbf{s},\mathbf{t}}$.

On vertices this is clear; hence it suffices to check that the specified edges in $\gtilde_{\mathbf{s},\mathbf{t}}$ give rise to the edges of $\gbar_{\mathbf{s},\mathbf{t}}$ as required. This is a direct verification.
\end{proof}

The above is applied below to calculate $\lad \ind F$ for two simple cases of $F \in \ob \apsh {\cc}$:
\begin{enumerate}
\item 
the constant functor $\kring$;
\item 
the functor $\sgn_{-} \boxtimes \sgn_{-} \in \ob \apsh {\cc_0}$ of Example \ref{exam:sgn_cc0} (considered as an object of $\apsh {\cc}$ via the inclusion $\apsh {\cc_0} \subset \apsh {\cc}$).
\end{enumerate}
Both can be considered as defining line bundles on $\gbar_{\mathbf{s},\mathbf{t}}$; to understand these and the calculation of $\lad \ind$ using  Proposition \ref{prop:lad_int}, it is convenient to lift to $\gtilde_{\mathbf{s},\mathbf{t}}$. This allows the signature on $\sym_s$ to be exploited. Moreover, one can then work with $\ghat_{\mathbf{s}, \mathbf{t}}$, which is useful when considering the right action of $\sym_t$ and the left action of $\sym_s$ corresponding to part of the structure of the  functor $\lad \ind F$.

\begin{prop}
\label{prop:ldm_Triv}
Suppose that $\frac{1}{2} \in \kring$.  Then, for $s,t \in \nat$ with $st>0$, there is an  isomorphism of $\sym_{s} \times \sym_t$-modules:
\[
\Big(\lad \ind  \kring \Big)(\mathbf{s}, \mathbf{t})
\cong 
\left\{
\begin{array}{ll}
\kring \sym_s & s=t \\
\sgn_{s} \boxtimes \sgn_t & s=t+1 \\
0 & \mbox{otherwise},
\end{array}
\right.
\]
where $\kring[\sym_s]$ is considered as a $\sym_s$-bimodule for the regular actions.
\end{prop}

\begin{proof}
Working with the constant functor $\kring$, the colimit given in Proposition \ref{prop:lad_int} is the cokernel of a map of the form 
\[
\bigoplus_{\mathrm{Edges} \ \gbar_{\mathbf{s}, \mathbf{t}}} \kring
\rightarrow 
\bigoplus_{\mathrm{Vertices} \ \gbar_{\mathbf{s}, \mathbf{t}}} \kring.
\]
Writing $x_\phi$ for the generator of $\kring$ indexed by the vertex corresponding to $\phi \in \sym_s$, the colimit serves to impose the relation $
x_\phi = - x_{\tau \phi}$.

If $s>t$, since the graph $\gbar_{\mathbf{s}, \mathbf{t}}$ is connected in this case and $\frac{1}{2} \in \kring$ by hypothesis, it follows that the value to be calculated is either $\kring$ or $0$. The proof proceeds by analysing the different cases:

If $s< t$, then $\Big(\lad \ind  \kring \Big)(\mathbf{s}, \mathbf{t})=0$. 

For $s =t$, the graph $\gbar_{\mathbf{s}, \mathbf{t}}$ has vertices indexed by $\sym_s$. The result follows immediately from  Proposition \ref{prop:lad_int} in this case. 

If $s=t+1$, since $s-t=1$, the graph $\gbar_{\mathbf{s},\mathbf{t}}$ identifies with $\gtilde_{\mathbf{s}, \mathbf{t}}$. In this case, the signs imposed by the relation $x_\phi = - x_{\tau \phi}$ are determined via the signature on $\sym_s$. Using this observation, it is straightforward to identify the value as a representation. 

It remains to show that, if $s-t >1$, then the value is trivial. First consider the case $(s,t)=(3,1)$: the graph $\gbar_{\mathbf{s},\mathbf{t}}$ in this case corresponds to the boundary of a triangle (cf. Example \ref{exam:gph}), which contains a cycle of odd length. One deduces that the colimit is zero in this case. This argument extends to treat all cases with $s-t>1$. 
\end{proof}

The same ideas apply to treat the case of the functor $\sgn_{-} \boxtimes \sgn_{-}$; the intervention of the sign representations means that the result is non-zero for many more values of $(s,t)$:

\begin{prop}
\label{prop:ldm_sgn}
Suppose that $\frac{1}{2} \in \kring$.  Then, 
for $s,t \in \nat$ with $st>0$, there is an  isomorphism of $\sym_{s} \times \sym_t$-modules:
\[
\Big(\lad \ind \ \sgn_{-} \boxtimes \sgn_{-} \Big)(\mathbf{s}, \mathbf{t})
\cong 
\left\{
\begin{array}{ll}
0 & s<t\\
\sgn_s \otimes \kring \sym_s  & s=t \\
\sgn_s \boxtimes \triv_t  & s>t, 
\end{array}
\right.
\]
\end{prop}

\begin{proof}
The proof is similar to that of Proposition \ref{prop:ldm_Triv}. The essential difference is exhibited by the case $(s,t)=(3,1)$: the presence of $\sgn_{s-t}$ ensures that the relations associated to the edges remain compatible, namely the behaviour can be interpreted using the signature for $\sym_s$. 

(For indications of an alternative approach to identifying the representations that occur, see Remark \ref{rem:isotypical_sgn} below.)
\end{proof}

\begin{rem}
\label{rem:isotypical_sgn}
The results of Propositions \ref{prop:ldm_Triv} and \ref{prop:ldm_sgn} should be compared with those of the following Section, notably Corollary \ref{cor:ldm_special_case_rho=1N}. For this, with respect to the notation used in Section \ref{sect:rep_compose}, one makes the identification $(s,t)= (m+n, n)$, so that the case $s=t$ corresponds to $m=0$ and $s=t+1$ to $m=1$.

In particular, for $s>t$  Corollary \ref{cor:ldm_special_case_rho=1N} identifies the representations occurring in  Propositions \ref{prop:ldm_Triv} and \ref{prop:ldm_sgn}. However, these Propositions contain more information, since they show that all other isotypical components (with respect to $\sym_s$) are zero when $s>t$.
\end{rem}

 \section{A representation-theoretic presentation of $\lad \ind$}
\label{sect:rep_compose}

As in Remark \ref{rem:lad_values}, if one only seeks to identify the values (as representations of the automorphism groups) of the  functor $\lad \ind $ defined on $\apsh{\cc}$, one can restrict to a much more limited datum than an object of $\apsh{\cc}$.
  Namely, for $G \in \ob \apsh{\cc}$ and $m,n \in \nat$, to calculate $\lad \ind G$ evaluated on $(\mathbf{m+n}, \mathbf{n})$ as a $\sym_{m+n} \times \sym_n$-module, it suffices to restrict $G$ to the full subcategory of $\cc_0$ with objects $(\mathbf{m-1}, \mathbf{n-1})$ and $(\mathbf{m}, \mathbf{n})$. 
  
The purpose of this Section is to implement this. Throughout the Section, $m,n \in \nat$ are fixed and the standard skeleton of $\cc_0$ with objects $(\mathbf{s}, \mathbf{t})$ for $(s,t) \in \nat^{\times 2}$ is used systematically.
  
\subsection{Reducing to representation theory}

Fix $m,n \in \nat$.

\begin{nota}
Denote by $\cc_{m,n}$ the full subcategory of $\cc_0$ with objects $(\mathbf{m-1}, \mathbf{n-1})$ and $(\mathbf{m}, \mathbf{n})$. 
\end{nota}

The restriction functor $\apsh{\cc} \rightarrow \apsh{\cc_{m,n}}$ has a section given by extension $\apsh{\cc_{m,n}} \hookrightarrow \apsh{\cc}$. The following is clear:

\begin{lem}
The category $\apsh{\cc_{m,n}}$ is equivalent to the full subcategory of $\apsh{\cc}$ of functors $G$ such that  $G(\mathbf{s}, \mathbf{t})=0$ unless $(s,t) \in \{ (m-1, n-1), (m,n)\}$;  the latter is a subcategory of $\apsh{\cc_0}$.
\end{lem}

The following introduces a module-theoretic model for $\apsh{\cc_{m,n}}$:

\begin{nota}
\label{nota:mc}
Let $\mc_{m,n}$ denote the category with objects triples $(M', M, \omega)$ where $M$ is a $\sym_m \times \sym_n$-module, $M'$ a $\sym_{m-1} \times \sym_{n-1}$-module and $\omega : M' \rightarrow M\downarrow ^{\sym_m \times \sym_n}_{\sym_{m-1} \times \sym_{n-1}}$ is a morphism of $\sym_{m-1} \times \sym_{n-1}$-modules.

The adjoint morphism $M'\uparrow  ^{\sym_m \times \sym_n}_{\sym_{m-1} \times \sym_{n-1}} \rightarrow M$ is denoted $\omega^\uparrow$.

 A morphism $(M_1',M_1, \omega_1) \rightarrow (M_2', M_2, \omega_2)$ is given by a pair of module maps $M_1' \rightarrow M_2'$, $M_1 \rightarrow M_2$ (equivariant with respect to the appropriate groups) that are compatible via $\omega_1$ and $\omega_2$ (i.e., the obvious diagram commutes). 
\end{nota}

By construction, one has:

\begin{lem}
\label{lem:mc_cc}
The category $\apsh{\cc_{m,n}}$ is equivalent to $\mc_{m,n}$. 
\end{lem}

 Hence we may consider $\mc_{m,n}$ as a subcategory of $\apsh{\cc}$. Conversely, the restriction functor $\apsh{\cc} \rightarrow \apsh{\cc_{m,n}}$ induces a functor 
$
 \apsh{\cc}  \rightarrow \mc_{m,n}$ written 
 $G \mapsto  G_{m,n}$.

Using the understanding of the projective generators of $\apsh{\cc_0}$ provided by Yoneda's Lemma (see Example \ref{exam:proj_cc0}), it is 
straightforward to deduce:

\begin{prop}
\label{prop:projectives_mc}
The following give projective generators of $\mc_{m,n}$:
\begin{eqnarray*}
(P^\finj_{\mathbf{m-1}} \boxtimes P^\finj_{\mathbf{n-1}})_{m,n}
&\cong & (\kring \sym_{m-1} \otimes \kring \sym_{n-1},\kring \sym_{m} \otimes \kring \sym_{n}, \omega) \\
(P^\finj_{\mathbf{m}} \boxtimes P^\finj_{\mathbf{n}})_{m,n}
&\cong &  (0,\kring \sym_{m} \otimes \kring \sym_{n}, 0),
\end{eqnarray*}
where $\omega$ is induced by the inclusions  $\sym_{m-1}\subset \sym_m$ and $\sym_{n-1}\subset \sym_n$.  

Moreover, for $mn>0$, there is a non-split short exact sequence 
\[
0
\rightarrow 
(P^\finj_{\mathbf{m}} \boxtimes P^\finj_{\mathbf{n}})_{m,n}
\rightarrow 
(P^\finj_{\mathbf{m-1}} \boxtimes P^\finj_{\mathbf{n-1}})_{m,n} 
\rightarrow 
(\kring \sym_{m-1} \otimes \kring \sym_{n-1},0,0)
\rightarrow 
0.
\]
In particular, the category $\mc_{m,n}$ is not semisimple for $mn>0$.
\end{prop}

The following observation is useful in calculations:

\begin{lem}
\label{lem:mc_simples}
Suppose that $(M', M, \omega)$ is an object of $\mc_{m,n}$, where $m,n>0$, such that 
\begin{enumerate}
\item 
$M'$ is a simple $\sym_{m-1} \times \sym_{n-1}$-module and $M$ is a simple $\sym_m \times \sym_n$-module; 
\item 
$\omega$ is non zero (hence injective). 
\end{enumerate}
Then there exists partitions $\mu \vdash m$, $\nu \vdash n$ and $\mu' \vdash m-1$, $\nu' \vdash n-1$ such that 
\begin{enumerate}
\item 
$M \cong S^\mu \boxtimes S^\nu$ and $M' \cong S^{\mu'} \boxtimes S^{\nu'}$; 
\item 
$\mu' \preceq \mu$ and $\nu' \preceq \nu$; 
\item 
$\omega$ is determined up to non-zero scalar multiple as the exterior tensor product of the monomorphisms $S^{\mu'} \hookrightarrow S^\mu \downarrow ^{\sym_m}_{\sym_{m-1}}$ and $S^{\nu'} \hookrightarrow S^\nu \downarrow ^{\sym_n}_{\sym_{n-1}}$ (unique up to non-zero scalar multiple) given by Pieri's rule.
\end{enumerate}
In particular, up to isomorphism, $(M', M, \omega)$ is determined by the tuple $(\mu' \preceq \mu ,\  \nu' \preceq \nu)$.
\end{lem}

\begin{proof}
The hypothesis that $M'$ and $M$ are simple implies that they are isomorphic to modules of the given form. The requirement that $\omega$ is non-zero then implies that it must be as given.
\end{proof}

\subsection{The functor $\ldm$}

The image of the composite $\mc_{m,n}
\subset 
\apsh {\cc}
\stackrel{\ind}{\rightarrow}
\apsh {\cd}$ is supported on $(\mathbf{m+n-2}, \mathbf{n-1})$ and $(\mathbf{m+n},\mathbf{n})$. (This composite functor will be written below simply as $\ind$.)

Applying the functor $\lad$ only changes the values on $(\mathbf{m+n},\mathbf{n})$ (cf. Example \ref{exam:F_cc}), hence we focus on the  composite:
\begin{eqnarray}
\label{eqn:ldm}
\mc_{m,n}
\subset 
\apsh {\cc}
\stackrel{\ind}{\rightarrow}
\apsh {\cd}
\stackrel{\lad}{\rightarrow}
\ce\dash\modules
\rightarrow 
(\sym_{m+n}\times \sym_n)\dash\modules
, 
\end{eqnarray}
in which the final functor is the evaluation on $(\mathbf{m+n}, \mathbf{n})$.

\begin{nota}
\label{nota:ldm}
Denote by $\ldm : \mc_{m,n} \rightarrow 
(\sym_{m+n}\times \sym_n)\dash\modules$  the composite functor (\ref{eqn:ldm}).
\end{nota}

The significance of $\ldm$ is stressed by the following Proposition, which follows directly from the constructions:

\begin{prop}
\label{prop:ldm_versus_lad}
For $G \in \ob \apsh{\cc}$, there is a natural isomorphism of $\sym_{m+n} \times \sym_n$-modules:
\[
\big(\lad \ind G \big) (\mathbf{m+n}, \mathbf{n}) 
\cong 
\ldm G_{m,n}.
\]
\end{prop}

From its construction, the following is clear:

\begin{prop}
\label{prop:ldm_right_exact}
The functor $\ldm$ is right exact.
\end{prop}

Proposition \ref{prop:calc_ldm} below shows that, in calculating $\ldm M$ one can reduce to the following cases of objects of $\mc_{m,n}$:
\begin{enumerate}
\item 
$(0, M, 0)$
\item 
$(M', M, \omega)$ where $\omega$ is injective and the induced map of $\sym_{m+n} \times \sym_n$-modules $\omega^\uparrow : M'\uparrow \rightarrow M$ is surjective. 
\end{enumerate}
In the first case, $\ldm$ reduces to the induction functor of Proposition \ref{prop:int_values}:
\[
\ldm (0,M,0) 
\cong 
\kring \sym_{m+n} \otimes_{\sym_m} M,
\]
which can be treated by standard methods of representation theory. Hence it will be the second which is the crux of the matter when calculating.

\begin{prop}
\label{prop:calc_ldm}
For  $(M', M, \omega) \in \ob \mc_{m,n}$: 
\begin{enumerate}
\item 
there is a canonical short exact sequence in $\mc_{m,n}$:
\[
0
\rightarrow 
(\ker \omega, 0 ,0 ) 
\rightarrow 
(M', M, \omega)
\rightarrow 
(M'/\ker \omega, M, \overline{\omega})
\rightarrow 
0
\]
where $\overline{\omega}$ is induced by $\omega$ and is injective; this induces the isomorphism 
$$\ldm (M', M, \omega) \cong \ldm (M'/\ker \omega, M, \overline{\omega})$$
of $\sym_{m+n}\times \sym_n$-modules; 
\item 
there is a canonical short exact sequence in $\mc_{m,n}$:
\[
0
\rightarrow 
(M', \mathrm{image} (\omega^\uparrow) ,\omega ) 
\rightarrow 
(M', M, \omega)
\rightarrow 
(0, M/ \mathrm{image} (\omega^\uparrow), 0)
\rightarrow 
0
\]
and this induces a short exact sequence of $\sym_{m+n}\times \sym_n$-modules:
\[
0
\rightarrow 
\ldm(M', \mathrm{image} (\omega^\uparrow) ,\omega ) 
\rightarrow 
\ldm(M', M, \omega)
\rightarrow 
\ldm(0, M/ \mathrm{image} (\omega^\uparrow), 0)
\rightarrow 
0.
\]
\end{enumerate}
\end{prop} 
 
\begin{proof}
The existence of the natural short exact sequences is clear.  Since $\ldm$ is right exact, the isomorphism for $\ldm$ in the first statement  follows immediately. 

For the second statement, one needs to check that $\ldm(M', \mathrm{image} (\omega^\uparrow) ,\omega ) 
\rightarrow 
\ldm(M', M, \omega)$ is injective. Now, the morphism $(M', \mathrm{image} (\omega^\uparrow) ,\omega ) 
\rightarrow 
(M', M, \omega)$ is given by the identity on $M'$ and the inclusion $\mathrm{image} (\omega^\uparrow) \subset M$. 
 One checks that this induces an injection on applying $\ldm$, by using the definition of this functor.
\end{proof}

\subsection{Calculating $\ldm$}

For $s,t \in \nat$ we fix the identification $\mathbf{s} \amalg \mathbf{t} = \mathbf{s+t}$  by extending the canonical inclusion $\mathbf{s} \subset \mathbf{s+t}$ by the order-preserving isomorphism $\mathbf{t} \cong ( \mathbf{s+t}) \backslash \mathbf{s}$. 
Thus, the canonical inclusions $\mathbf{m-1} \subset \mathbf{m}$ and $\mathbf{n-1} \subset \mathbf{n}$ induce the inclusion
\[
\mathbf{m+n-2} = (\mathbf{m-1}) \amalg (\mathbf{n-1}) \hookrightarrow \mathbf{m} \amalg \mathbf{n} = \mathbf{m+n}
\]
that extends to an isomorphism $(\mathbf{m+ n-2})  \amalg \mathbf{2} \cong \mathbf{m+n}$. These  give the following Young subgroups:
\begin{eqnarray*}
\sym_{m-1} \times \sym_{n-1} &\subset & \sym_{m+n-2} \\
\sym_{m-1} \times \sym_{n-1} \times \sym_2 &\subset & \sym_{m+n}\\
\sym_m \times \sym_n &\subset & \sym_{m+n}
\end{eqnarray*}
that are used below and which are compatible via the above identifications in the obvious sense.

By Proposition \ref{prop:int_values}, the values of $\ind (M',M,\omega)$ are given by:
\begin{eqnarray*}
\ind (M',M,\omega) (\mathbf{m+n-2}, \mathbf{n-1} ) &\cong & \kring \sym_{m+n-2} \otimes _{\sym_{m-1}} M' \\
\ind (M',M,\omega) (\mathbf{m+n}, \mathbf{n}) &\cong & \kring \sym_{m+n} \otimes _{\sym_{m}} M,
\end{eqnarray*}
where the respective actions of $\sym_{n-1}$ and $\sym_n$ are diagonal. 

The morphism $\omega$ induces a morphism of $\sym_{m+n-2} \times \sym_{n-1}$-modules:
\[
\kring \sym_{m+n-2} \otimes _{\sym_{m-1}} M'
\rightarrow 
\kring \sym_{m+n} \otimes _{\sym_{m}} M
\]
using the restricted module structure on the codomain. 

By the induction/ restriction adjunction for the left actions of $\sym_{m+n-2}$ and $\sym_{m+n}$ respectively, this gives a morphism of $\sym_{m+n} \times \sym_{n-1}$-modules denoted here by 
\[
\widetilde{\omega} : 
\kring \sym_{m+n} \otimes _{\sym_{m-1}} M'
\rightarrow 
\kring \sym_{m+n} \otimes _{\sym_{m}} M
\]
(exploiting the compatibility of the inclusions of the Young subgroups).

Now, via the inclusion $\sym_{m-1} \times \sym_{n-1} \times \sym_2 \subset \sym_{m+n}$, the group $\sym_2$ acts (on the right) on the domain by morphisms of $\sym_{m+n} \times \sym_{n-1}$-modules. Hence one can form the `trace':
\[
\widetilde{\omega} \circ (\id + \tau) : 
\kring \sym_{m+n} \otimes _{\sym_{m-1}} M'
\rightarrow 
\kring \sym_{m+n} \otimes _{\sym_{m}} M,
\]
where $(\id + \tau) \in \kring \sym_2$ is the sum of the generators. This is a morphism of $\sym_{m+n} \times \sym_{n-1}$-modules.

\begin{nota}
\label{nota:int_omega}
Denote by 
\begin{enumerate}
\item
$\Omega : (\kring \sym_{m+n} \otimes _{\sym_{m-1}} M')\uparrow_{\sym_{n-1}}^{\sym_n}
\rightarrow 
\kring \sym_{m+n} \otimes _{\sym_{m}} M $  the morphism of $\sym_{m+n}\times \sym_n$-modules induced from $\widetilde{\omega} \circ (\id + \tau)$;
\item 
$\Omega_2 : (\kring \sym_{m+n} \otimes _{\sym_{m-1}} M')^{\sym_2}\uparrow_{\sym_{n-1}}^{\sym_n}
\rightarrow 
\kring \sym_{m+n} \otimes _{\sym_{m}} M $ the morphism of $\sym_{m+n}\times \sym_n$-modules induced from the restriction of $\widetilde{\omega}$ to the invariants $(\kring \sym_{m+n} \otimes _{\sym_{m-1}} M')^{\sym_2} \subset \kring \sym_{m+n} \otimes _{\sym_{m-1}} M'$.
\end{enumerate}
\end{nota}

The behaviour of the $\sym_2$-invariants appearing in the above is described by the following, which requires no hypothesis on $\kring$.

\begin{lem}
\label{lem:sym2_invts}
For $m,n$ and $M'$ as above, there is a natural isomorphism of $\sym_{m+n} \times \sym_{n-1}$-modules
\[
(\kring \sym_{m+n} \otimes _{\sym_{m-1}} M')^{\sym_2}
\cong 
\kring \sym_{m+n}^{\sym_2} \otimes _{\sym_{m-1}} M'.
\]
Hence:
\[
(\kring \sym_{m+n} \otimes _{\sym_{m-1}} M')^{\sym_2}\uparrow_{\sym_{n-1}}^{\sym_n}
\cong 
(\kring \sym_{m+n}^{\sym_2} \otimes _{\sym_{m-1}} M')\uparrow_{\sym_{n-1}}^{\sym_n}.
\]

Moreover, there is a natural isomorphism of $\sym_{m+n} \times \sym_n$-modules:
\[
(\kring \sym_{m+n} \otimes _{\sym_{m-1}} M')^{\sym_2}
\uparrow_{\sym_{n-1}}^{\sym_n}
\cong 
\big((\kring \sym_{m+n} \otimes _{\sym_{m-1}} M') \uparrow_{\sym_{n-1}}^{\sym_n}   \big) ^{\sym_2}.
\]
\end{lem}

\begin{proof}
This is a straightforward verification, exploiting the fact that $\kring \sym_{m+n}$ is considered as a right $\sym_{m-1} \times \sym_{n-1} \times \sym_2$-module by restriction of the right regular action, in particular it is free. 
\end{proof}

The maps $\Omega$ and $\Omega_2$ identify explicitly as follows, in which Lemma \ref{lem:sym2_invts} has been used to consider the domain of $\Omega_2$ as being $(\kring \sym_{m+n}^{\sym_2} \otimes _{\sym_{m-1}} M')\uparrow_{\sym_{n-1}}^{\sym_n}$.

\begin{lem}
\label{lem:identify_int_omega}
 \ 
 \begin{enumerate}
 \item 
 For $\Theta \in \kring \sym_{m+n}$, $x \in M'$ and $g \in \sym_n$, the map $\Omega$ is given by
\[
(\Theta \otimes x) \otimes g \mapsto \Theta (1 + \tau) g \otimes \omega (x) g,
\]
where the terms represent classes of $(\kring \sym_{m+n} \otimes _{\sym_{m-1}} M')\uparrow_{\sym_{n-1}}^{\sym_n}$ and $\kring \sym_{m+n} \otimes _{\sym_{m}} M$ respectively.
\item
Similarly, for $\Psi \in \kring \sym_{m+n}^{\sym_2}$, $x \in M'$ and $g \in \sym_n$,
the map $\Omega_2$ is given by
\[
(\Psi \otimes x) \otimes g \mapsto \Psi g \otimes \omega (x) g,
\]
\end{enumerate}
In particular, the restriction of $\Omega$ to $(\kring \sym_{m+n}^{\sym_2} \otimes _{\sym_{m-1}} M')\uparrow_{\sym_{n-1}}^{\sym_n}$ equals $2 \Omega_2$.
\end{lem}

\begin{proof}
This follows directly from the construction of the two maps. 
\end{proof}

The following result is a reformulation of Proposition \ref{prop:lad_int} for  the functor $\ldm$.

\begin{thm}
\label{thm:lad_int}
For $(m,n) \in \nat^{\times 2}$ and $(M, M', \omega) \in \ob \mc_{m,n}$, there is a natural isomorphism of $\sym_{m+n} \times \sym_n$-modules 
\[
\ldm(M',M, \omega) (\mathbf{m+n}, \mathbf{n}) 
\cong 
\mathrm{Coker} \Big( 
\big(\kring \sym_{m+n} \otimes _{\sym_{m-1}} M'\big)\uparrow_{\sym_{n-1}}^{\sym_n}
\stackrel{\Omega}{\longrightarrow} 
\kring \sym_{m+n} \otimes _{\sym_{m}} M 
\Big).
\]

If $\frac{1}{2} \in \kring$, then 
\[
\ldm(M',M, \omega) (\mathbf{m+n}, \mathbf{n}) 
\cong 
\mathrm{Coker} \Big( 
(\kring \sym_{m+n} \otimes _{\sym_{m-1}} M')^{\sym_2} \uparrow_{\sym_{n-1}}^{\sym_n}
\stackrel{\Omega_2}{\longrightarrow} 
\kring \sym_{m+n} \otimes _{\sym_{m}} M 
\Big).
\]
\end{thm}

\begin{proof}
The first statement follows from Proposition \ref{prop:lad_explicit} by extending the isomorphism given by Proposition \ref{prop:int_values}.

If $\frac{1}{2} \in \kring$ then the $\kring$-module $\big(\kring \sym_{m+n} \otimes _{\sym_{m-1}} M'\big)\uparrow_{\sym_{n-1}}^{\sym_n}$ splits as 
\[
\big(\kring \sym_{m+n} \otimes _{\sym_{m-1}} M'\big)^{\sym_2}\uparrow_{\sym_{n-1}}^{\sym_n}
\oplus 
\big(\kring \sym_{m+n} \otimes _{\sym_{m-1}} M'\big)^{- \sym_2}\uparrow_{\sym_{n-1}}^{\sym_n},
\]
where the second factor denotes the anti-invariants for the $\sym_2$-action. Clearly the map $\Omega$ vanishes on the second factor and, as in Lemma \ref{lem:identify_int_omega}, the restriction of $\Omega$ to the first factor coincides with $2 \Omega_2$. The result follows, since $\frac{1}{2} \in \kring$, by hypothesis.
\end{proof}

\subsection{Restricting to isotypical components}
\label{subsect:restrict_isotypic_ldm}

In this subsection, $\kring$ is taken to be a field of characteristic zero.

We  are usually interested in focussing upon a single isotypical component of $\ldm(M',M, \omega) (\mathbf{m+n}, \mathbf{n})$, corresponding to a fixed partition $\rho \vdash N$, where $N = m+n$. This is equivalent to calculating the  $\sym_n$-representation after applying the functor $S^\rho \otimes_{\sym_N} -$. 

Theorem \ref{thm:lad_int} gives:

\begin{cor}
\label{cor:ldm_isotypical}
For $(m,n) \in \nat^{\times 2}$, $(M, M', \omega) \in \ob \mc_{m,n}$ and $\rho \vdash (m+n)$, 
 there is an isomorphism of $\sym_n$-modules:
\[
S^\rho \otimes_{\sym_{m+n} } \ldm(M',M, \omega) (\mathbf{m+n}, \mathbf{n}) 
\cong 
\mathrm{Coker} \Big( 
((S^\rho)^{\sym_2} \otimes _{\sym_{m-1}} M')\uparrow _{\sym_{n-1}}^{\sym_n} \rightarrow S^\rho \otimes_{\sym_m} M
\Big) 
\]
in which the map is given explicitly, for $\mathfrak{X} \in (S^\rho)^{\sym_2}$, $x \in M'$ and $g \in \sym_n$, by 
\[
(\mathfrak{X} \otimes x) \otimes g \mapsto \mathfrak{X} g \otimes \omega (x) g.
\]
\end{cor}

\begin{rem}
It is worth stressing the structure appearing in Corollary \ref{cor:ldm_isotypical}: 
\begin{enumerate}
\item 
In $S^\rho \otimes_{\sym_m} M$, we are considering the restriction $S^{\rho} \downarrow_{\sym_m \times \sym_n}^{\sym_N}$ and the $\sym_n$-action on the tensor product is diagonal.
\item
In $(S^\rho)^{\sym_2} \otimes _{\sym_{m-1}} M'$, we are considering the restriction $S^{\rho} \downarrow_{\sym_{m-1} \times \sym_{n-1} \times \sym_2}^{\sym_N}$ that leads to the $\sym_{m-1}\times \sym_{n-1}$ action on $(S^\rho)^{\sym_2}$; the $\sym_{n-1}$-action on the tensor product $(S^\rho)^{\sym_2} \otimes _{\sym_{m-1}} M'$ is diagonal. 
\end{enumerate}
\end{rem}

\begin{rem}
\label{rem:decompose_M_M'}
Since $\kring$ is a field of characteristic zero, the modules $M$ and $M'$ can be decomposed respectively as 
\begin{eqnarray*}
M &\cong & \bigoplus_{i \in \mathscr{I}} S^{\mu(i)} \boxtimes S^{\nu(i)} \\
M' &\cong &  \bigoplus_{j \in \mathscr{J}} S^{\mu'(j)} \boxtimes S^{\nu'(j)},
\end{eqnarray*}
where $\mu(i) \vdash m$, $\nu(i) \vdash n$ and  $\mu'(j) \vdash m-1$, $\nu'(j) \vdash n-1$. 

In this case, the map appearing in Corollary \ref{cor:ldm_isotypical} can be rewritten as:
\[
\bigoplus_{j \in \mathcal{J}}
((S^{\rho/ \mu'(j)})^{\sym_2} \otimes S^{\nu'(j)}))\uparrow_{\sym_{n-1}}^{\sym_n}
\rightarrow 
\bigoplus_{i \in \mathcal{I}}
S^{\rho/\mu(i)} \otimes S^{\nu(i)}.
\]
Here the skew representation $S^{\rho/ \mu'(j)}$ is understood to be zero if $\mu'(j) \not \preceq \rho$ and likewise for $S^{\rho/\mu(i)}$. This map is determined by $\omega$.
\end{rem}

To illustrate Remark \ref{rem:decompose_M_M'}, consider the case where  $(M',M, \omega)$ is as  in Lemma \ref{lem:mc_simples}:

\begin{prop}
\label{prop:ldm_special_case}
Suppose that $0< m, n \in \nat$ and that $(M', M , \omega) = (S^{\mu'} \boxtimes S^{\nu'}, S^\mu \boxtimes S^\nu, \omega)$ with $\omega \neq 0$, as in Lemma \ref{lem:mc_simples}. For $\rho \vdash N$, where $N= m+n$, the $\rho$-isotypical component of $\ldm(M',M, \omega) (\mathbf{N}, \mathbf{n})$ is given by
\[
\mathrm{Coker} \Big( 
\big((S^{\rho /\mu'}) ^{\sym_2} \otimes S^{\nu'}\big)\uparrow_{\sym_{n-1}}^{\sym_n}
\rightarrow 
S^{\rho/\mu} \otimes S^{\nu}
\Big)
\]
where the respective skew representations are understood to be zero if $\mu' \not \preceq \rho$ (resp. $\mu \not \preceq \rho$).  In particular, the cokernel  is zero if $\mu \not \preceq \rho$.

In terms of the identification of $\omega$ given in Lemma \ref{lem:mc_simples}, the adjoint of the above map, $(S^{\rho /\mu'}) ^{\sym_2} \otimes S^{\nu'}
\rightarrow 
(S^{\rho/\mu} \otimes S^{\nu})\downarrow_{\sym_{n-1}}^{\sym_n}$, is the tensor product of 
\begin{enumerate}
\item 
the restriction to $(S^{\rho /\mu'}) ^{\sym_2}$ of the morphism  $S^{\rho /\mu'} \rightarrow S^{\rho/\mu}$ of Lemma \ref{lem:skew_morphism}, determined by the given $S^{\mu'} \hookrightarrow S^\mu \downarrow^{\sym_m}_{\sym_{m-1}}$; 
\item 
the given  $S^{\nu'} \rightarrow S^\nu \downarrow^{\sym_n}_{\sym_{n-1}}$.
\end{enumerate}
\end{prop}

\begin{proof}
The result follows from Corollary \ref{cor:ldm_isotypical} using the identification of $\omega$ that is given in Lemma  \ref{lem:mc_simples}.
\end{proof}

This result is particularly simple in the case where $\rho = (1^N)$, so that $S^\rho = \sgn_N$.

\begin{cor}
\label{cor:ldm_special_case_rho=1N}
In the setting of  Proposition \ref{prop:ldm_special_case} with $\rho = (1^N)$, where $N=m+n$ with $m,n \in \nat$, the $(1^N)$-isotypical component of $$\ldm(M',M, \omega) (\mathbf{N}, \mathbf{n})$$ is zero unless $\mu = (1^m)$.

If $\mu = (1^m)$,  the $(1^N)$-isotypical component is $\sgn_N \boxtimes S^{\nu^\dagger}$. 
\end{cor}

\begin{proof}
The non-triviality criterion requires that $\mu \preceq (1^N)$, so that $\mu = (1^m)$. 
This implies that $\mu'$ must be equal to $(1^{m-1})$, so that the skew representation $S^{(1^N)/ \mu'}$ is isomorphic to $\sgn_{N-m+1}$. It follows immediately that the $\sym_2$ invariants $(S^{(1^N)/ \mu'})^{\sym_2}$ are zero. Proposition \ref{prop:ldm_special_case} thus identifies the isotypical component as $\sgn_N \boxtimes (S^{(1^N)/\mu} \otimes S^\nu)$. 

The skew representation $S^{(1^N)/\mu}$ is isomorphic to $\sgn_n$ so that $S^{(1^N)/\mu} \otimes S^\nu \cong \sgn_n \otimes S^\nu \cong S^{\nu^\dagger}$, as required.
\end{proof}

In conjunction with Proposition \ref{prop:ldm_versus_lad}, this gives:

\begin{cor}
\label{cor:isotyp_1N_F}
For $F \in \ob \apsh{\cc_0}$ and $N \in \nat$, the $(1^N)$-isotypical component of 
$\lad \ind F (\mathbf{N}, \mathbf{n})$ is 
\[
\sgn_N \boxtimes \big( \bigoplus_{\nu \vdash n} (S^{\nu^\dagger}) ^{\bigoplus c_F (N,\nu)}\big),
\] 
where $c_F (N,\nu)$ is the multiplicity of the composition factor $\sgn_{N-n} \boxtimes S^\nu$ in $F(\mathbf{N-n} , \mathbf{n})$ as a $\sym_{N-n} \times \sym_n$-module.

In particular, this coincides with the $(1^N)$-isotypical component of $\ind F (\mathbf{N}, \mathbf{n})$.
\end{cor}

The behaviour for $\rho =(1^N)$ given in Corollary \ref{cor:ldm_special_case_rho=1N} contrasts with that for $\rho = (N)$:

\begin{prop}
\label{prop:ldm_special_case_rho=N}
In the setting of  Proposition \ref{prop:ldm_special_case} with $\rho = (N)$, where $N=m+n$ with $m,n \in \nat$, the $(N)$-isotypical component of $$\ldm(M',M, \omega) (\mathbf{N}, \mathbf{n})$$ is zero unless $\mu = (m)$ and $n=0$. 

In the latter case, the $(N)$-isotypical component identifies as $\triv_N \boxtimes \triv_0$.
\end{prop}

\begin{proof}
As in the proof of Corollary \ref{cor:ldm_special_case_rho=1N}, one reduces to the case $\mu = (m)$ and $\mu'= (m-1)$. It follows in this case that the $(N)$-isotypical component is given by the cokernel of the morphism of $\sym_n$-modules:
\[
S^{\nu'}\uparrow_{\sym_{n-1}}^{\sym_n}
\rightarrow 
S^\nu.
\]
Since $S^\nu$ is simple, this is zero unless $n=0$, when $S^\nu = \triv_0$. 
\end{proof}

\part{Applications to baby beads}
\label{part:applications}

\section{Introducing $\kring \hfi(-,-)\trans$}
\label{sect:exam_trans}

The objective of this Section is to introduce and consider first properties of the $\cc$-module $\kring \hfi (-,-)\trans$.

\subsection{The $\cc$-module $\kring \hfi (-,-)\trans$}

Proposition \ref{prop:cc_to_sets} exhibits a functor $\hfi (-,-) : \cc \rightarrow \sets$ given by $(X,Y) \mapsto \hfi (Y,X)$. This gives rise to the functor $\kring \hfi (-,-)$ in $\apsh{\cc}$ by $\kring$-linearization, which restricts to  $\apsh{\cc_0}$. Then, by precomposing with the involution $\invcc$ of $\cc_0$ given in Proposition \ref{prop:involution_cc0}, one obtains the functor on $\cc_0$:
\begin{eqnarray}
\label{eqn:hfi_cc0}
(X,Y) \mapsto \kring \hfi (X,Y).
\end{eqnarray}
More is true: this is the restriction of a functor on $\cc$; to define the full structure, we  exploit the following `transpose' structure for certain $\finj\op$-modules:

\begin{lem}
\label{lem:transpose}
Let $\g$ be a presheaf  on $\finj$ with values in finite sets, i.e., a functor 
$
\g \ : \  \finj \op \rightarrow \fsets.
$ 
Then there is a  functor from $\finj$ to $\kring$-vector spaces $\kring \g\trans$  defined on objects by $\kring \g \trans(\mathbf{a}) := \kring \g (\mathbf{a})$; for $i \in \hfi( \mathbf{a}, \mathbf{b})$, the $\kring$-linear morphism 
\[
\kring \g \trans (i) : 
\kring \g \trans (\mathbf{a}) 
\rightarrow 
\kring \g\trans (\mathbf{b}) 
\]
sends a generator $[x]$  (for $x \in \g (\mathbf{a}) $) to $\sum_{y \in \g(i)^{-1} (x)} [y]$, summing over the fibre of $\g (i): \g (\mathbf{b}) \rightarrow \g (\mathbf{a})$ over $x$. 

The association $\g \mapsto \kring \g \trans$ is natural with respect to isomorphisms between presheaves on $\finj$ with values in finite sets.
\end{lem}

\begin{proof}
The presheaf $\g$ gives the object $\kring ^\g$ of $\apsh {\finj}$ given by compositing with the functor of set maps $\mathrm{Map} (-, \kring)$. 

For a finite set $X$, there is a canonical isomorphism $\kring X \cong (\kring  X)^\sharp$ of vector spaces; this corresponds to the pairing 
\begin{eqnarray*}
\kring X \otimes \kring X & \rightarrow & \kring \\
(x, y) &\mapsto & \left\{ \begin{array}{ll} 1 & x= y \\
0 & \mbox{otherwise}
\end{array}
\right.
\end{eqnarray*}
for elements $x, y \in X$. This pairing is clearly $\aut(X)$-equivariant for the trivial $\aut(X)$-action on $\kring$. 

Using this, for each $a \in \nat$, one has the isomorphism of $\kring$-vector spaces 
$\kring \g (\mathbf{a}) \cong \kring ^{\g (\mathbf{a})}$ of $\sym_a$-modules. Transferring the structure of $\kring ^\g$ (considered as an object of $\apsh{\finj}$) across these isomorphisms gives  $\kring \g\trans$.
\end{proof}

\begin{prop}
\label{prop:hfi_transpose}
There is a unique $\cc$-module structure on  $\kring \hfi (-,-)\trans$ such that:
\begin{enumerate}
\item 
the restriction to $\cc_0$ is the module 
$
(X,Y) \mapsto  \kring \hfi (X,Y)
$
given by (\ref{eqn:hfi_cc0});
\item 
the restriction to $\finj \times \fb$ is, for $Y \in \ob \fb$,  the transpose associated to $\hfi (-, Y) : \finj \op \rightarrow \sets$ by Lemma \ref{lem:transpose}, with $\aut(Y)$ acting by naturality of the transpose structure.
\end{enumerate}
\end{prop}

\begin{proof}
By Proposition \ref{prop:factor_cc}, if it exists, the $\cc$-module structure is uniquely specified by the given restrictions.

The transpose structure given by Lemma \ref{lem:transpose} is natural with respect to isomorphisms between presheaves taking values in finite sets.  Hence, considering $\hfi (-,-)$ as a functor from $\finj\op \times \fb$ to finite sets, the Lemma yields $\kring \hfi (-,-)\trans$ as a functor on $\finj \times \fb$. One checks that the restriction to $\fb^{\times 2}$ identifies with the restriction of the  $\cc_0$-module given by (\ref{eqn:hfi_cc0}) to $\fb^{\times 2}$. 

Hence, by Proposition \ref{prop:factor_cc}, to check that these structures in $\apsh {\finj \times \fb}$ and $\apsh {\cc_0}$ define an object of $\apsh{\cc}$, it suffices to check that they are compatible. This reduces to checking the commutativity of the diagram obtained by applying $\kring \hfi (-,-)\trans$  to the commutative square given in Proposition \ref{prop:factor_cc}. This is a straightforward verification. 
\end{proof}

\begin{rem}
\ 
\begin{enumerate}
\item
The usage of  $\trans$ above is slightly ambiguous; the Proposition makes clear that no confusion should result from this.
\item 
The category $\cc$ was introduced precisely so as to be able to encode the structure of $\kring \hfi (-,-)\trans$.
\end{enumerate}
\end{rem}

\subsection{$\kring \hfi (-,-) \trans$ as a $\cc_0$-module}

It is important to understand the underlying $\cc_0$-module structure of $\kring \hfi (-,-) \trans$ in the applications. This is based upon the identifications of Section \ref{subsect:bimodules}.

For this, by Remark \ref{rem:structure_cc},  it is sufficient to analyse (for varying $n,t \in \nat$) the behaviour of the natural inclusion of $\sym_n \times \sym_t$-modules:
\[
\kring \hfi (\mathbf{n}, \mathbf{t}) \hookrightarrow \kring \hfi (\mathbf{n+1}, \mathbf{t+1}) \downarrow^{\sym_{n+1} \times \sym_{t+1}} _{\sym_n \times \sym_t}
\]
sending a generator $[f]$ corresponding to $f \in \hfi (\mathbf{n}, \mathbf{t})$ to  $[f\amalg \id_{\mathbf{1}}]$,  where  $f\amalg \id_{\mathbf{1}} \in \hfi (\mathbf{n+1}, \mathbf{t+1}) $ is the extension of $f$ given by $n+1 \mapsto t+1$.

The above spaces are zero if $n> t$, hence we may assume that $n \leq t$. Take $\mathbf{n} \subset \mathbf{t}$ to be the canonical inclusion of $\{ 1, \ldots, n \}$ in $\{ 1, \ldots ,t\}$ and $(\mathbf{t-n}) \hookrightarrow \mathbf{t}$ to be the order-preserving complement. This defines the Young subgroup $\sym_n \times \sym_{t-n} \subset \sym_t$. 
 By composition with the canonical inclusion $\mathbf{t} \subset \mathbf{t+1}$, this also gives the inclusion $\mathbf{n} \amalg (\mathbf{t-n}) \hookrightarrow  \mathbf{t+1}$ and hence the Young subgroup $\sym_n \times \sym_{t-n} \subset \sym_{t+1}$. 

In the following, we use the above inclusions of Young subgroups and the notation given in Example \ref{exam:proj_finj}:

\begin{lem}
\label{lem:compat_hfi}
For $n \leq t \in \nat$, there is a commutative diagram of morphisms of $\sym_n \times \sym_t$-modules:
\[
\xymatrix{
\kring \sym_t \otimes _{\sym_{t-n}} \triv_{t-n}
\ar[d]_{\cong}
\ar[r]
&
\kring \sym_{t+1} \downarrow^{\sym_{t+1}}_{\sym_t} \otimes _{\sym_{t-n}} \triv_{t-n}
\ar[d]^{\cong}
\\
\kring \hfi (\mathbf{n}, \mathbf{t})
\ar[r]
&
\kring \hfi (\mathbf{n+1}, \mathbf{t+1}) \downarrow^{\sym_{n+1} \times \sym_{t+1}} _{\sym_n \times \sym_t}
}
\]
where the top horizontal map is induced by $\sym_t\subset \sym_{t+1}$ and the vertical maps are given respectively by $[e_t] \mapsto [\iota_{\mathbf{n}, \mathbf{t}}]$ and $[e_{t+1}] \mapsto [\iota_{\mathbf{n}, \mathbf{t}} \amalg \id_{\mathbf{1}}]$.

Here $\sym_t$ acts on $\kring \sym_t \otimes _{\sym_{t-n}} \triv_{t-n}$ by the left regular action on $\kring \sym_t$ and $\sym_n$ acts via the restriction of the right regular action. Likewise for $\kring \sym_{t+1} \downarrow^{\sym_{t+1}}_{\sym_t} \otimes _{\sym_{t-n}} \triv_{t-n}$. 
\end{lem}

\begin{proof}
This follows  from the definition of the bottom horizontal map together with the identification of the $\sym_n \times \sym_t$-module structures.
\end{proof}

The notation $\widehat{\lambda}$ used below was introduced in Notation \ref{nota:partitions}.

\begin{prop}
\label{prop:identify_hom_hfi}
For $\kring $ a field of characteristic zero and $n \leq t \in \nat$, there is an isomorphism of $\sym_n \times \sym_t$-modules:
\[
\kring \hfi (\mathbf{n}, \mathbf{t})
\cong 
\bigoplus_{\lambda \vdash t} S^\lambda \boxtimes S^{\lambda / \triv_{t-n}} 
\cong 
\bigoplus _{\substack{\lambda \vdash t, \nu \vdash n \\
\widehat{\lambda} \preceq \nu \preceq \lambda }} 
S^\lambda \boxtimes S^{\nu}.
\]

Moreover, with respect to this decomposition (and the respective decomposition for $ \kring \hfi (\mathbf{n+1}, \mathbf{t+1})$), the component of the map $\kring \hfi (\mathbf{n}, \mathbf{t}) \hookrightarrow \kring \hfi (\mathbf{n+1}, \mathbf{t+1}) \downarrow^{\sym_{n+1} \times \sym_{t+1}} _{\sym_n \times \sym_t}$ corresponding to $\lambda \vdash t$ and $\mu \vdash t+1$ is zero unless $\lambda \preceq \mu$.

If $\lambda \preceq \mu$, the map  identifies (up to non-zero scalar) as 
\[
S^\lambda \boxtimes S^{\lambda / \triv_{t-n}}
\hookrightarrow 
S^\mu \downarrow^{\sym_{t+1}}_{\sym_t} \boxtimes \  (S^{\mu / \triv_{t-n}})\downarrow^{\sym_{n+1}}_{\sym_n} 
\]
the exterior tensor product of the non-zero inclusion $S^\lambda \hookrightarrow S^\mu  \downarrow^{\sym_{t+1}}_{\sym_t}$ and the induced $S^{\lambda / \triv_{t-n}} \hookrightarrow  (S^{\mu / \triv_{t-n}})\downarrow^{\sym_{n+1}}_{\sym_n} $ given by 
\[
S^{\lambda / \triv_{t-n}}  = S^\lambda \otimes_{\sym_{t-n}} \triv_{t-n}
\hookrightarrow 
(S^{\mu / \triv_{t-n}})\downarrow^{\sym_{n+1}}_{\sym_n} 
= 
(S^{\mu} \otimes _{\sym_{t-n}} \triv_{t-n}) \downarrow^{\sym_{n+1}}_{\sym_n} .
\] 
\end{prop}

\begin{proof}
The first isomorphism follows from Proposition \ref{prop:decompose_ksym_bimodule} together with the definition of the skew representation, by using the isomorphism given in Lemma \ref{lem:compat_hfi}. The second isomorphism follows by identifying the skew representation. 

The identification of the bimodule map then follows by using Lemma \ref{lem:morphisms_simple_bimodules} in conjunction with Proposition 
\ref{prop:bimod_inclusions}. That the given map is a monomorphism follows since $\kring$ has characteristic zero.
\end{proof}

In general, the skew representations $S^{\lambda/ \triv_{t-n}}$ and $S^{\mu / \triv_{t-n}}$ are not simple. However, if both $\lambda_1 = {t-n}$ and $\mu_1 = t-n$, then they are; explicitly, one has:
\begin{eqnarray*}
S^{\lambda/ \triv_{t-n}}
&\cong & 
S^{\widehat{\lambda}}
\\
S^{\mu / \triv_{t-n} }
&\cong &  
S^{\widehat{\mu}}.
\end{eqnarray*}
Moreover, in this case, $\lambda \preceq \mu$ if and only if $\widehat{\lambda} \preceq \widehat{\mu}$. If these conditions hold then, up to non-zero scalar multiple, there is a unique non-trivial morphism
\[
S^{\widehat{\lambda}} \hookrightarrow S^{\widehat{\mu}} \downarrow ^{\sym_{n+1}}_{\sym_n}.
\]

Taking  the exterior tensor product with the (unique up to non-zero scalar multiple) morphism
\[
S^\lambda  \hookrightarrow  S^{\mu} \downarrow^{\sym_{t+1}}_{\sym_t}
\]
gives a generator of 
$ 
\hom_{\sym_t \times \sym_n} ( S^\lambda \boxtimes S^{\widehat{\lambda}},  S^{\mu} \boxtimes S^{\widehat{\mu}}) \cong \kring.
$

Putting these facts together, Proposition \ref{prop:identify_hom_hfi}
 has the following Corollary, which is important for Theorem \ref{thm:calc} below.
 
 \begin{cor}
 \label{cor:restrict_widehat}
 For $\kring $ a field of characteristic zero,  $n \leq t \in \nat$ and $\lambda \vdash t$, $\mu \vdash t+1$ such that $\lambda \preceq \mu$,  $\lambda_1 = {t-n}$ and $\mu_1 = t-n$, the composite:
 \[
 S^\lambda \boxtimes S^{\lambda/ \triv_{t-n}} 
 \hookrightarrow 
 \kring \hfi (\mathbf{n}, \mathbf{t})
 \hookrightarrow 
  \hfi (\mathbf{n+1}, \mathbf{t+1}) \downarrow^{\sym_{n+1} \times \sym_{t+1}} _{\sym_n \times \sym_t}
  \twoheadrightarrow 
 (S^\mu  \boxtimes S^{\mu/ \triv_{t-n}})\downarrow   ^{\sym_{n+1} \times \sym_{t+1}} _{\sym_n \times \sym_t}
 \]
 is a generator of 
$
\hom_{\sym_t \times \sym_n} ( S^\lambda \boxtimes S^{\widehat{\lambda}},  S^{\mu} \boxtimes S^{\widehat{\mu}}) \cong \kring.
$
 \end{cor}

\subsection{General properties of $\kring \hfi (-,-)\trans$ and $\hcc \kring \hfi (-,-)\trans$}

First we consider the restriction to $\cc_0$, namely the functor of equation (\ref{eqn:hfi_cc0}). The fact that $\cc_0$ decomposes into connected components (see Lemma \ref{lem:subcat_cc}) has the following immediate consequence: 

\begin{lem}
\label{lem:decompose_kring_hfi_over_cc0}
There is a direct sum splitting in $\apsh{\cc_0}$:
\[
\kring \hfi(-,-) \trans \downarrow _{\cc_0}^\cc
\cong 
\bigoplus
_{d \in \nat}
\kring \hfi(-,-)_{(d)} 
\]
where $\kring \hfi(-,-)_{(d)}$ is supported on $\cc_0^{(-d)}$.
\end{lem}

The functors $\kring \hfi(-,-)_{(d)}$ are described in terms of the standard projectives of $\apsh{\cc_0}$ by the following:

\begin{prop}
\label{prop:present_kring_hfi_over_cc0}
For $d \in \nat$, the unique element of $\hfi(\mathbf{0}, \mathbf{d})$ induces a surjection 
\[
P^{\cc_0} _{(\mathbf{0}, \mathbf{d})}
\twoheadrightarrow 
\kring \hfi(-,-)_{(d)}
\]
and this induces an isomorphism 
\[
P^{\cc_0} _{(\mathbf{0}, \mathbf{d})}/\sym_d
\cong 
\kring \hfi(-,-)_{(d)},
\]
for the canonical action of $\sym_d= \aut (\mathbf{0}, \mathbf{d})$ on the projective.

In particular, if $\kring$ is a field of characteristic zero, then $\kring \hfi(-,-)_{(d)}$ is projective in $\apsh{\cc_0}$.
\end{prop}

\begin{proof}
The morphism corresponds by Yoneda's lemma to the generator $\init \in \hfi (\mathbf{0}, \mathbf{d})$. 
It is straightforward to see that it is surjective by using Proposition \ref{prop:factor_cc}. 
Moreover, it is clear that this factors across the coinvariants for the $\sym_d$-action. 

To show that one obtains an isomorphism, one analyses the situation at the level of the set-valued functors (cf.  Example \ref{exam:hom0D_cc0}).  
Consider $\hom_{\cc_0} ((\mathbf{0}, \mathbf{d}), (X,Y))$, where $|Y|-|X|=d$. By definition of $\cc_0$,  this is isomorphic to 
$\hfi (X \amalg \mathbf{d} , Y)$. Moreover, the projection map is the $\kring$-linearization of the restriction:
\[
\hfi (X \amalg \mathbf{d} , Y) \twoheadrightarrow \hfi (X, Y).
\]
(The reader should verify the $\cc_0$-structures involved and that the above is a natural transformation.)

This surjection induces an isomorphism $\hfi (X \amalg \mathbf{d} , Y)/\sym_d \stackrel{\cong}{\rightarrow} \hfi (X, Y)$, since the passage to the quotient set corresponds to forgetting the map $\mathbf{d} \rightarrow Y$ (the image of this map can be recovered as the complement of that of $X$).

The final statement follows from the fact that the coinvariants $P^{\cc_0} _{(\mathbf{0}, \mathbf{d})}/\sym_d$ are a direct summand of the projective $ P^{\cc_0} _{(\mathbf{0}, \mathbf{d})}$ if $\kring$ is a field of characteristic zero.
\end{proof}

\begin{rem}
The above ingredients also serve to show that $\hfi (\mathbf{a}, \mathbf{b})$ is the transitive $\sym_b \times \sym_a$-set given by $\sym_b/ \sym_{b-a}$, where $\sym_b$ acts via the canonical left action and $\sym_a$ acts via the induced right action arising from the right regular action (compare Example \ref{exam:proj_finj}).

The advantage of the above presentation is that it encodes the compatibilities of these identifications. 
\end{rem}

Now we  consider $\hcc \kring \hfi (-,-)\trans$, for which it is essential to work at the level of the functors to $\kring$-modules rather than sets. By Proposition \ref{prop:extn_left_adj_cc}, for $G \in \ob \apsh{\cc}$, $\hcc G$ is given by
\[
\hcc G (W,Z) 
=
\mathrm{Coker} 
\Big(
\bigoplus_{\substack{ i: U \hookrightarrow W\\ |U|= |W|-1}}
G(U,Z) 
\stackrel{\bigoplus G(i,\id,\init)}{\longrightarrow}
G(W,Z)
\Big).
\]
Applied to $\hcc \kring \hfi (-,-)\trans$, with respect to the splitting of Lemma \ref{lem:decompose_kring_hfi_over_cc0}, the only non-trivial components that appear in the map on the right hand side for $|Z|-|W|= d \in \nat$ (conserving the notation above) are those of the form:
\[
\kring \hfi (U,Z)_{(d+1)} 
\rightarrow 
\kring \hfi (W,Z)_{(d)}.
\]

In particular, taking $Z = \mathbf{d+1}$ and $W= \mathbf{1}$ (so that $U= \emptyset = \mathbf{0}$), this gives the fundamental map:
\[
\kring = \kring \hfi (\mathbf{0} ,\mathbf{d+1})_{(d+1)} 
\rightarrow 
\kring \hfi (\mathbf{1},\mathbf{d+1})_{(d)}.
\]
that sends the canonical generator to the sum $\sum_{x \in \mathbf{d+1}} [x]$, where $x$ is considered as a map $\mathbf{1} \rightarrow \mathbf{d+1}$. 

\begin{prop}
\label{prop:decompose_hcc_kring_hfi}
The functor $\hcc \kring \hfi (-,-)\trans$ of $\apsh{\cc_0}$ admits a decomposition 
\[
\hcc \kring \hfi (-,-)\trans 
\cong 
\bigoplus_{d\in \nat} \big( \hcc \kring \hfi (-,-)\trans \big)_{(d)},
\]
where the functor $\big( \hcc \kring \hfi (-,-)\trans \big)_{(d)}$ is supported on $\cc_0^{(-d)}$ and has presentation:
\[
P^{\cc_0}_{(\mathbf{1},\mathbf{d+1}) }
\rightarrow 
P^{\cc_0} _{(\mathbf{0}, \mathbf{d})}/\sym_d
\rightarrow 
\big( \hcc \kring \hfi (-,-)\trans \big)_{(d)}
\rightarrow 
0.
\]
Here the first map corresponds by Yoneda's lemma to the class $$\sum_{x \in \mathbf{d+1}} [x] \in \kring \hfi (\mathbf{1},\mathbf{d+1})_{(d)} \cong \big(P^{\cc_0} _{(\mathbf{0}, \mathbf{d})}/\sym_d\big) (\mathbf{1},\mathbf{d+1}). $$

If $\kring$ is a field of characteristic zero, this gives a projective presentation of $\big( \hcc \kring \hfi (-,-)\trans \big)_{(d)}$ in $\apsh{\cc_0}$.
\end{prop}

\begin{proof}
That $\hcc \kring \hfi (-,-)\trans$ admits a splitting follows from the decomposition of $\cc_0$ into connected components. The precise form of the splitting is given by combining Lemma \ref{lem:decompose_kring_hfi_over_cc0} with the discussion preceding the statement. 

To deduce that one has a presentation of the given form, it suffices to observe that the kernel of the induced surjection
\[
\kring \hfi(-,-)_{(d)}
\twoheadrightarrow
 \big( \hcc \kring \hfi (-,-)\trans \big)_{(d)}
\]
is generated as a $\cc_0$-module by the element $\sum_{x \in \mathbf{d+1}} [x]$. 

The final statement follows from the projectivity statement given in Proposition \ref{prop:present_kring_hfi_over_cc0} over a field of characteristic zero.
\end{proof}

\subsection{Examples for small $d$}
The cases $d\in \{ 0, 1 \}$ of $ \big( \hcc \kring \hfi (-,-)\trans \big)_{(d)}$ can be treated directly. 

\begin{prop}
\label{prop:d=0_hcc_k_hfi}
One has 
\[
\big( \hcc \kring \hfi (-,-)\trans \big)_{(0)} (X,Y) \cong 
\left\{
\begin{array}{ll}
\kring & X=Y=\emptyset \\
0 & \mbox{otherwise}.
\end{array}
\right.
\]
\end{prop}

\begin{proof}
Since we are considering a functor supported on $\cc_0^{(0)}$, we may restrict to the case $|X|=|Y|$.  The case $X=Y = \emptyset$ is clear, hence suppose that $X \neq \emptyset$.

From the definitions, one sees that $\big( \hcc \kring \hfi (-,-)\trans \big)_{(0)} (X,Y)$ is the cokernel of the morphism
\[
\bigoplus_{\substack{U \subset X \\ |U|= |X|-1}}
\kring \hfi (U, Y) 
\rightarrow 
\kring \hfi (X,Y)
\]
where the map is induced by the isomorphisms $\hfi (U,Y) \cong \hfi (X,Y)$ resulting from the hypothesis that $|X|=|Y|$ and $|U|= |X|-1$, so that there is a unique way to extend an injective map $U \hookrightarrow Y$ to an isomorphism $X \cong Y$.
 It follows that the displayed morphism is surjective. The result follows.
 \end{proof}

\begin{prop}
\label{prop:d=1_hcc_k_hfi}
One has 
\[
\big( \hcc \kring \hfi (-,-)\trans \big)_{(1)} (X,Y) \cong 
\left\{
\begin{array}{ll}
\kring & |Y|=|X|+1 \\
0 & \mbox{otherwise}.
\end{array}
\right.
\]
More precisely, this functor is supported on $\cc_0^{(-1)}$ on which it identifies with the functor $\ori \boxtimes \ori \downarrow_{\cc_0^{(-1)}}$ of  Example \ref{exam:sgn_cc0}.
\end{prop}

\begin{proof}
Since $\big( \hcc \kring \hfi (-,-)\trans \big)_{(1)}$ is supported on $\cc_0^{(-1)}$, it is zero on $(X,Y)$ unless $|Y|=|X|+1$.
 In the latter case, evaluated on $(X,Y)$, it is given by the cokernel of the map:
 \[
 \bigoplus_{\substack{U \subset X \\ |U | = |X|-1 } } 
 \kring \hfi (U, Y) 
 \rightarrow 
 \kring \hfi (X,Y)
 \]
 again corresponding to the $\finj$-module structure of $\kring \hfi (-,-)\trans$. Here, by choosing an inclusion $X\subset Y$, one can identify $\hfi (X,Y)$ with $\aut (Y)$ (using the hypothesis on $X$), with $\aut(X)$ acting via restriction. 
 
Then, given $f: U \hookrightarrow Y$, where $|U|= |X|-1 = |Y|-2$, under the above map and with the above identification (so that $U \subset X \subset Y$), the image in $\kring \aut(Y) $ is $f_1 + f_2$, where the $f_i$ are the two possible ways of extending $f$ to an automorphism of $Y$. These are related by $f_1 \tau_U = f_2$, where $\tau_U$ is the transposition associated to $Y \backslash U$.

In particular, for any $g \in \aut (Y)$, choosing any $U \subset X$ and taking $f= g |_U$, the restriction of $g$ to $U$, one deduces the following relation in the cokernel: 
\[
g \equiv - g\tau_U.
\]
Since the transpositions of the form $\tau_U$ generate the symmetric group $\aut (Y)$, this identifies the representations.

It remains to check the behaviour on the standard morphisms $(X, Y) \rightarrow (X \amalg \mathbf{1}, Y \amalg \mathbf{1})$, for $|Y|=|X|+1 $. This follows (by passage to the quotient) from the behaviour  of 
\[
\kring \hfi (X,Y) \rightarrow \kring \hfi (X \amalg \mathbf{1}, Y \amalg \mathbf{1}).
\]
Now, similarly to Proposition \ref{prop:cc_00}, this identifies with the canonical inclusion
\[
\kring \aut (Y) 
\hookrightarrow 
\kring \aut (Y \amalg \mathbf{1})  
\]
with the appropriate bimodule structures, using the identifications $\hfi (X,Y) \cong \aut (Y)$ and $\hfi (X \amalg \mathbf{1}, Y \amalg \mathbf{1}) \cong \aut (Y \amalg \mathbf{1})$ that result from the hypothesis $|Y|=|X|+1 $.  On passage to the quotient, one obtains the claimed structure.
\end{proof}

\begin{rem}
Propositions \ref{prop:d=0_hcc_k_hfi} and  \ref{prop:d=1_hcc_k_hfi} should be compared with the respective identifications 
$\kring \hfi (-,-)\trans_{(0)}\cong  P^{\cc_0}_{(\mathbf{0}, \mathbf{0})}$
 and 
$  \kring \hfi (-,-)\trans_{(1)} \cong P^{\cc_0}_{(\mathbf{0}, \mathbf{1})}$.
 In both cases, applying the functor $\hcc$ has drastic effects.
\end{rem}

\subsection{Calculating over a field of characteristic zero}
For $a, b \in \nat$, $\kring \hfi (\mathbf{a}, \mathbf{b})$ is a permutation $\sym_a \times \sym_b$-module. Over a field of characteristic zero it is thus straightforward to identify this representation. It is a remarkable fact that one can also calculate the representation associated to the functor $H^\cc_{\cc_0} \kring \hfi (-,-)\trans$. This is the main result of \cite{P_bb_finj}:

\begin{thm}
\label{thm:calc}
For $\kring$ a field of characteristic zero and $1 \leq a\leq b \in \nat$, there are  isomorphisms of $\sym_a\times \sym_b$-modules:
\begin{eqnarray*}
\big(\kring \hfi (-, -)\trans \big) (\mathbf{a},\mathbf{b})&\cong&
\bigoplus_{\substack{\lambda \vdash b , \nu \vdash a \\ \widehat{\lambda} \preceq \nu\preceq \lambda}}
 S^{\nu} \boxtimes S^\lambda 
 \\
 \big( H^\cc_{\cc_0} \kring \hfi (-, -)\trans \big) (\mathbf{a},\mathbf{b})
&\cong &
\bigoplus_{\substack{\lambda \vdash b \\ \lambda_1 =b-a}}
 S^{\widehat{\lambda}}\boxtimes S^\lambda.
\end{eqnarray*}
In particular, the quotient map $ \kring \hfi (-,-) \twoheadrightarrow \hcc \kring \hfi (-,-)\trans$ identifies as the projection:
\begin{eqnarray*}
\bigoplus_{\substack{\lambda \vdash b , \nu \vdash a \\ \widehat{\lambda} \preceq \nu\preceq \lambda}}
 S^{\nu} \boxtimes S^\lambda 
\twoheadrightarrow 
\bigoplus_{\substack{\lambda \vdash b \\ \lambda_1 =b-a}}
S^{\widehat{\lambda}} \boxtimes S^\lambda.
\end{eqnarray*}

The structure of $H^\cc_{\cc_0} \kring \hfi (-,-)\trans$ as an object of $\apsh{\cc_0}$ is determined by the morphisms 
\[
 \big( H^\cc_{\cc_0} \kring \hfi (-, -)\trans \big) (\mathbf{a},\mathbf{b})
\rightarrow 
 \big( H^\cc_{\cc_0} \kring \hfi (-, -)\trans \big) (\mathbf{a}\amalg \mathbf{1} ,\mathbf{b} \amalg \mathbf{1}).
\]
Identifying $\mathbf{a}\amalg \mathbf{1}$ with $\mathbf{a+1}$ and $\mathbf{b} \amalg \mathbf{1}$ with $\mathbf{b+1}$, this morphism of $\sym_a \times \sym_b$-modules identifies as 
\[
\bigoplus_{\substack{\lambda \vdash b \\ \lambda_1 =b-a}}
S^{\widehat{\lambda}} \boxtimes S^\lambda
\rightarrow 
\bigoplus_{\substack{\mu \vdash b+1 \\ \mu_1 =b-a}}
\big(S^{\widehat{\mu}} \boxtimes S^\mu\big)\downarrow^{\sym_{a+1} \times \sym_{b+1}} 
_{\sym_a \times \sym_b},
\]
of which the component indexed by a pair $(\lambda, \mu)$ is zero unless $\lambda \preceq \mu$, in which case it identifies (up to non-zero scalar multiple) with the generator of 
\[
\hom_{\sym_a \times \sym_b} ( S^{\widehat{\lambda}} \boxtimes S^{\lambda},  S^{\widehat{\mu}} \boxtimes S^{\mu}) \cong \kring.
\]
\end{thm}

\begin{proof}
The first part of the Theorem is proved in \cite{P_bb_finj}. 

It remains to identify the structure as a module on $\cc_0$. The restriction to $\fb^{\times 2}$ has already been identified, hence, as in Remark \ref{rem:structure_cc}, it suffices to consider the behaviour on the morphisms of $\cc_0$ of the form 
 $
(\mathbf{a}, \mathbf{b}) 
\rightarrow 
(\mathbf{a} \amalg \mathbf{1} , \mathbf{b} \amalg \mathbf{1}) 
$,  
corresponding to the canonical inclusions. 
 The result follows by combining Proposition \ref{prop:identify_hom_hfi} with Corollary \ref{cor:restrict_widehat}, based upon  the material of Section \ref{subsect:bimodules}.
\end{proof}

 \section{Further structure of $\kring \hfi (-,-)\trans$}
\label{sect:further} 
 
This Section shows how the general  theory developed in this paper allows a more in-depth analysis of the structure of the $\cc$-module   $\kring \hfi (-,-)\trans$. (This material is not required for the main applications of this paper, hence the reader may prefer to skip it on first reading and pass directly to Section \ref{sect:exam_beads}.)
 
\subsection{Precomposing with $- \amalg (\mathbf{0}, \mathbf{1})$}
Consider the functor $\kring \hfi( -, -) \trans \circ (- \amalg (\mathbf{0}, \mathbf{1}))$; this is given by 
 $
(X, Y) \mapsto \kring \hfi (X, Y \amalg \mathbf{1} )\trans
$, 
which has the structure of  a $\cc$-module.

\begin{nota}
\label{nota:pi}
Denote by $\pi_{X,Y}$,  for $(X,Y) \in \ob \cc$, the surjection  
\[
\pi_{X,Y} : 
\kring \hfi (X, Y \amalg \mathbf{1} )
\twoheadrightarrow 
\kring \hfi (X, Y )
\]
 which sends generators  in $\hfi (X, Y \amalg \mathbf{1}) \backslash \hfi (X, Y )$ to zero. (This is a retract of the inclusion $\kring \hfi (X, Y ) \hookrightarrow \kring \hfi (X, Y \amalg \mathbf{1} )$ induced by $Y \hookrightarrow Y \amalg \mathbf{1}$.)
\end{nota}

The following is clear:

\begin{lem}
\label{lem:pi}
The morphisms $\pi_{X,Y}$ define a natural transformation of functors on $\cc_0$:
\[
\pi : 
\kring \hfi( -, -) \circ (- \amalg (\mathbf{0}, \mathbf{1}))
\twoheadrightarrow
 \kring \hfi( -, -).
\]
Moreover, this admits a section in $\apsh{\cc_0}$. 

For $d \in \nat$, this restricts to
$
\pi : 
\kring \hfi( -, -)_{(d)}  \circ (- \amalg (\mathbf{0}, \mathbf{1}))
\twoheadrightarrow
 \kring \hfi( -, -)_{(d-1)}
$.
\end{lem}

However, we are principally interested in $\kring \hfi (-,-)\trans$, not its restriction to $\cc_0$. In this case, one has:

\begin{prop}
\label{prop:pi_nat_trans_cc}
The morphisms $\pi$ define a natural surjection of functors on $\cc$:
\[
\pi : 
\kring \hfi( -, -)\trans \circ (- \amalg (\mathbf{0}, \mathbf{1}))
\twoheadrightarrow
 \kring \hfi( -, -)\trans.
\]
This does not admit a section.
\end{prop}

\begin{proof}
By Proposition \ref{prop:factor_cc} and using Lemma \ref{lem:pi}, it suffices to establish that, for all $(X,Y) \in \ob \cc$, the following diagram, induced by the canonical morphism $(X, Y) \rightarrow (X \amalg \mathbf{1}, Y )$, commutes:
\[
\xymatrix{
\kring \hfi (X, Y \amalg \mathbf{1})
\ar[r]^{\pi_{X,Y}} 
\ar[d]
&
\kring \hfi (X,Y) 
\ar[d]
\\
\kring \hfi (X\amalg \mathbf{1}, Y \amalg \mathbf{1})
\ar[r]_(.55){\pi_{X\amalg \mathbf{1},Y}} 
&
\kring \hfi (X\amalg \mathbf{1}, Y ),
}
\]
in which the vertical maps are the transpose morphisms given by Lemma \ref{lem:transpose}. 

For this, it suffices to consider the behaviour on a generator given by a map $f : X \hookrightarrow Y \amalg \mathbf{1}$. 
\begin{enumerate}
\item 
If $\mathbf{1} \subset \mathrm{image} (f)$, then clearly both composites are zero. 
\item 
If $\mathrm{image} (f) \subset Y$, then comparing the two transpose maps via the inclusion $Y \subset Y \amalg \mathbf{1}$, the only difference is that the left hand transpose map has the additional term given by the extension $1 \mapsto 1$ (using the obvious notation). 
However, this term vanishes under $\pi_{X\amalg \mathbf{1},Y}$, whence the required commutativity.
\end{enumerate}

It remains to show that this does not admit a section. For this, we use the fact that, if $|X|= |Y|>0$, then the transpose map 
\[
\kring \hfi (X, Y \amalg \mathbf{1}) 
\rightarrow 
\kring \hfi (X\amalg \mathbf{1}, Y \amalg \mathbf{1})
\]
is an isomorphism, since there is a unique way to extend any given injective map. 

Suppose that there exists a section; naturality would require the commutativity of the diagram
\[
\xymatrix{
\kring \hfi (X,Y) 
\ar@{^(-->}[r]
\ar[d]
&
\kring \hfi (X, Y \amalg \mathbf{1})
\ar[d]
\\
\kring \hfi (X\amalg \mathbf{1}, Y ),
\ar@{^(-->}[r]
&
\kring \hfi (X\amalg \mathbf{1}, Y \amalg \mathbf{1}),
}
\]
in which the dashed arrows denote the section and the vertical arrows the transpose morphisms.

Taking $|X|=|Y|>0$ gives a contradiction: passing around the top of the diagram is injective (since the right hand vertical map is an isomorphism), whereas passing around the bottom is zero, since $\kring \hfi (X\amalg \mathbf{1}, Y )=0$ in this case.
\end{proof}

\begin{rem}
This morphism is exploited in \cite{P_bb_finj} in the following form. On restricting to $\finj \times \fb$ and evaluating on $Y= \mathbf{b-1}$ for $b \in \nat$,  one can consider $\pi$ as yielding a surjection of $\finj$-modules
\[
\pi :
\kring \hfi (-, \mathbf{b}) \trans \downarrow _{\sym_{b-1}}^{\sym_b} 
\rightarrow 
\kring \hfi (-, \mathbf{b-1})
\]
that is $\sym_{b-1}$-equivariant.

Crucially, one can identify the kernel, leading to a short exact sequence of $\finj$-modules that is $\sym_{b-1}$-equivariant:
 \[
0
\rightarrow
\kring \langle 1 \rangle \conv \kring \hfi (-, \mathbf{b-1}) 
\rightarrow 
\kring \hfi (-, \mathbf{b})\downarrow ^{\sym_b}_{\sym_{b-1}}
\stackrel{\pi}{\rightarrow} 
\kring \hfi (-, \mathbf{b-1}) 
\rightarrow 
0.
\]
Here, $ \kring \langle 1 \rangle \conv \kring \hfi (-, \mathbf{b-1}) $ is the $\finj$-module given by 
\[
U \mapsto \bigoplus_{\substack{U' \subset S \\ |U'| = |U |-1}}
\kring \hfi (U', \mathbf{b-1}). 
\]

In particular, this provides an inductive strategy for analysing the $\finj$-modules 
$\kring \hfi (-, \mathbf{b})$ by increasing induction on $b$.
\end{rem}

\subsection{Filtering $\big(\hcc \kring \hfi (-,-)\trans\big)_{(d)}$}

Proposition \ref{prop:pi_nat_trans_cc} allows us to pass to the category $\apsh{\cc_0}$ by applying the functor $\hcc$:

\begin{cor}
\label{cor:hcc_hfi_pi}
The natural transformation $\pi$ induces a natural surjection:
\[
\Big(
\hcc \kring \hfi (-,-)\trans
\Big)
\circ (- \amalg (\mathbf{0}, \mathbf{1}) )
\twoheadrightarrow 
\hcc \kring \hfi (-,-)\trans.
\]

Hence, by adjunction, there is a natural morphism:
\[
\hcc \kring \hfi (-,-)\trans
\rightarrow 
\radj \Big( \hcc \kring \hfi (-,-)\trans \Big).
\]

This restricts, for $d \in \nat$ to 
\[
\big(\hcc \kring \hfi (-,-)\trans\big)_{(d)}
\rightarrow 
\radj \Big( \big(\hcc \kring \hfi (-,-)\trans\big)_{(d-1)} \Big).
\]
\end{cor}

\begin{proof}
The first statement follows from Proposition \ref{prop:pi_nat_trans_cc} by applying the functor $\hcc$ and using the fact that $\hcc$ commutes with $\circ (- \amalg (\mathbf{0}, \mathbf{1}) )$, which can be deduced from the final statement of Proposition \ref{prop:extn_left_adj_cc}. 

The second statement follows immediately, since $\radj$ is right adjoint to the precomposition functor. The final statement simply refines this using the decomposition of $\cc_0$ into connected components, recalling that $\big(\hcc \kring \hfi (-,-)\trans\big)_{(d)} \in \ob \cc_0^{(-d)}$.
\end{proof}

\begin{nota}
\label{nota:tilde_pi}
Denote by
$ 
\tilde{\pi}_{(d)} : 
\big(\hcc \kring \hfi (-,-)\trans\big)_{(d)}
\rightarrow 
\radj \Big( \big(\hcc \kring \hfi (-,-)\trans\big)_{(d-1)} \Big)
$ 
the morphism given in Corollary \ref{cor:hcc_hfi_pi}.
\end{nota}

Corollary \ref{cor:hcc_hfi_pi} gives information on the structure of the functors $\big(\hcc \kring \hfi (-,-)\trans\big)_{(d)}$ arising from their relationships via the morphisms $\tilde{\pi}_{(d)}$. 

Recall that $\rho_{(\mathbf{0}, \mathbf{1})}$ is an exact functor, by Theorem \ref{thm:radj_adjoint};  Corollary \ref{cor:hcc_hfi_pi} thus gives:

\begin{prop}
\label{prop:filter_hcc_hfi_d}
For $d \in \nat$, the functor $\big(\hcc \kring \hfi (-,-)\trans\big)_{(d)}$ has a filtration of length $d+1$:
\[
0
\subset 
\filt_{(d)}^{d+1} \subset 
\filt_{(d)}^{d} \subset 
\ldots 
\subset
\filt_{(d)}^{0} =
 \big(\hcc \kring \hfi (-,-)\trans\big)_{(d)}
\]
defined recursively by 
$
\filt_{(d)}^{j} := (\tilde{\pi}_{(d)})^{-1} \Big(\rho_{(\mathbf{0}, \mathbf{1})} \filt_{(d-1)}^j\Big)$,
 where $\filt_{(d-1)}^{d+1}$ is taken to be zero, so that $\filt_{(d)}^{d+1}  = \ker \tilde{\pi}_{(d)}$.
\end{prop}

\section{Introducing $\kring \bbd (-,-)$}
\label{sect:exam_beads}

This Section introduces 
$ 
\kring \bbd (-,-) 
$ 
the $\cd$-module that  corresponds to the `baby bead representations' of the title. This explains the interest for us  of the category $\cd$.
 The key result of the Section is Theorem \ref{thm:bbd_int} which shows that $ 
\kring \bbd (-,-) 
$  is isomorphic to the  image   of  $\kring \hfi (-,-)\trans$ under $\ind : \apsh{\cc} \rightarrow \apsh{\cd}$.

Throughout, the identifications stemming from  Remark \ref{rem:group_groupoid_op} are applied implicitly to simplify the notation.

\subsection{Introducing $\bbd (-,-)$}
\label{subsect:beads}

The purpose of this section is to introduce the functor $\kring \bbd (-,-) \in \ob \apsh{\cd}$. 
We first construct the  $\fb^{\times 2}$-set $(X, Y) \mapsto \bbd (X,Y)$. For notational convenience, we work with the usual skeleton of $\fb$ and fix $n, b \in \nat$.

\begin{defn}
\label{defn:bbd}
\ 
\begin{enumerate}
\item 
A $b$-column arrangement of $n$-beads is an ordered decomposition of $\mathbf{n}$ into $b$-columns, where each column contains either $1$ or $2$ elements;
\item 
$\bbd (\mathbf{n}, \mathbf{b})$ is the set of $b$-column arrangements of $n$-beads;
\item 
the group $\sym_b$ acts via permutation of the columns and $\sym_n$ by relabelling. 
\end{enumerate}
\end{defn}

\begin{exam}
For example, taking $b=5$ and $n=8$, the following represents an element of $\bbd (\mathbf{8}, \mathbf{5})$:
\begin{center}
\begin{tikzpicture}[scale = 1.5]
 \draw [fill=white,thick] (.5,0) circle[radius = .15];
 \node at (.5,0) {$1$};
 \draw [fill=white,thick] (1,0) circle [radius = .15];
 \node at (1,0) {$4$};
\draw [fill=white,thick] (1.5,0) circle [radius = .15];
 \node at (1.5,0) {$7$};
\draw [fill=white,thick] (2,0) circle [radius = .15];
 \node at (2,0) {$8$};
 \draw [fill=white,thick] (2.5,0) circle [radius = .15];
 \node at (2.5,0) {$3$};
 \draw [fill=white,thick] (1,.5) circle [radius = .15];
 \node at (1,.5) {$5$};
 \draw [fill=white,thick] (2,.5) circle [radius = .15];
 \node at (2,.5) {$2$};
  \draw [fill=white,thick] (2.5,.5) circle [radius = .15];
 \node at (2.5,.5) {$6$};
 \node [below right] at (2.5,0) {.};
\end{tikzpicture}
\end{center}

\noindent
As pictured, the arrangement is viewed as two rows: the bottom contains $b$ (in this case, five) beads, whereas the top row contains $ n-b$ (in this case, three) beads.
 \end{exam}

\begin{rem}
\label{rem:bbd}
\ 
\begin{enumerate}
\item 
The set $\bbd (\mathbf{n}, \mathbf{b})$ is empty if $n >2b$ or if $n <b$.
\item 
For $b \leq n \leq 2n$,  $\bbd (\mathbf{n}, \mathbf{b})$ is a non-empty, transitive $\sym_n \times \sym_b$-set.  
\item 
An element of $\bbd(\mathbf{n}, \mathbf{b})$ is equivalent to giving a surjective set map $\pi : \mathbf{n} \twoheadrightarrow \mathbf{b}$ such that $|\pi^{-1} (i)|\leq 2$ for each $i \in \mathbf{b}$, together with a specified order on the fibres of cardinal two, i.e., if $|\pi^{-1} (i)|=2$, an isomorphism $\mathbf{2} \stackrel{\cong}{\rightarrow }\pi^{-1} (i)$.
\end{enumerate}
\end{rem} 
 
By construction, $(X,Y) \mapsto \bbd (X,Y) $ defines a functor $\fb^{\times 2} \rightarrow \sets$. This is the restriction of a functor on $\cc_0$, by the following:

\begin{lem}
\label{lem:bbd_cc0_presheaf}
The assignment $(X,Y) \mapsto \bbd(X,Y)$ has the structure of a functor $\cd_0 \rightarrow \sets$ such that the morphism $(X, Y) \rightarrow (X \amalg \mathbf{2}, Y \amalg \mathbf{1})$ (as in Remark \ref{rem:structure_cd}) extends a bead arrangement by defining the fibre over $1 \in \mathbf{1}$ to be $\mathbf{2}$ equipped with the canonical order.
\end{lem}
 
\begin{proof}
This result is almost tautological, given the definition of the category $\cd_0$. Namely, a morphism $(X,Y) \rightarrow (U,V)$ of $\cd_0$ 
 consists of a triple $(i,j, \zeta)$ where $\zeta : \mathbf{2} \times V \backslash j (Y) \stackrel{\cong}{\rightarrow} U \backslash i (X)$. The morphism $\zeta$ is equivalent to defining an element of $\bbd (U \backslash i (X), V \backslash j (Y)) $ in which each fibre has cardinal two. The naturality is then defined using the  obvious notion of concatenation of bead arrangements, together with the functoriality over $\fb^{\times 2}$. 
\end{proof}
 
Passing to the $\kring $-linearization, one obtains $\kring \bbd (-,-) \in \ob \apsh{\cd_0}$. As for $\kring \hfi (-,-) \trans$, this has additional functoriality. To show this requires establishing the $\finj\times \fb$-module structure that is analogous to the transpose structure of Lemma \ref{lem:transpose}:

\begin{lem}
\label{lem:finj_fb_bbd}
There is a unique $\finj\times \fb$-module structure on $\kring \bbd (-,-)$ such that:
\begin{enumerate}
\item 
the restriction to $\fb ^{\times 2}$ is the $\kring$-linearization of the $\fb^{\times 2}$-structure given by Definition \ref{defn:bbd}; 
\item 
the canonical morphism $(X, Y) \hookrightarrow (X \amalg \mathbf{1}, Y) $ (as in Remark \ref{rem:structure_cd}) sends a generator of $\kring \bbd (X,Y)$ to the sum of all bead arrangements of $\bbd(X \amalg \mathbf{1}, Y)$ that can be obtained by adding a bead labelled $1 \in \mathbf{1}$ to the top row.
\end{enumerate}
\end{lem}

\begin{proof}
The  basic idea  is identical to that of Lemma \ref{lem:transpose}, together with its application in the proof of Proposition \ref{prop:hfi_transpose}. The functoriality of $(X,Y) \mapsto \kring \bbd (X, Y)$ with respect to automorphisms of $Y$ is clear; that of $X$ in $\finj$ is checked directly. 
\end{proof}

\begin{prop}
\label{prop:cd_kring_bbd}
There is a unique $\cd$-module structure on $\kring \bbd (-,-)$ such that:
\begin{enumerate}
\item 
the restriction to $\cd_0$ is the $\kring$-linearization of the structure provided by Lemma \ref{lem:bbd_cc0_presheaf}; 
\item 
the restriction to $\finj \times \fb$ is the structure of Lemma \ref{lem:finj_fb_bbd}.
\end{enumerate}
\end{prop}

\begin{proof}
The proof is analogous to that of Proposition \ref{prop:hfi_transpose}. Unicity is a consequence of Proposition \ref{prop:factor_cd}, if the  structure exists. 

One first checks that the restrictions of the two given structures to $\fb^{\times 2}$ are the same. Then, in order to establish existence, it suffices to check that the diagram induced  by applying $\kring \bbd (-,-)$ to  the commutative square of Proposition \ref{prop:hfi_transpose} is commutative. This is a straightforward verification: the key point is that the action of the canonical morphism $(X, Y) \hookrightarrow (X \amalg \mathbf{1}, Y)$ as in Lemma \ref{lem:finj_fb_bbd} only adds beads to fibres with cardinal one.
\end{proof}

\begin{rem}
\ 
\begin{enumerate}
\item 
The category $\cd$ was introduced precisely so as to encode the above naturality of $\kring \bbd (-,-)$. 
\item 
There is no direct analogue of the functor $\kring \hfi (-,-)$ of Proposition \ref{prop:cc_to_sets} in this setting. This is related to the fact  that $\cd_0$ does not have an immediate analogue of the involution $\invcc$ on $\cc_0$. 
\end{enumerate}
\end{rem}

\subsection{Relating $\kring \hfi (-,-)\trans$ and $\kring \bbd (-,-)$}

As advertised at the beginning of the Section, the functor $\kring \bbd (-,-)$ is in the image of $\ind : \apsh{\cc} \rightarrow \apsh{\cd}$. More precisely:

\begin{thm}
\label{thm:bbd_int}
There is an isomorphism in $\apsh{\cd}$:
\[
\kring \bbd (-,-)
\cong 
\ind \kring \hfi (-,-)\trans 
.
\]
\end{thm}

\begin{proof}
The key point is to define a morphism
\[
\ind \kring \hfi (-,-)\trans 
\rightarrow
\kring \bbd (-,-).
\]
By Theorem \ref{thm:int_left_adjoint}, $\ind$ is left adjoint to the functor $\dbl^* : \apsh{\cd} \rightarrow \apsh{\cc}$. Hence, to define a morphism, it suffices to exhibit a  morphism in $\apsh{\cc}$:
\[
\kring \hfi (-,-)\trans 
\rightarrow
\dbl^* \kring \bbd (-,-).
\]

Now, $\dbl^* \kring \bbd (-,-)$ applied to $(X,Y) \in \ob \cc$ gives $\kring \bbd (X \amalg Y, Y)$. Consider an element $f$ of $\hfi (X,Y)$; using the viewpoint on $\bbd(-,-)$ given in Remark \ref{rem:bbd}, one associates to $f$ the bead arrangement on $X \amalg Y$ with underlying surjective map:
\[
X \amalg Y \stackrel{f \amalg \id_Y}{\rightarrow } Y
\]
using that $\amalg$ is the coproduct in finite sets; the fibres of cardinal two are ordered  by  $y<x$ for  $y \in Y$ and $x \in X$. This extends linearly to: 
\[
\kring \hfi (X, Y) 
\rightarrow 
\kring \bbd (X \amalg Y, Y) 
\]
and one verifies that this defines a natural transformation on $\cc$. 

It remains to establish that the adjoint morphism  $\ind \kring \hfi (-,-)\trans 
\rightarrow
\kring \bbd (-,-)$ is an isomorphism. To do so, one can restrict to $\cd_0$, which allows one to work at the level of functors with values in sets. 

Namely, the natural transformation evaluated on $(U,V) \in \ob \cd_0$ is given by  the $\kring$-linearization of the map:
\[
\amalg_{\kappa \in \hfi (V,U)} \hfi (U \backslash \kappa (V), V) 
\rightarrow 
\bbd (U,V).
\] 

Now, the codomain and domain are finite sets and it is straightforward to check that they have the same cardinality. Hence, to prove that the transformation is an isomorphism, it suffices to show that this is surjective. This follows since $\bbd (U,V)$ is  transitive as an $\aut(U) \times \aut(V)$-set (see Remark \ref{rem:bbd}).
\end{proof}

\begin{cor}
\label{cor:hcd_k_bbd}
There is an isomorphism in $\apsh{\cd_0}$:
\[
\hcd \kring \bbd (-,-)
\cong 
\ind  \hcc \kring \hfi (-,-)\trans 
.
\]
\end{cor}

\begin{proof}
This follows by combining Theorem \ref{thm:bbd_int} with Corollary \ref{cor:compat_int_H}.
\end{proof}

Corollary \ref{cor:hcd_k_bbd} gives an effective description of $\hcd \kring \bbd (-,-)$, since 
$\hcc \kring \hfi (-,-)\trans $ is understood by Theorem \ref{thm:calc} and the functor $\ind$ can be calculated explicitly. In particular, for the underlying representations, one has:

\begin{cor}
\label{cor:hcd_k_beads_values}
For $N, n \in \nat$, there is an isomorphism of $\sym_N \times \sym_n$-modules:
\[
\hcd \kring \bbd (\mathbf{N},\mathbf{n})
\cong 
\bigoplus_{\rho \vdash N} 
S^\rho \boxtimes 
\Big( 
\bigoplus_{\substack{\lambda \vdash n \\ \lambda_1 = 2n -N \\ \widehat{\lambda} \preceq \rho }}
(S^{\rho / \widehat{\lambda}}  \otimes S^\lambda)
\Big), 
\]
where the tensor product is given the diagonal $\sym_n$-module structure. 
\end{cor}

\begin{proof}
The result is an application of Proposition \ref{prop:int_values}, using the formulation given in Proposition \ref{prop:isotypical_rho} for the isotypical components. 
Namely, this is applied to $ \hcc \kring \hfi (-,-)\trans$ using Theorem \ref{thm:calc}, taking $b=n$ and $a= N-n$.  
\end{proof}

\section{Introducing anticommutativity: from $\kring \bbd (-,-)$ to $\bba (-,-)$}
\label{sect:exam_anti}

This short Section takes up the thread from Section \ref{sect:exam_beads} by introducing anticommutativity. This gives the object that is central to the paper, the functor $ \bba(-,-)$.

\subsection{The $\cd$-module $\bba(-,-)$}

In Section \ref{sect:exam_beads}, the functor $\bbd (-,-) : \cd_0 \rightarrow \sets$ was introduced, together with its $\kring$-linearization $\kring \bbd (-,-)$, which carries the structure of a $\cd$-module. The purpose of this Section is to introduce the quotient $\cd$-module 
\[
\kring \bbd (-,-) 
\twoheadrightarrow 
\bba (-,-)
\]
given by imposing an anticommutativity relation on the fibres of cardinal two.

It is important to note that, even after restriction to $\cd_0$,  this is not  the $\kring$-linearization of a functor $\cd_0 \rightarrow \sets$.

\begin{defn}
\label{defn:bba}
For finite sets $X$, $Y$, let $\bba (X,Y)$ be the quotient $\kring$-module of $\kring \bbd (X,Y)$ modulo the relations 
\[
\phi + \tau \phi =0
\]
where $\phi \in \bbd(X,Y)$ is a bead arrangement and $\tau \in \aut (X)$ is a transposition that permutes the labels of a cardinal-two column of $\phi$. 
\end{defn}

\begin{prop}
\label{prop:quotient_bba}
There is a unique $\cd$-module structure on $\bba$ such that the quotient map 
\[
\kring \bbd (-,-) 
\twoheadrightarrow 
\bba (-,-).
\]
is a morphism of $\cd$-modules.
\end{prop}

\begin{proof}
It suffices to prove that the kernel of the quotient map is a sub $\cd$-module. This is a straightforward verification.
\end{proof}

\subsection{Relating to $\lad \kring \bbd(-,-)$}

It is clear  from its construction, that $\bba (-,-)$ belongs to the full subcategory $\ce\dash\modules \subset \apsh{\cd}$.  
 Hence, the defining surjection induces
 \[
 \lad \kring \bbd(-,-) 
 \twoheadrightarrow 
 \bba (-,-)
 \]
in $\ce \dash\modules$. In fact:

\begin{prop}
\label{prop:bba_via_lad}
The canonical morphism $ \lad \kring \bbd(-,-) 
 \twoheadrightarrow 
 \bba (-,-)
 $ is an isomorphism in  $\ce \dash\modules$. 
   Hence there is an isomorphism in $\ce\dash\modules$:
\[
\bba (-,-) \cong \lad \ind \big( \kring \hfi (-,-)\trans \big).
\]
\end{prop}

\begin{proof}
The first statement  follows by comparing the explicit description of $\lad$ given in Proposition \ref{prop:lad_explicit} with the presentation of $\bba$. The second then follows immediately from Theorem \ref{thm:bbd_int}.
\end{proof}

Corollaries \ref{cor:compat_int_H}  and \ref{cor:lad_cd0} together with Proposition \ref{prop:bba_via_lad} therefore give:

\begin{cor}
\label{cor:hce_bba}
There is an isomorphism in $\ce_0\dash\modules$:
\[
\hce \bba (-,-) \cong \lad \ind[0] \big(\hcc \kring \hfi (-,-)\trans \big).
\]
\end{cor}

\begin{rem}
Proposition \ref{prop:bba_via_lad}
 and \ref{cor:hce_bba} explain the interest of considering the auxiliary structures provided by $\cc$ and $\cd$, since it gives a description of $\bba$ and $\hce \bba$ in terms of these. In particular, up to the application of $\lad \ind[0]$, the calculation of $\hce \bba (-,-)$ is reduced to the simpler calculation of $\hcc \kring \hfi (-,-)\trans$. Over a field of characteristic zero, Theorem \ref{thm:calc} provides the base for describing this functor.
\end{rem}

\section{Identifying the associated Schur functors}
\label{sect:schur}

This short Section  serves to show that $\bba(-,-)$ and $\hce\bba (-,-)$ do indeed correspond to the objects that motivated this study, as asserted in the Introduction. This is done working over a field $\kring$ of characteristic zero and by considering the associated Schur functors (cf. Section \ref{subsect:schur}).

\begin{prop}
\label{prop:schur_bbd}
For $b \in \nat$,  the Schur functor associated to $ \kring \bbd (-, \mathbf{b})$ identifies as 
\[
\kring \bbd (-, \mathbf{b}) (V) \cong \big( V \oplus V^{\otimes 2}) ^{\otimes b}.
\]
Moreover, this is $\aut (\mathbf{b})$-equivariant, where $\aut(\mathbf{b})$ acts on the right hand side by place permutations. 
\end{prop}

\begin{proof}
This identification follows by standard methods and is left to the reader.
\end{proof}

It is then straightforward to pass to the antisymmetrization $\bba$:

\begin{cor}
\label{cor:schur_bba}
For $b \in \nat$,  the Schur functor associated to $ \bba (-, \mathbf{b})$ identifies as 
\[
\bba (-, \mathbf{b}) (V) \cong \big( V \oplus \Lambda^2 V) ^{\otimes b} = (\lie_{\leq 2} (V)) ^{\otimes b}.
\]
Moreover, this is $\aut (\mathbf{b})$-equivariant, where $\aut(\mathbf{b})$ acts on the right hand side by place permutations. 
\end{cor}

It remains to treat the cases of $H^\cd _{\cd_0} \kring \bbd (-,-)$ and $\hce \bba (-,-)$. Now, by Propositions \ref{prop:extn_cd} and \ref{prop:ce_adjunctions} respectively, for fixed $\mathbf{b}$, the values of these can be understood in terms of $H^\finj _0$ applied to the appropriate functors. 

Now, the structure morphism of a $\finj$-module $G$ determines (and is determined by) the induced morphism at the level of Schur functors, which identifies as 
\[
 G(V) \otimes V \rightarrow G(V).
\]
(Here, for notational convenience, the action has been written from the right.) 
Then one identifies the Schur functor associated to $H^\finj_0 G$ with 
\[
\mathrm{Coker} \big( G(V) \otimes V \rightarrow G(V) \big).
\]

\begin{exam}
\label{exam:bbd_H0_schur}
In the case $G=\kring \bbd (-, \mathbf{b})$, equipped with the $\finj$-module structure given in Lemma \ref{lem:finj_fb_bbd}, the associated structure morphism identifies as 
\[
\sum_{i=1}^b \overline{m}_i : 
\big( V \oplus V^{\otimes 2}) ^{\otimes b} \otimes V 
\rightarrow 
\big( V \oplus V^{\otimes 2}) ^{\otimes b},
\] 
where $\overline{m}_i$ is the map given by applying the following composite with respect to the $i$th tensor factor of $\big( V \oplus V^{\otimes 2}) ^{\otimes b}$:
  \[
  (V \oplus V^{\otimes 2}) \otimes V 
  \subset 
  (\overline{T}V) \otimes V 
  \rightarrow 
  \overline{T}V
  \twoheadrightarrow 
  (V \oplus V^{\otimes 2})
  \]
in which $\overline{T}V = \bigoplus_{l \geq 1} V^{\otimes l}$, the first map is the canonical inclusion, the last the canonical projection and the middle map is multiplication in the tensor algebra.
\end{exam}

Passing to $\bba(-,-)$, the same reasoning applies; here the structure morphism is induced by 
\[
[\ ,\ ] : V \otimes V \twoheadrightarrow \Lambda^2 V = \lie_2 (V).
\]

\begin{prop}
\label{prop:schur_hce_bba}
For $b \in \nat$, there is a natural $\sym_b$-equivariant isomorphism of functors:
\[
\big( \hce \bba (-,\mathbf{b}) \big) (V) 
\cong 
H_0 (\lie (V) ; \lie_{\leq 2} (V)^{\otimes b} ) , 
\]
where the right hand side denotes Lie algebra homology in degree zero. 
\end{prop}

\begin{proof}
For any (right) $\lie (V)$-module $M$, the coinvariants $H_0 (\lie (V) ;M)$ identify naturally with the cokernel of 
\[
M \otimes V \rightarrow M
\]
given by the restriction of the $\lie (V)$-action on $M$ to $V \subset \lie (V)$. 

Hence, to prove the result, it suffices to observe that the $\finj$-structure morphism 
\[
\big( V \oplus \Lambda^2V) ^{\otimes b} \otimes V 
\rightarrow 
\big( V \oplus \Lambda^2 V) ^{\otimes b}
\] 
identifies with the (right) adjoint action on $\lie_{\leq 2}(V)$. This follows as in Example \ref{exam:bbd_H0_schur}.
\end{proof}

\section{Calculating isotypical components: two infinite families of examples}
\label{sect:first_exam}

The main problem of the paper is  equivalent to calculating the $\ce_0$-module
\[
\hce \bba (-,-),
\]
by  Proposition \ref{prop:schur_hce_bba}.
More specifically, for a given partition $\rho \vdash N$, we focus upon the $\rho$-isotypical component of $\hce \bba(-,-)$ evaluated on $(\mathbf{N}, \mathbf{n})$, for varying $n \in \nat$.

After considering the general case, we  focus on the following two families of partitions:
\begin{eqnarray*}
\rho &=& (1^N) \\
\rho &=& (2, 1^{N-2}), \mbox{ for $N>2$}.
\end{eqnarray*}

The first is delightfully simple and is treated in Theorem \ref{thm:isotyp_1N}; this recovers families of representations occurring in the work of Turchin and Willwacher. The second is the subject of Theorem \ref{thm:isotypical_lad_int_2,1s} and requires rather more work, although it allows for a complete calculation. 

Throughout the Section, $\kring$ is  a field of characteristic zero.
\subsection{The general approach}

By Corollary \ref{cor:hce_bba}, one has 
\[
\hce \bba (-,-)
\cong 
\lad \ind[0] \big(\hcc \kring \hfi (-,-)\trans\big).
\]
Hence, for fixed $N \in \nat$ and  $n \in \nat$, we seek to understand the following $\sym_N \times \sym_n$-modules and the maps between them:
\[
\xymatrix{
 \ind \big( \kring \hfi (-,-)\trans\big )(\mathbf{N}, \mathbf{n})
\ar@{->>}[r]
\ar@{->>}[d]
&
 \ind[0] \big(\hcc \kring \hfi (-,-)\trans\big)  (\mathbf{N}, \mathbf{n})
\ar@{->>}[d]
\\
\lad \ind \big( \kring \hfi (-,-)\trans\big) (\mathbf{N}, \mathbf{n})
\ar@{->>}[r]
&
\lad \ind[0] \big(\hcc \kring \hfi (-,-)\trans\big) (\mathbf{N}, \mathbf{n}),
}
\]
where the vertical surjections correspond to the adjunction unit of Proposition \ref{prop:lad} and the horizontal surjections are induced by the quotient map 
\[
\kring \hfi (-,-)\trans
\twoheadrightarrow 
\hcc \kring \hfi (-,-) \trans
\]
corresponding to the adjunction unit for the adjunction $\hcc : \apsh{\cc} \rightleftarrows \apsh{\cc_0} $ of Proposition \ref{prop:extn_left_adj_cc}.

By Theorem \ref{thm:calc}, evaluated on $(\mathbf{m}, \mathbf{n})$, the morphism 
$
\kring \hfi (-,-)\trans (\mathbf{m}, \mathbf{n})
\twoheadrightarrow 
\hcc \kring \hfi (-,-) \trans (\mathbf{m}, \mathbf{n})
$
of $\sym_m \times \sym_n$-modules identifies as the projection:
\[
\bigoplus_{\substack{\lambda \vdash n , \nu \vdash m \\ \widehat{\lambda} \preceq \nu \preceq \lambda}}
 S^{\nu} \boxtimes S^\lambda 
\twoheadrightarrow 
\bigoplus_{\substack{\lambda \vdash n \\ \lambda_1 =n-m}}
S^{\widehat{\lambda}} \boxtimes S^\lambda.
\]
It follows that calculating the top row of the above diagram is straightforward, by using Proposition \ref{prop:int_values}, as in Corollary \ref{cor:hcd_k_beads_values}.

The remaining step is to calculate the right hand vertical map, using the description of $\lad \ind $ given in Section \ref{sect:rep_compose}, which involves the structure of $\hcc \kring \hfi (-,-)\trans$ as a functor on $\cc_0$.

 Restricting to the $\rho$-isotypical component for a fixed $\rho \vdash N$ makes this simpler, by using the results of Section \ref{subsect:restrict_isotypic_ldm}. Namely, one has the following extension of Corollary \ref{cor:hcd_k_beads_values}:

\begin{thm}
\label{thm:isotyp_hce_bba}
For $N, n \in \nat$ and $\rho \vdash N$, the $\rho$-isotypical component of 
$\hce \bba (\mathbf{N},\mathbf{n})$ is isomorphic as a $\sym_n$-module to the cokernel of 
\[
\bigoplus_{\substack{\lambda' \vdash n-1 \\ \lambda'_1 = 2n -N \\ \widehat{\lambda'} \preceq \rho }}
\big(
(S^{(\rho / \widehat{\lambda'})})^{\sym_2} \otimes S^{\lambda'}\big)\uparrow_{\sym_{n-1}}^{\sym_n}
\rightarrow 
\bigoplus_{\substack{\lambda \vdash n \\ \lambda_1 = 2n -N \\ \widehat{\lambda} \preceq \rho }}
(S^{\rho / \widehat{\lambda}}  \otimes S^\lambda).
 \]

In particular, the component of the map between the factors indexed by $\lambda'$ and $\lambda$ is zero unless the following two conditions are satisfied:
\begin{enumerate}
\item 
$\lambda' \preceq \lambda$ and $\widehat{\lambda} \preceq \rho$;
\item 
$(S^{(\rho / \widehat{\lambda'})})^{\sym_2} \neq 0$.
\end{enumerate}
This component is determined, up to non-zero scalar multiple, by Proposition \ref{prop:ldm_special_case} applied to $(S^{\widehat{\lambda'}} \boxtimes S^{\lambda'}, S^{\widehat{\lambda}} \boxtimes S^{\lambda}, \omega)$ with $\omega \neq 0$.
\end{thm}

\begin{proof}
This follows from Theorem \ref{thm:lad_int} on passage to isotypical components as in Section \ref{subsect:restrict_isotypic_ldm}, using Corollary \ref{cor:hce_bba} and the identification of $\hcc \kring \hfi (-,-) \trans $ given by Theorem \ref{thm:calc}.
\end{proof}

\begin{rem}
\label{rem:reindex_mu_mu'}
The map appearing in Theorem \ref{thm:isotyp_hce_bba} can be reindexed using $\mu'=\widehat{\lambda'}$ and $\mu = \widehat{\lambda}$ as the indexing partitions, as follows:
\[
\bigoplus_{\substack{\mu' \vdash m-1 \\ \mu'_1 \leq n-m \\ \mu' \preceq \rho }}
(S^{(\rho / \mu')})^{\sym_2}  \otimes S^{(n-m)\cdot\mu'})\uparrow_{\sym_{n-1}}^{\sym_n}
\rightarrow 
\bigoplus_{\substack{\mu \vdash m \\ \mu_1 \leq n-m \\ \mu \preceq \rho }}
S^{\rho / \mu}  \otimes S^{(n-m)\cdot \mu},
 \]
 using  Notation \ref{nota:partitions}.
 \end{rem}

\subsection{The case $\rho = (1^N)$}

The implementation of Theorem \ref{thm:isotyp_hce_bba} in the case  of the $(1^N)$-isotypical component of $\hce \bba (\mathbf{N}, -)$ is particularly simple, as predicted by Corollary \ref{cor:isotyp_1N_F}:

\begin{thm}
\label{thm:isotyp_1N}
For $0< N \in \nat$ and $n \in \nat$, the $(1^N)$-isotypical component of $\hce \bba  (\mathbf{N}, \mathbf{n})$ is isomorphic to the $\sym_n$-module
\[
\begin{array}{ll}
S^{(m+1, 1^{N-2m-1})}  & 0 \leq m, \  2m< N \\
0 & \mbox{otherwise,}
\end{array}
\]
where $m = N-n $.
\end{thm}
 
\begin{rem}
Theorem \ref{THMBbefore} of the introduction follows from Theorem \ref{thm:isotyp_1N} by using Proposition \ref{prop:schur_hce_bba} to relate to the corresponding isotypical component of 
$H_0 (\lie(V); \tlie(V)^{\otimes n})$ and then Proposition \ref{prop:1^N} to pass to $H_0 (\lie(V); \lie(V)^{\otimes n})$.
\end{rem}

\begin{exam} 
\label{exam:isotyp_1N}
The pattern exhibited in Theorem \ref{thm:isotyp_1N} is made clear by considering behaviour for small $N$. For this, only the partition $(m+1, 1^{N-2m-1})$ that arises is indicated:

\[
\begin{array}{|l|ll|}
\hline
N=1 
&
(1) & m=0.
\\
\hline
N=2 &
(1,1) & m=0. 
\\
\hline
N=3 &
(1,1,1) & m=0
\\
&(2)& m=1 . 
\\
\hline
N=4 
&
(1,1,1,1) & m=0
\\
&(2,1)& m=1.  
\\
\hline
N=5 &
 (1^5) & m=0
\\
&(2,1^2)& m=1
\\
&(3) & m=2.  
\\
\hline
N=6& 
(1^6) & m=0
\\
&
(2,1^3)& m=1
\\
&(3,1) & m=2.  
\\
\hline
N=7 &
(1^7) & m=0
\\
&(2,1^4)& m=1
\\
&(3,1^2) & m=2
\\
&(4)& m=3.  
\\
\hline
\end{array}
\]
 \end{exam}
 
\begin{rem}
\label{rem:relate_TW}
Via the theory outlined in Section \ref{sect:tw}, these representations correspond to those occurring in Turchin and Willwacher's work \cite[Theorem  2]{MR3653316} and \cite[Theorem 1]{MR3982870}, once the full structure as objects of $\flie$ is taken into account. 

This observation is equivalent to that made in \cite[Section 19.2]{2018arXiv180207574P}. In particular,  the $(1^N)$-isotypical component of $\hce \bba (\mathbf{N}, -)$, considered as an object of $\flie$, corresponds (via duality $\apsh{\gr}\op \stackrel{^\sharp}{\rightarrow} \apsh{\gr\op}$)  to the object $\omega \beta_{N} S_{(1^N)} \in \apsh{\gr}$ of  \cite{2018arXiv180207574P}. The identification of the composition factors should be compared with \cite[Corollary 19.9]{2018arXiv180207574P}. 
\end{rem} 
 \subsection{The case $\rho = (2, 1^{N-2})$, with $N>2$}
 
In this Section,  Theorem \ref{thm:isotyp_hce_bba} is implemented in the case $\rho = (2, 1^{N-2})$, with $N>2$. This is more involved than Theorem \ref{thm:isotyp_1N} since it requires  calculating the map, which is non-trivial in this case. 

The following convention is applied:

\begin{conve}
\label{conve:non_partitions}
For $\underline{a} := (a_1, \ldots,a_l)$ a sequence of integers, $S^{\underline{a}}$ denotes 
 the simple $\sym_{\sum a_i}$-module indexed by $\underline{a}$ if this is a partition (i.e., if $a_1 \geq a_2 \geq \ldots \geq a_l \geq 0$) and zero otherwise.
\end{conve}

The main result of the section is the following:

\begin{thm}
\label{thm:isotypical_lad_int_2,1s}
Suppose that $N>2$. Then the $(2,1^{N-2})$-isotypical component of 
 $\hce \bba  (\mathbf{N}, \mathbf{n})$ is isomorphic to the $\sym_n$-module:
 \[
 \begin{array}{ll}
 m=0 & 
 S^{(2,1^{N-2})}
 \\
 m=1, N=3 & 0 
 \\
  m=1, N>3 &
 S^{(2,1^{N-3})} \oplus S^{(3,1^{N-4})} \oplus S^{(2,2,1^{N-5})} 
  \\
m>1, N-2m \geq 2 &
S^{(m+1,1^{N-2m-1}) }
\oplus 
S^{(m+2, 1^{N-2m -2})} 
\oplus 
S^{(m,2, 1^{N-2m-2})}
\oplus 
S^{(m+1, 2, 1^{N-2m-3})} 
\\
 \mbox{otherwise} & 0, 
\end{array} 
 \] 
In particular, this is multiplicity-free.
\end{thm}

\begin{rem}
Theorem \ref{THMB} of the introduction follows from Theorem \ref{thm:isotypical_lad_int_2,1s} by using Proposition \ref{prop:schur_hce_bba} to relate to the corresponding isotypical component of 
$H_0 (\lie(V); \tlie(V)^{\otimes n})$ and then Theorem \ref{thm:2_1_N-2_isotypical} to pass to $H_0 (\lie(V); \lie(V)^{\otimes n})$.
\end{rem}

\begin{exam}
\label{exam:isotyp_2,1s}
Theorem \ref{thm:isotypical_lad_int_2,1s} can be illustrated by  explicit calculations for small $N$. Due to Convention \ref{conve:non_partitions}, these exhibit differing behaviour depending 
on the parity of $N$ (together with some exceptional behaviour in the case $m=1$ for $N\leq 5$). 

The following lists the partitions corresponding to the simple composition factors, as in Example \ref{exam:isotyp_1N}:

\[
\begin{array}{|l|ll|}
\hline
N=3& (2,1) & m=0.
\\ 
\hline
N=4&
 (2,1,1) & m=0 \\
& (2,1), \ (3)& m=1 .
\\
\hline
N=5&
 (2,1,1,1) & m=0 \\
 &(2,1,1), \ (3,1), \ (2,2) & m=1 .
 \\
 \hline
 N=6&
 (2,1^4) & m=0 \\
 &(2,1^3), \ (3,1,1), \ (2,2,1) & m=1 
 \\
&(2,2), \  (3,1), \ (4) & m=2  .
 \\
 \hline
 N=7&
 
 (2,1^5) & m=0 \\
 &(2,1^4), \ (3,1^3), \ (2,2,1^2) & m=1 
 \\
&(2,2,1), \  (3,1^2), \ (4,1), \  (3,2) & m=2  .
 \\
 \hline
 N=8&
 (2,1^6) & m=0 \\
& (2,1^5), \ (3,1^4), \ (2,2,1^3) & m=1 
 \\
&(2^2,1^2), \  (3,1^3), \ (4,1^2), \  (3,2,1) & m=2 
\\
&(3,2), \  (4,1), \ (5) & m=3 .
\\
\hline 
N=9& 
 (2,1^7) & m=0 \\
& (2,1^6), \ (3,1^5), \ (2,2,1^4) & m=1 
 \\
&(2^2,1^3), \  (3,1^4), \ (4,1^3), \  (3,2,1^2) & m=2 
\\
&(3,2,1), \  (4,1^2), \ (5,1),\  (4,2) & m=3 .
\\
\hline
 \end{array}
\]
This is sufficient to establish the pattern given by the above results.
\end{exam}

In the proof, we use the following identification, which is established by standard techniques:

\begin{lem}
\label{lem:identify_skew}
For $\rho = (2, 1^{N-2})$ ($N>2$) and $\mu \vdash m$,
\begin{enumerate}
\item 
if $m<N$, $\mu \preceq \rho$ if and only if $\mu \in \{ (1^m) , (2,1^{m-2}) \}$, where $(2,1^{m-2})$ is understood to be $(2)$ for $m=2$ and is not defined for $m<2$. 

The skew representation $S^{\rho/\mu}$ identifies as 
\[
S^{\rho/ \mu}
\cong 
\left\{ 
\begin{array}{ll}
\sgn_n & \mu = (2,1^{m-2}) \\
\sgn_{n-1}\uparrow _{\sym_{n-1}}^{\sym_n} & \mu = (1^m) \\
0 &  \mbox{otherwise},
\end{array}
\right.
\]
where $n= N-m$.
\item 
if $N=m$, $\mu \preceq \rho$ if and only if $\mu = \rho$, in which case the Skew representation $S^{\rho/\mu}$ is $\triv_0$.
\end{enumerate}
\end{lem}

By Theorem \ref{thm:isotyp_hce_bba} and its reformulation in Remark \ref{rem:reindex_mu_mu'}, for fixed $N,n$ and $m= N-n$, we require to calculate the cokernel of
\begin{eqnarray}
\label{eqn:mu_mu'_21^N-2}
\bigoplus_{\substack{\mu' \vdash m-1 \\ \mu'_1 \leq n-m \\ \mu' \preceq (2,1^{N-2}) }}
(S^{((2,1^{N-2}) / \mu')})^{\sym_2}  \otimes S^{(n-m)\cdot\mu'})\uparrow_{\sym_{n-1}}^{\sym_n}
\rightarrow 
\bigoplus_{\substack{\mu \vdash m \\ \mu_1 \leq n-m \\ \mu \preceq (2,1^{N-2}) }}
S^{(2,1^{N-2}) / \mu}  \otimes S^{(n-m)\cdot \mu}.
\end{eqnarray}
 
The conditions $\mu' \preceq (2,1^{N-2})$ and $\mu \preceq (2,1^{N-2 })$ are governed by Lemma \ref{lem:identify_skew}.

\begin{lem}
\label{lem:mu_mu'}
\ 
\begin{enumerate}
\item 
The sum over $\mu$ in (\ref{eqn:mu_mu'_21^N-2}) is indexed by the following partitions, with distinguished values depending upon $(N,m)$:
\[
\begin{array}{|l|l|l|l|} 
\hline 
m & n& \mu & (n-m)\cdot \mu 
\\
\hline
\hline
0 & N& (0) & (N) \\
\hline
1 & N-1 & (1) & (N-2,1)
\\
\hline 
m>1, N-2m >1  & n=N-m & (2,1^{m-2}) & (N-2m, 2,1^{m-2})
\\
\hline
m>1, N -2m \geq 1  & n= N-m & (1^m) & (N-2m, 1^m) \\
\hline
\end{array}.
\]
\item 
The domain of (\ref{eqn:mu_mu'_21^N-2}) is zero if $m=0$ and, for $m>0$, identifies with
\[
\big((S^{(2,1^{N-2}) /(1 ^{m-1})}) ^{\sym_2} \otimes S^{(n-m,1^{m-1})}\big)\uparrow_{\sym_{n-1}}^{\sym_n}.
\]
Here, $(S^{(2,1^{N-2}) /(1 ^{m-1})}) ^{\sym_2} \cong \sgn_{n-1}$.
\end{enumerate}
\end{lem}

\begin{proof}
The first statement simply identifies the partitions $\mu \vdash m$ such that $\mu \preceq (2,1^{N-2})$ and  $\mu_1 \preceq n-m$, by applying Lemma \ref{lem:identify_skew}.

The second case proceeds similarly: the conditions $\mu' \vdash m-1$ and $\mu' \preceq (2,1^{N-2})$ imply that  $\mu'$ is one of $(1^{m-1})$ or $(2,1^{m-3})$.
The second possibility only occurs if $m \geq 3$. Now, 
$S^{(2,1^{N-2})/(2,1^{m-3})} = \sgn_{n+1}$ as a $\sym_{n+1}$-module, so that $(S^{(2,1^{N-2})/(2,1^{m-3})})^{\sym_2}=0$ (cf. the proof of Corollary \ref{cor:ldm_special_case_rho=1N}).

Thus one reduces to the case $\mu' = (1^{m-1})$, which only occurs if $m>0$.  Lemma \ref{lem:sym2_invt_generates} gives that $ (S^{(2,1^{N-2}) /(1 ^{m-1})}) ^{\sym_2}$ is isomorphic to $\sgn_{n-1}$.
\end{proof}

To prove Theorem \ref{thm:isotypical_lad_int_2,1s}, we need to understand the map (\ref{eqn:mu_mu'_21^N-2}), for which we may suppose that $m>0$ and $N>2m$. We focus upon the component corresponding to $\mu'= (1^{m-1})$ and $\mu=(1^m)$, namely the morphism of $\sym_n$-modules:
\begin{eqnarray}
\big((S^{(2,1^{N-2}) /(1 ^{m-1})}) ^{\sym_2} \otimes S^{(n-m,1^{m-1})}\big)\uparrow_{\sym_{n-1}}^{\sym_n}
\longrightarrow 
S^{(2,1^{N-2}) / (1^m)}  \otimes S^{(n-m,1^m)}.
\end{eqnarray}

This map is described explicitly (up to non-zero scalar multiple)  by Proposition \ref{prop:ldm_special_case}, applied with $\mu'= (1^{m-1}) \preceq \mu = (1^m)$ and $\nu' = (n-m,1^{m-1})\preceq \nu = (n-m,1^m)$, since Theorem \ref{thm:calc} ensures that it must be non-zero.

The following is the key non-trivial ingredient in the proof of Theorem \ref{thm:isotypical_lad_int_2,1s}:

\begin{lem}
\label{lem:injectivity}
Suppose that $1 \leq m$ and $N>2m$. 
The morphism of $\sym_n$-modules 
\[
\big(
(S^{(2,1^{N-2}) /(1 ^{m-1})}) ^{\sym_2} \otimes S^{(n-m,1^{m-1})}\big)\uparrow_{\sym_{n-1}}^{\sym_n}
\longrightarrow 
S^{(2,1^{N-2}) / (1^m)}  \otimes S^{(n-m,1^m)} .
\]
is injective.
\end{lem}

\begin{proof}
This follows from Corollary \ref{cor:mono_induction_argument} by using Lemma \ref{lem:sym2_invt_generates}.
\end{proof}

It remains to identify the representations that occur in the terms of (\ref{eqn:mu_mu'_21^N-2}). One has the general result:

\begin{lem}
\label{lem:isotypical_2,1s}
For $N>2$ and $m<N$, the $\rho = (2,1^{N-2})$-isotypical component of $\kring \sym_{N} \otimes_{\sym_m} (S^\mu \boxtimes S^\nu) $ is:
\[
\begin{array}{lll}
 S^{\nu ^\dagger} && \mu = (2,1^{m-2}), m \geq 2 \\
 \Big(\bigoplus_{\substack{\kappa \preceq \nu  \\ |\kappa | = n-1 }}
(S^{ \kappa ^\dagger})\uparrow_{\sym_{n-1}}^{\sym_n} \Big)
&
\cong
\ 
 \Big(\bigoplus_{\substack{\kappa \preceq \nu  \\ |\kappa | = n-1 }}\bigoplus_{\substack{\delta \vdash n \\ \kappa \preceq \delta}}
S^{ \delta ^\dagger} \Big)
& \mu = (1^m), m>0 
\\
(S^{(2,1^{N-2})} \otimes S^\nu)
&& \mu = (0) 
\\
0 && \mbox{otherwise},
\end{array}
\]
where $n=N-m$.

In the case $N=m$ (so that $n=0$), the isotypical component is $\triv_0$  if $ \mu = (2,1^{N-2})$ and zero otherwise.
\end{lem}

\begin{proof}
We consider the case $m<N$; the remaining case is straightforward. 

The result is proved by applying Proposition \ref{prop:isotypical_rho}, so that we need to calculate $S^{\rho/\mu} \otimes S^{\nu}$.  The two possible non-trivial cases for $\mu$ are identified by Lemma \ref{lem:identify_skew}. 

For $\mu = (2,1^{m-2})$, $S^{(2,1^{N-2})/ \mu}$ is the representation $\sgn_n$, by Lemma \ref{lem:identify_skew}, and the result follows immediately in this case.

For $\mu=(1^m)$ with $m>0$, by Lemma \ref{lem:identify_skew} $S^{(2,1^{N-2})/ \mu}$ is the representation $\sgn_{n-1}\uparrow _{\sym_{n-1}}^{\sym_n}$ so we need to identify the $\sym_n$-module:
$$
\big(\sgn_{n-1}\uparrow _{\sym_{n-1}}^{\sym_n}\big) \otimes S^\nu
\cong 
\big(\sgn_{n-1} \otimes S^\nu \downarrow _{\sym_{n-1}}^{\sym_n}\big) \uparrow _{\sym_{n-1}}^{\sym_n}, 
$$
where the isomorphism is given by 
Proposition \ref{prop:induct_restr}.

The first identification then follows from Pieri's rule $S^\nu \downarrow _{\sym_{n-1}}^{\sym_n} \cong \bigoplus_{\substack{\kappa \preceq \nu  \\ |\kappa | = n-1 }}S^\kappa$. 

Pieri's rule for induction gives 
$$
S^{ \kappa ^\dagger}\uparrow_{\sym_{n-1}}^{\sym_n} \cong \bigoplus_{\substack{\delta \vdash n \\ \kappa \preceq \delta}} S^{ \delta ^\dagger},
$$
(one can choose to apply induction either before or after applying $^\dagger$).
The second expression follows.

For $m=0$, $\mu = (0)$ and the skew representation is $S^{(2, 1^{N-2})}$, which gives the result in this case.
\end{proof}

Lemma \ref{lem:isotypical_2,1s} applies to the case in hand, paying attention to the exceptional cases (notably $N=3$ and $m=1$, which should be treated as a case of the instance $N-2m=1$, and the instance $N-2m=1$):

\begin{lem}
\label{lem:int_isotyp_2,1s}
Suppose that $N>2$. The $\sym_n$-module 
$\bigoplus_\mu \big (S^{(2,1^{N-2}) / \mu}  \otimes S^{(N-2m)\cdot \mu} \big)$ appearing in (\ref{eqn:mu_mu'_21^N-2})  identifies as:
 \[
\begin{array}{ll}
m=0 &
S^{(2,1^{N-2})} 
\\
m=1, N=3 & 
\triv_2 \oplus \sgn_2 
\\  
m=1, N>3 &
\sgn_{N-1} \oplus (S^{(2,1^{N-3})})^{\oplus 2} \oplus S^{(3,1^{N-4})} \oplus S^{(2,2,1^{N-5})} 
\\ 
m>1, N-2m >1&
 (S^{(m+1,1^{N-2m-1}) })^{\oplus 2}
\oplus 
S^{(m, 1^{N-2m} )} 
\oplus 
S^{(m+2, 1^{N-2m -2})} 
\oplus 
(S^{(m,2, 1^{N-2m-2})})^{\oplus 2} 
\oplus 
S^{(m+1, 2, 1^{N-2m-3})}
\\
N-2m=1 &
\triv_{m+1} \oplus S^{(m,1)} 
\\
N\leq 2m	&
0 .
 \end{array}
 \] 
\end{lem}

\begin{lem}
\label{lem:relations}
Suppose that $N>2$. The $\sym_n$-module
$\bigoplus_{\mu'}
(S^{((2,1^{N-2}) / \mu')})^{\sym_2}  \otimes S^{(n-m)\cdot\mu'})\uparrow_{\sym_{n-1}}^{\sym_n}$ appearing in equation (\ref{eqn:mu_mu'_21^N-2})  identifies as:

\noindent
$
\begin{array}{ll}
m=0 &
0  \\
m=1 &
\sgn_{N-1} \oplus S^{(2,1^{N-3})} \\
m>1, N-2m \geq 1&
S^{(m+1,1^{N-2m-1})} \oplus S^{(m,2,1^{N-2m-2})} \oplus S^{(m,1^{N-2m})}  \\
N\leq 2m & 0 .
\end{array}
$
\end{lem}

\begin{proof}
By Lemma \ref{lem:mu_mu'}, the domain is zero for $m=0$ and identifies with 
\[
\big((S^{(2,1^{N-2}) /(1 ^{m-1})}) ^{\sym_2} \otimes S^{(n-m,1^{m-1})}\big)\uparrow_{\sym_{n-1}}^{\sym_n} 
\]
for $m>0$, where  $S^{(2,1^{N-2}) /(1 ^{m-1})}) ^{\sym_2} \cong \sgn_{n-1}$. The calculation of the representation is then straightforward, paying attention to the degenerate cases $m=1$ and $N-2m \in \{1,2\}$.
\end{proof}

\begin{proof}[Proof of Theorem \ref{thm:isotypical_lad_int_2,1s}]
The result follows from Theorem \ref{thm:isotyp_hce_bba}, namely by  calculating the cokernel of (\ref{eqn:mu_mu'_21^N-2}). By Lemma \ref{lem:injectivity}, the map is injective, hence the result follows from subtracting the contribution given by Lemma \ref{lem:relations} from that of Lemma \ref{lem:int_isotyp_2,1s}.
\end{proof}


\begin{thebibliography}{{Pow}22b}

\bibitem[BCGY21]{2021arXiv210903302B}
Christin {Bibby}, Melody {Chan}, Nir {Gadish}, and Claudia~He {Yun},
  \emph{{Homology representations of compactified configurations on graphs
  applied to $M_{2,n}$}}, arXiv e-prints (2021), arXiv:2109.03302.

\bibitem[CE17]{MR3654111}
Thomas Church and Jordan~S. Ellenberg, \emph{Homology of {FI}-modules}, Geom.
  Topol. \textbf{21} (2017), no.~4, 2373--2418. \MR{3654111}

\bibitem[GH22]{2022arXiv220212494G}
Nir {Gadish} and Louis {Hainaut}, \emph{{Configuration spaces on a wedge of
  spheres and Hochschild-Pirashvili homology}}, arXiv e-prints (2022),
  arXiv:2202.12494.

\bibitem[GK94]{MR1301191}
Victor Ginzburg and Mikhail Kapranov, \emph{Koszul duality for operads}, Duke
  Math. J. \textbf{76} (1994), no.~1, 203--272. \MR{1301191}

\bibitem[Koi89]{MR991410}
Kazuhiko Koike, \emph{On the decomposition of tensor products of the
  representations of the classical groups: by means of the universal
  characters}, Adv. Math. \textbf{74} (1989), no.~1, 57--86. \MR{991410}

\bibitem[{Pow}21]{2021arXiv211001934P}
Geoffrey {Powell}, \emph{{On analytic contravariant functors on free groups}},
  arXiv e-prints (2021), arXiv:2110.01934.

\bibitem[{Pow}22a]{P_bb_finj}
\bysame, \emph{{On the $\mathbf{FI}$-homology of the injective cogenerators}},
  preprint (2022).

\bibitem[{Pow}22b]{2022arXiv220113307P}
\bysame, \emph{{Outer functors and a general operadic framework}}, arXiv
  e-prints (2022), arXiv:2201.13307.

\bibitem[PV18]{2018arXiv180207574P}
Geoffrey {Powell} and Christine {Vespa}, \emph{{Higher Hochschild homology and
  exponential functors}}, ArXiv e-prints (2018), arXiv:1802.07574.

\bibitem[Reu93]{MR1231799}
Christophe Reutenauer, \emph{Free {L}ie algebras}, London Mathematical Society
  Monographs. New Series, vol.~7, The Clarendon Press, Oxford University Press,
  New York, 1993, Oxford Science Publications. \MR{1231799}

\bibitem[{Sage}22]{sagemath}
{The Sage Developers}, \emph{{S}agemath, the {S}age {M}athematics {S}oftware
  {S}ystem ({V}ersion 9.4)}, 2022, {\tt https://www.sagemath.org}.

\bibitem[SS15]{MR3376738}
Steven~V. Sam and Andrew Snowden, \emph{Stability patterns in representation
  theory}, Forum Math. Sigma \textbf{3} (2015), Paper No. e11, 108.
  \MR{3376738}


\bibitem[TW17]{MR3653316}
Victor Turchin and Thomas Willwacher, \emph{Commutative hairy graphs and
  representations of {${\rm Out}(F_r)$}}, J. Topol. \textbf{10} (2017), no.~2,
  386--411. \MR{3653316}

\bibitem[TW19]{MR3982870}
\bysame, \emph{Hochschild-{P}irashvili homology on suspensions and
  representations of {${\rm Out}(F_n)$}}, Ann. Sci. \'{E}c. Norm. Sup\'{e}r.
  (4) \textbf{52} (2019), no.~3, 761--795. \MR{3982870}

\bibitem[Zin12]{MR3013090}
G.~W. Zinbiel, \emph{Encyclopedia of types of algebras 2010}, Operads and
  universal algebra, Nankai Ser. Pure Appl. Math. Theoret. Phys., vol.~9, World
  Sci. Publ., Hackensack, NJ, 2012, pp.~217--297. \MR{3013090}

\end{thebibliography}
\providecommand{\bysame}{\leavevmode\hbox to3em{\hrulefill}\thinspace}
\providecommand{\MR}{\relax\ifhmode\unskip\space\fi MR }
\providecommand{\MRhref}[2]{%
  \href{http://www.ams.org/mathscinet-getitem?mr=#1}{#2}
}
\providecommand{\href}[2]{#2}

\end{document}